\documentclass[a4paper,11pt,oneside]{article}
\usepackage[utf8]{inputenc}
\usepackage[T1]{fontenc}
\usepackage[english]{babel}  
\usepackage{amsthm}
\usepackage{amssymb}
\usepackage{amsmath}
\usepackage{mathrsfs}
\usepackage{graphicx}
\usepackage{xcolor}
\usepackage{eurosym}
\usepackage{pifont}
\usepackage{fancybox}
\usepackage{multicol}
\usepackage{enumitem}
\usepackage{url}
\usepackage{bbm}
\usepackage[all]{xy}
\usepackage{authblk}
\reversemarginpar
\numberwithin{equation}{section}
\usepackage[top=2.5cm, bottom=3cm, left=1.5cm , right=1.5cm]{geometry}
\usepackage[colorlinks=true,linkcolor=black,citecolor=black,urlcolor=black,hyperfootnotes=false]{hyperref}

\theoremstyle{definition} \newtheorem{Def}{Definition}[section]
\theoremstyle{plain}\newtheorem{Thm}[Def]{Theorem}
\theoremstyle{plain}\newtheorem{Prop}[Def]{Proposition}
\theoremstyle{definition}\newtheorem{Rem}[Def]{Remark}
\theoremstyle{definition}
\theoremstyle{plain}\newtheorem{Lem}[Def]{Lemma}
\theoremstyle{plain}\newtheorem{Cor}[Def]{Corollary}
\theoremstyle{remark} \newtheorem*{Claim*}{Claim}

\newcommand{\R}{\mathbb{R}}
\newcommand{\N}{\mathbb{N}}
\newcommand{\Z}{\mathbb{Z}}

\newcommand{\T}{\mathbb{T}^3}
\newcommand{\Td}{\mathbb{T}^d}
\newcommand{\tor}[1]{\mathbb{T}^{#1}}
\newcommand{\la}{\lambda}

\newcommand{\Prob}{\mathbb{P}}
\newcommand{\E}[1]{\mathbb{E}\left[#1\right]}

\newcommand{\V}[1]{\mathrm{Var}{\left[#1\right]}}
\newcommand{\Cov}[2]{\mathrm{Cov}{\left[#1,#2\right]}}

\newcommand{\eqLaw}{ \overset{\mathcal{L}}{=} }
\newcommand{\Law}{\xrightarrow[]{\mathcal{L}} }
\newcommand{\toP}{\xrightarrow[]{\Prob} }
\newcommand{\toLtwo}{\xrightarrow[]{L^2(\Prob)} }

\newcommand{\norm}[1]{\lVert#1\rVert}
\newcommand{\scal}[2]{\langle #1,#2\rangle}

\newcommand{\eps}{\varepsilon}
\newcommand{\ind}[1]{\mathbb{I}
\left\{#1\right\}}

\newcommand{\notcon}[3]{ {#1} \not\equiv {#2} \pmod {#3}  }

\newcommand{\Nn}{\mathcal{N}_n}

\DeclareMathOperator{\proj}{proj}
\DeclareMathOperator{\Jac}{Jac}
\DeclareMathOperator{\Leb}{Leb}
\DeclareMathOperator{\Id}{Id}
\DeclareMathOperator{\bX}{\mathbf{X}}
\DeclareMathOperator{\bT}{\mathbf{T}}

\allowdisplaybreaks

\begin{document}

\title{\normalsize{\textbf{\uppercase{Fluctuations of nodal sets on the $3$-torus and \\ general cancellation phenomena}}}}
\date{ }  
\author[*]{\textsc{Notarnicola} Massimo }
\affil[*]{\textit{Department of Mathematics, Universit\'e du Luxembourg}}
\affil[ ]{E-mail: \href{mailto:massimo.notarnicola@uni.lu}{massimo.notarnicola@uni.lu}}
\maketitle
\abstract{
In 2017, Benatar and Maffucci \cite{BM17} established an asymptotic law for the variance of the nodal surface of arithmetic random waves on the $3$-torus in the high-energy limit. In a subsequent work, Cammarota  \cite{Cam17} proved a universal non-Gaussian limit theorem for the nodal surface. In this paper, we study the nodal intersection length and the number of nodal intersection points associated, respectively, with two and three independent arithmetic random waves of same frequency on the $3$-torus. For these quantities, we compute their expected value, asymptotic variance as well as their limiting distribution. Our results are based on Wiener-It\^{o} expansions for the volume and naturally complement the findings of Cammarota \cite{Cam17}. At the core of our analysis lies an abstract cancellation phenomenon applicable to the study of level sets of arbitrary Gaussian random fields, that we believe has independent interest. \\

\noindent
\textbf{Keywords:} Arithmetic Random Waves, Limit theorems, Random fields, Nodal volumes, Gaussian analysis, Wiener-It\^{o} decomposition, Berry's Cancellation Phenomenon   \\
\textbf{AMS MSC 2010:} 60G60, 60B10, 60D05, 58J50, 35P20

\tableofcontents

\section{Introduction}
\subsection{Overview}
The present paper deals with the high-energy behaviour of the nodal set associated with arithmetic random waves (ARW) on the $3$-torus, $\T$. ARWs (first introduced in \cite{ORW08,RW08} for tori of arbitrary dimension) are Gaussian stationary eigenfunctions of the Laplace operator on the torus. In recent years, such a model has been intensively studied, in the framework of a more general program, focussing on the high-energy behaviour of local and non-local functionals of random Laplace eigenfunctions on generic manifolds (see e.g. \cite{CS14,WS19, R16, WY19,KKW13,RW08,GW17,MPRW16,DNPR16,Tod18,Tod19,PV19,DEL19}). 

Our specific aim is to extend the findings of \cite{BM17}, that first provided an exact asymptotic variance for the nodal surface area of the nodal set of ARW on $\T$, and \cite{Cam17}, that subsequently derived the limiting distribution of the normalised nodal surface area. 
More precisely, the goal of this paper is to study the high-energy behaviour of two further geometric quantities associated with vectors of ARWs, namely: (i) the nodal length of so-called \textit{dislocation lines} of ARWs (see e.g. \cite{Den01}), obtained when intersecting the zero sets of two independent ARWs with the same eigenvalue and (ii) 
the number of intersection points obtained when intersecting the zero sets of three independent ARWs with the same eigenvalue. For both quantities, we provide the exact expected value, precise variance asymptotics and second-order limit results. Our findings recover and extend the work of \cite{Cam17}. Such a contribution is the latest installment in a series of works exploiting Wiener chaos techniques for deriving limit results of geometric functionals associated with Gaussian fields (see e.g. \cite{CMW16,CMW16b,EL16,DNPR16,MPRW16,NPR17,Cam17,DEL19}). Our main source of arithmetic results, serving as building blocks for the nodal variance asymptotics, is \cite{BM17}.

An important contribution of our analysis is a detailed study of the Wiener-It\^{o} chaos expansion associated with non-linear geometric functionals of (possibly multi-dimensional) Gaussian fields admitting an integral representation in terms of generalised Jacobians (see Appendix \ref{AbsCanProof}). In particular, our findings of Section \ref{SubAbsCan} provide a full description of a general cancellation phenomenon that (i) explains all exact cancellations for the nodal length of Gaussian Laplace eigenfunctions on manifolds without boundary encountered so far (see e.g. \cite{DNPR16,MRW17,MPRW16,Cam17}); (ii) contains as special cases the projection formulae (see also Appendix \ref{Constants4}) for nodal length and number of phase singularities of Berry's Random Wave model (see \cite{NPR17}).

\paragraph{Notation.}
Throughout this paper, every random object is defined on a probability space $(\Omega,\mathcal{F},\Prob)$. We denote by
$\E{\cdot}$ and $\V{\cdot}$ the mathematical expectation and the variance with respect to $\Prob$, respectively. Also, 
$\gamma(x):=(2\pi)^{-1/2}e^{-x^2/2}$ denotes the standard Gaussian probability density on the real line. 

For sequences $\{A_n:n\geq1\}, \{B_n:n\geq1\}$, we will use the notation $A_n \ll B_n$ or $A_n=O(B_n)$ to indicate that $A_n \leq C B_n$ for some absolute constant $C$. We write $A_n=o(B_n)$ whenever $A_n/B_n \to 0$ as $n \to \infty$. Also, we write $A_n \sim B_n$ whenever $A_n/B_n \to 1$ as $n \to \infty$. For random variables, the symbols $\eqLaw$ and $\Law$ denote equality and convergence in distribution, respectively. 

For an integer $n\geq1$, we write $[n]:=\{1,\ldots,n\}$. For $n \geq 0$, we denote by $\Id_n$ the $n$-dimensional identity matrix with the convention that $\Id_0:=0 \in \R$.  For $A \in \mathrm{Mat}_{p,q}(\R)$ and $B \in \mathrm{Mat}_{p',q'}(\R)$, we write 
\begin{equation*}
A \oplus B := \begin{pmatrix}
A & 0 \\
0 & B 
\end{pmatrix} \in \mathrm{Mat}_{p+p',q+q'}(\R) 
\end{equation*} 
for the direct sum of $A$ and $B$ with the convention $A \oplus \Id_0:=A$ for every $A \in \mathrm{Mat}_{p,q}(\R)$. Finally, we denote by $\ind{\cdot}$ the indicator function.

\subsection{Models of ARW and relevant existing results}
Let $(M,g)$ be a smooth compact Riemannian manifold and let $\Delta$ be the associated Beltrami-Laplace operator. The spectrum of $\Delta$ is purely discrete, that is: (i) there exists a non-decreasing sequence 
$\{\lambda_j:j\geq0\}$ of non-negative eigenvalues of $-\Delta$, customarily called the \textit{energy levels of $M$}, and (ii) the associated eigenfunctions $\{f_j:j\geq0\}$, satisfying 
\begin{eqnarray}\label{Helm}
\Delta f_j + \lambda_j f_j  = 0 \ , \quad j\geq 0\ ,
\end{eqnarray}
form an $L^2(M)$-orthonormal system. The \textit{nodal set} of $f_j$ is its zero set $f_j^{-1}(\{0\})$. 
In \cite{Ch76} it is shown that, except on a closed set of lower dimension, $f_j^{-1}(\{0\})\subset M$ is a submanifold of codimension one. Of particular interest are quantities associated with the nodal set of $f_j$, such as the \textit{nodal volume}, in the \textit{high-energy} regime, that is, as $\lambda_j \to \infty$. 
Yau's conjecture \cite{Yau82,Yau93} asserts that 
there exist constants $c_M,C_M>0$, uniquely  depending on $M$, such that 
\begin{eqnarray*}
c_M \sqrt{\lambda_j} \leq \mathrm{vol}(f_j^{-1}(\{0\})) \leq C_M \sqrt{\lambda_j} \ ,
\end{eqnarray*}
with $\mathrm{vol}(\cdot)$ denoting the volume measure on $M$. 
This conjecture was proven for real-analytic manifolds $M$ in \cite{DF88}, whereas the lower bound is a result by \cite{Log18} in the more general case where $M$ is smooth.

\paragraph{Arithmetic random waves on $\mathbb{T}^d$.}
Let us specialize the above framework to the setting of the $d$-dimensional torus. 
Let $d\geq 1$ be an integer, let $M=\mathbb{T}^d=\R^d/\Z^d=[0,1]^d/_{\sim}$ denote the $d$-dimensional flat torus, and let $\Delta$ be the Laplace-Beltrami operator on it. 
One is interested in quantities associated with the nodal sets of real-valued random eigenfunctions of $\Delta$, i.e. random solutions $f: \mathbb{T}^d \to \R$ of \eqref{Helm} for some appropriate $\lambda_j$. It is a known fact that the eigenvalues of $-\Delta$ are positive real numbers of the form $E=E_n=4\pi^2 n$, where $n\in S_d$, with
\begin{eqnarray*}
S_d:= \left\{m \geq 1: \exists (m_1,\ldots,m_d) \in \Z^d, m= m_1^2+\ldots+m_d^2 \right\} \ , 
\end{eqnarray*}
that is, $n$ is an integer expressible as a sum of $d$ integer squares. For $n \in S_d$, we introduce the set of \textit{frequencies}
\begin{eqnarray*}
\Lambda_n := \left\{ \lambda = (\lambda_1,\ldots,\lambda_d) \in \Z^d: \lambda_1^2+\ldots+\lambda_d^2 = n \right\} \ , 
\end{eqnarray*} 
and write $\mathrm{card}(\Lambda_n)=:\Nn$ ($\mathrm{card}$ denoting the cardinality; note that we do not mark the dependency on $d$) to indicate the number of ways in which $n$ can be represented as a sum of $d$ integer squares. 
An $L^2(\mathbb{T}^d)$-orthonormal system for the eigenspace $\mathscr{E}(E_n)$ associated with $E_n$ is given by the complex exponentials 
\begin{eqnarray*}
\left\{e_{\lambda}(\cdot ): = \exp(2\pi i \scal{\lambda}{\cdot}): \lambda \in \Lambda_n\right\}  \ , 
\end{eqnarray*}
so that $\dim\mathscr{E}(E_n)=\mathrm{card}(\Lambda_n)=\Nn$.
For $n \in S_d$, the \textit{arithmetic random wave of order $n$}, denoted by $T_n$, is defined as the following random linear combination of complex exponentials
\begin{eqnarray*}
T_n(x) = \frac{1}{\sqrt{\Nn}} \sum_{\lambda\in \Lambda_n}
a_{\lambda}e_{\lambda}(x) \ , \quad x \in \T \ ,
\end{eqnarray*}
where the coefficients $\{ a_{\lambda}: \lambda \in \Lambda_n\}$ are complex $\mathcal{N}(0,1)$-distributed\footnote{We say that a random variable $X$ has the complex $\mathcal{N}(0,1)$ distribution, if  $X=Y+iZ$ where $Y,Z$ are independent real $\mathcal{N}(0,1/2)$ random variables.} and independent except for the relation $a_{\lambda} = \overline{a_{-\lambda}}$, which makes $T_n$ real-valued. 
It is easily seen that the law of $T_n$ is uniquely characterized by the property of being a centred Gaussian field on $\tor{d}$ with covariance function
\begin{eqnarray}\label{covARW}
r_n(x,y)  := \E{T_n(x)\cdot T_n(y)} = \frac{1}{\Nn} \sum_{\lambda \in \Lambda_n} e_{\lambda}(x-y) =: r_n(x-y) \ .
\end{eqnarray}
The function $r_n$ depends only on the difference of the arguments, meaning that the field $\{T_n(x):x\in \mathbb{T}^d\}$ is stationary. Note that the normalisation factor $\Nn^{-1/2}$ in the definition of $T_n(x)$ does not change the zero set of $T_n$, and appears purely for computational reasons; indeed, it implies that $r_n(0)=1$, that is: for every $x \in \T$, the variance of $T_n(x)$ is equal to $1$. 

\paragraph{Equidistribution of lattice points on $S^{d-1}$.}
The set of frequencies $\Lambda_n$ induces a probability measure on the unit sphere $S^{d-1}\subset \R^d$, given by  
\begin{eqnarray*}
\mu_{n,d} := \frac{1}{\Nn} \sum_{\lambda\in \Lambda_n}
\delta_{\lambda/\sqrt{n} } \ ,
\end{eqnarray*}
where $\delta_{\lambda/\sqrt{n}}$ denotes the Dirac mass at $\lambda/\sqrt{n}$. Since the measure $\mu_{n,d}$ is compactly supported,
it is determined by its Fourier coefficients 
\begin{eqnarray*}
\widehat{\mu_{n,d}}(k) := \int_{S^{d-1}} z^{-k} \mu_{n,d}(dz) \ , \quad k \in \Z \ .
\end{eqnarray*}
Up to rescaling its argument, the measure $\mu_{n,d}$ is the spectral measure of the Gaussian field $\{T_n(x):x\in \mathbb{T}^d\}$, as can be seen by rewriting \eqref{covARW} as 
\begin{eqnarray*}\label{spec}
r_n(x-y) = \int_{S^{d-1}} \exp\big(2\pi i \scal{\sqrt{n}\xi}{x-y}\big) \mu_{n,d}(d\xi) \ .
\end{eqnarray*}
The problem of angular distribution of the lattice points in  dimension $d$ has been investigated by Linnik \cite{Lin67}. A notable difference arises when comparing dimensions $d=2$ and $d=3$: indeed, it is known that there exists a density 1 subsequence $\{n_j:j\geq1\} \subset S_2$ such that $\mu_{n_j,2}$ converges weakly to the uniform distribution on the unit circle as $\mathcal{N}_{n_j} \to \infty$ \cite{EH99}, but there are infinitely many other weak limits of $\{\mu_{n,2}: n \in S_2\}$; such limits are referred to as \textit{attainable measures} \cite{KW17}. Instead, when $d=3$, subject to the condition $n \to \infty, \notcon{n}{0,4,7}{8}$, the probability measures $\{\mu_{n, 3}: n \in S_3\}$ converge weakly to the uniform probability measure on $S^2$  \cite{D88}, implying asymptotic equidistribution
\cite{DS90}.  In this context, the arithmetic condition $\notcon{n}{0,4,7}{8}$ arises naturally from the result by Gauss and Legendre asserting that $n \in S_3$ if and only if $n$ is not of the form $4^a(8b+7)$ (see e.g. \cite{G85}).

\paragraph{Previous work on this model.}
ARWs on the $d$-dimensional torus have been introduced in \cite{ORW08}, where the authors consider the \textit{Leray measure} of the nodal set of ARWs and study its asymptotic variance. A quantitative Central Limit Theorem for the Leray measure on the two-dimensional torus (in the high-frequency limit) is provided in \cite{PR16}.
In \cite{RW08}, the authors
take interest in the $(d-1)$-dimensional nodal volume of ARWs. Denoting by $Z_n$ the zero set of $T_n$ and $\mathcal{V}_n:=\mathcal{H}_{d-1}(Z_n)$ its $(d-1)$-dimensional Hausdorff measure, the expected nodal volume is shown to be a constant multiple of the square root of the energy level, that is, $\E{\mathcal{V}_n} = C_d \sqrt{E_n}$, where $C_d$ is an explicit constant depending only on the dimension, which is in particular consistent with Yau's conjecture. Concerning the variance of the nodal volume, the authors derive the asymptotic upper bound
\begin{eqnarray*}
\V{\mathcal{V}_{n}} \ll \frac{E_n}{\sqrt{\Nn}} \ , \quad \Nn \to \infty 
\end{eqnarray*}
and conjecture the stronger bound $\ll E_n/\Nn$ to hold.

Recent developments on the two and three-dimensional torus concerning exact asymptotic laws for variances and subsequent second-order results for fluctuations of quantities associated with the nodal set of Laplacian eigenfunctions have gained great attention in the literature. We will now briefly discuss these works.

\medskip\noindent\textit{\underline{Work on the two-dimensional torus.}}  
In \cite{KKW13}, for any probability measure $\mu$ on the circle, the authors define  
\begin{eqnarray*}
c(\mu):= \frac{1+\hat{\mu}(4)^2}{512} 
\end{eqnarray*} 
and derive a precise asymptotic law for the variance of the nodal length $\mathcal{L}_n$ of ARW, namely 
\begin{eqnarray}\label{Varc2}
\V{\mathcal{L}_n} \sim c(\mu_{n,2}) \cdot \frac{E_n}{\Nn^2}   \ , \quad \Nn \to \infty \ .
\end{eqnarray} 
This suggests that, if $\{n_j:j\geq1\} \subset S_2$ is a subsequence such that $\mu_{n_j,2}$ converges weakly to some symmetric probability measure $\mu$ on $S^1$, then $c(\mu_{n_j,2}) \to c(\mu)$ as $\mathcal{N}_{n_j} \to \infty$ and hence
\begin{eqnarray}\label{VL2}
\V{\mathcal{L}_{n_j}} \sim c(\mu) \cdot \frac{E_{n_j}}{\mathcal{N}_{n_j}^2}   \ , \quad \mathcal{N}_{n_j} \to \infty \ ,
\end{eqnarray}  
yielding an asymptotic variance estimate with non-fluctuating order of magnitude. 
In particular, the order of magnitude of the variance is $E_n/\Nn^2$, which significantly improves the previously conjectured bound $E_n/\Nn$ in \cite{RW08}. Such a lower order of magnitude is known as \textit{Berry's arithmetic cancellation phenomenon}, which follows from the exact vanishing of the second-order projection of the Wiener-It\^{o} expansion of the nodal length, as pointed out in \cite{MPRW16};  
such a cancellation phenomenon is not observed when dealing with non-zero level sets, in which case the variance would be commensurate to $E_n/\Nn$.

The asymptotic estimate in \eqref{VL2} depends on the angular distribution of the lattice points, and is therefore referred to as a \textit{non-universal} result.
Second-order results of the normalised nodal length were addressed in \cite{MPRW16}, where the authors show that for a subsequence 
$\{n_j:j\geq1\} \subset S_2$ such that $|\widehat{\mu_{n_j,2}}(4)| \to \eta$, for some $\eta \in [0,1]$ and $\mathcal{N}_{n_j} \to \infty$, 
\begin{eqnarray*}
\frac{\mathcal{L}_{n_j}- \E{\mathcal{L}_{n_j}}}{\sqrt{\V{\mathcal{L}_{n_j}}}} \Law 
\frac{1}{2\sqrt{1+\eta^2}} 
\left( 2-(1+\eta) X_1^2 + (1-\eta) X_2^2 \right) \ ,  
\end{eqnarray*}
where $(X_1,X_2)$ is a standard Gaussian vector in dimension two. In particular, this shows that the limiting probability distribution of the normalised nodal length is parametrised by $\eta \in [0,1]$, which depends on the high-energy behaviour of the spectral measures $\mu_{n,2}$ via the fourth Fourier coefficient. This fact emphasizes that, similarly to the asymptotic law for the variance, the limiting distribution of the normalised length is also non-universal. It is easily checked that the above limiting distributions are different for distinct values of $\eta$ and non-Gaussian. A quantitative version of this limit theorem is proven in \cite{PR16}. 

Phase singularities of complex ARWs on the $2$-torus have been investigated in \cite{DNPR16}; there, the authors consider the number of intersection points of the nodal sets of two independent ARWs of same energy level. More precisely, if $T_n$ and $T_n'$ denote two independent ARWs associated with eigenvalue $E_n$ and $I_n:=\mathrm{card}(T_n^{-1}(\{0\}) \cap T_n'^{-1}(\{0\}))$, the authors establish the following non-universal asymptotic law for the variance: as $\Nn\to \infty$, 
\begin{eqnarray*}
\V{I_n} \sim c(\mu_{n,2}) \cdot \frac{E_n^2}{\Nn^2} \ , \quad 
c(\mu_{n,2}):= \frac{3\widehat{\mu_{n,2}}(4)^2+5}{128\pi^2} \ . 
\end{eqnarray*}
Similar to the asymptotic variance of the nodal length, the variance of $I_n$ fluctuates due to the fact that lattice points are not necessarily asymptotically equidistributed. The following distributional limit result is also provided: for $\{n_j:j\geq1\} \subset S_2$ such that $|\widehat{\mu_{n_j,2}}(4)| \to \eta$, for some $\eta \in [0,1]$ and $\mathcal{N}_{n_j} \to \infty$, 
\begin{eqnarray*}
\frac{I_{n_j}-\E{I_{n_j}}}{\sqrt{\V{I_{n_j}}}} \Law 
\frac{1}{2\sqrt{10+6\eta^2}} \left(
\frac{1+\eta}{2}A + \frac{1-\eta}{2}B- 2(C-2)
\right) \ ,
\end{eqnarray*}
where $A,B,C$ are independent random variables such that $A \eqLaw B \eqLaw 2X_1^2+2X_2^2-4X_3^2$ and $C\eqLaw X_1^2+X_2^2$, and $(X_1,X_2,X_3)$ is a standard Gaussian vector in dimension three.

Related work on the two-dimensional torus include the study of the volume of the nodal set intersected with a fixed reference curve \cite{RW18}, or line segment \cite{Maf17a}.
In \cite{BMW17} the authors restrict the nodal length of ARWs to shrinking balls and prove that the restricted nodal length is asymptotically fully correlated with the total nodal length. In \cite{GW17}, Granville and Wigman study the small scale distribution of the $L^2$-mass of Laplacian eigenfunctions.

\medskip\noindent\textit{\underline{Work on the three-dimensional torus.}}
Statements on the three-dimensional torus include the arithmetic relation $\notcon{n}{0,4,7}{8}$ and, unlike the two-dimensional case, they do not rely on the spectral measures $\{\mu_{n,3}:n \in S_3\}$ due to equidistribution of lattice points on the unit two-sphere. 
The existing literature in $d=3$ considers the nodal set $Z_n$ of $T_n$ and its two-dimensional Hausdorff measure $\mathcal{A}_n:=\mathcal{H}_2(Z_n)$, that is the nodal surface of $Z_n$. In \cite{BM17}, an exact asymptotic law for the variance is provided, namely as $n\to \infty, \notcon{n}{0,4,7}{8}$, 
\begin{eqnarray} \label{VarMaf}
\V{\mathcal{A}_n} = \frac{n}{\Nn^2} \left( \frac{32}{375} + O\left(n^{-1/28+o(1)}\right)\right) \ .
\end{eqnarray}
Similarly to the two-dimensional case, the order of magnitude of the variance is commensurate to $E_n/\Nn^2$, which originates from the cancellation of the second chaotic projection in the Wiener chaos expansion of the nodal surface. As a consequence of the asymptotic equidistribution of lattice points on $S^2$, the leading coefficient in front of $n/\Nn^2$ in \eqref{VarMaf} does not fluctuate. The limiting distribution of the normalised nodal surface was investigated in \cite{Cam17}, where the following non-Gaussian, \textit{universal} result was derived: as $n\to \infty,\notcon{n}{0,4,7}{8}$, 
\begin{eqnarray*}
\frac{\mathcal{A}_n-\E{\mathcal{A}_n}}{\sqrt{\V{\mathcal{A}_n}}}
\Law \frac{1}{\sqrt{10}} \cdot \big(5-\chi^2(5)\big) \ ,
\end{eqnarray*}
where $\chi^2(5)$ denotes a chi-squared random variable with $5$ degrees of freedom. 
This distributional limit result is analogous to the case $d=2$ in the sense that the limiting distribution is a linear combination of independent chi-squared random variables, but does not involve any non-universality phenomenon.  

Results on the intersection of nodal sets against a surface can be found in \cite{RWY15,RW16}, see also \cite{M19} for a study of the intersection length obtained when intersecting nodal sets of ARWs with planes.

\subsection{Our main results}
Let $T_n$ be an arithmetic random wave on $\T$ and  
$T_n^{(1)},T_n^{(2)},T_n^{(3)}$ be i.i.d. copies of $T_n$. 
Fix $\ell \in [3]$ and consider the centred $\ell$-dimensional Gaussian field
\begin{equation}\label{Tnl}
\bT_n^{(\ell)} := \left\{ \bT_n^{(\ell)}(x) := (T_n^{(1)}(x),\ldots, T_n^{(\ell)}(x)) : x \in \T\right\} \ ,
\end{equation}
to which we associate the quantity 
\begin{equation}\label{defL}
L_n^{(\ell)} := \mathcal{H}_{3-\ell}\bigg(\bigcap_{i=1}^{\ell} 
\big(T_n^{(i)}\big)^{-1}(\{0\})  \bigg) \ ,
\end{equation}
where, for a $k$-dimensional measurable domain $A\subset \T$, $\mathcal{H}_k(A)$ denotes the $k$-dimensional Haussdorff measure of $A$, that is $(\mathcal{H}_2,\mathcal{H}_1,\mathcal{H}_0) = (\mathrm{area}, \mathrm{length}, \mathrm{card})$.
We denote the normalised nodal volume by
\begin{equation*}
\widetilde{L_n^{(\ell)}} := \frac{L_n^{(\ell)}-\E{L_n^{(\ell)}}}{\V{L_n^{(\ell)}}^{1/2} }  \ .
\end{equation*} 
Since $T_n^{(1)},T_n^{(2)}$ and $T_n^{(3)}$ are i.i.d. copies of $T_n$, we have 
\begin{equation*}
r^{(i)}_n(x-y) := \E{T_n^{(i)}(x)\cdot T_n^{(i)}(y)} = r_n(x-y) \ , \quad i \in [\ell]  , 
\end{equation*}
where $r_n$ is as in \eqref{covARW}.

Our main result, stated in Theorem \ref{MainThm2} below, provides exact second order results for the three quantities $L_n^{(1)},L_n^{(2)},L_n^{(3)}$, and thus contains the findings of \cite{Cam17} in the special case $\ell=1$. The statement is divided into three parts: (i) gives the precise expected nodal volume, (ii) is an asymptotic law for the nodal variance and (iii) concerns the second-order fluctuations of the normalised version of the nodal volume. 

\begin{Thm}\label{MainThm2}
Let the above notation prevail. Then the following holds:
\begin{enumerate}[label=\rm{(\roman*)}]
\item (Expected nodal volume)
For every $n \in S_3$,  
\begin{equation*}
\E{L_n^{(\ell)}} = \left\{
\begin{array}{lr}
\displaystyle{\frac{2 \sqrt{E_n}}{\sqrt{3}\pi}} \ ,  &  \ell=1 \\[0.3cm]
\displaystyle{\frac{E_n}{3\pi}} \ ,  &   \ell=2 \\[0.3cm]
\displaystyle{\frac{E_n^{3/2}}{3\sqrt{3}\pi^2}} \ , &   \ell=3
\end{array} 
\right.
\end{equation*}
\item (Universal asymptotic nodal variance)
As $n\to \infty, \notcon{n}{0,4,7}{8}$, 
\begin{equation*}
\V{L_n^{(\ell)}} \sim \left\{
\begin{array}{lr}
\displaystyle{\frac{E_n}{\Nn^2} \cdot \frac{8}{375\pi^2}} \ , & \ell=1 \\[0.3cm]
\displaystyle{\frac{E_n^2}{\Nn^2}\cdot \frac{316}{3375\pi^2}} \ ,  & \ell=2 \\[0.3cm]
\displaystyle{\frac{E_n^3}{\Nn^2}\cdot \frac{62}{675\pi^4}} \ , & \ell=3
\end{array} 
\right.
\end{equation*}
\item (Universal asymptotic distribution of the nodal volume) 
As $n\to \infty, \notcon{n}{0,4,7}{8}$,
\begin{equation*}
\widetilde{L_n^{(\ell)}} \Law \left\{
\begin{array}{lr}
\displaystyle{-\frac{1}{\sqrt{10}} \hat{\xi}_1(5)} \ , & \ell=1 \\[0.3cm]
\displaystyle{5\sqrt{\frac{15}{79}} \cdot \bigg(
-\frac{1}{50}\hat{\xi}_1(10) 
- \frac{1}{25}\hat{\xi}_2(5)
+ \frac{1}{25}\hat{\xi}_3(5) 
+ \frac{1}{50} \hat{\xi}_4(5)
- \frac{1}{6}\hat{\xi}_5(3)
\bigg) } \ ,  & \ell=2 \\[0.3cm]
\displaystyle{5\sqrt{\frac{2}{31}} \cdot \bigg(
-\frac{1}{50}\hat{\xi}_1(15) 
- \frac{1}{25}\hat{\xi}_2(15)
+ \frac{1}{25}\hat{\xi}_3(15) 
+ \frac{1}{50} \hat{\xi}_4(15)
- \frac{1}{6}\hat{\xi}_5(9)
\bigg)} \ , & \ell=3
\end{array} 
\right.
\end{equation*}
where, in each line, the symbols $\hat{\xi}_{i}(k_i)$ denote independent centred chi-squared random variables with $k_i$ degrees of freedom. 
\end{enumerate}
\end{Thm}

\paragraph{Some remarks.} 
\textbf{(a)}  We point out that the results stated separately in Theorem \ref{MainThm2} can be written in a compact form. For integers $1\leq \ell \leq  k $, we set  
\begin{equation}\label{defalk}
\alpha(\ell,k):= \frac{(k)_{\ell}\kappa_k}{(2\pi)^{\ell/2}\kappa_{k-\ell}}   \ , 
\end{equation}
where $(k)_{\ell}:=k!/(k-\ell)!$ and $\kappa_k:=\frac{\pi^{k/2}}{\Gamma(1+k/2)}$ stands for the volume of the unit ball in $\R^k$. Then, the content of Theorem \ref{MainThm2} can be restated as follows: for every $\ell \in [3]$, one has that
\begin{enumerate}[label=(\roman*)]
\item For every $n \in S_3$, 
\begin{equation}\label{meanL} 
\E{L_n^{(\ell)}} = \bigg(\frac{E_n}{3}\bigg)^{\ell/2}\frac{\alpha(\ell,3)}{(2\pi)^{\ell/2}} \ .
\end{equation}
\item  As $n\to \infty, \notcon{n}{0,4,7}{8}$, 
\begin{equation}\label{AsyVar}
\V{L_n^{(\ell)}} \sim \big(c_n^{(\ell)}\big)^2  
\bigg( \ell \cdot \frac{1}{250} + \frac{\ell(\ell-1)}{2} \cdot \frac{76}{375}\bigg) \ ,
\end{equation}
where 
\begin{equation*}
c_n^{(\ell)} = \bigg(\frac{E_n}{3}\bigg)^{\ell/2}
\frac{2}{(2\pi)^{\ell/2}} \frac{\alpha(\ell,3)}{\Nn} \ . 
\end{equation*}
\item As $n\to \infty, \notcon{n}{0,4,7}{8}$,
\begin{equation}\label{AsyDist}
\widetilde{L_n^{(\ell)}}  \Law 
\bigg(\ell\cdot \frac{1}{250}+ \frac{\ell(\ell-1)}{2}\cdot \frac{76}{375}\bigg)^{-1/2} Y^{(\ell)} M^{(\ell)} (Y^{(\ell)})^T \ ,
\end{equation}
where $Y^{(\ell)} \sim \mathcal{N}_{\ell(9\ell-4)}(0,\Id_{\ell(9\ell-4)})$ is a $\ell(9\ell-4)$-dimensional standard Gaussian vector and $M^{(\ell)} \in \mathrm{Mat}_{\ell(9\ell-4),\ell(9\ell-4)}(\R)$ is the deterministic matrix given by 
\begin{equation*}
M^{(\ell)} = \frac{-1}{50} \Id_{5\ell} 
\oplus \frac{-1}{25}\Id_{\frac{5\ell(\ell-1)}{2}}
\oplus \frac{1}{25}\Id_{\frac{5\ell(\ell-1)}{2}}
\oplus \frac{1}{50}\Id_{\frac{5\ell(\ell-1)}{2}}
\oplus \frac{-1}{6} \Id_{\frac{3\ell(\ell-1)}{2}} 
\ . 
\end{equation*}
\end{enumerate}
For the point (iii) above, we observe that  $Y^{(\ell)} M^{(\ell)} (Y^{(\ell)})^T $ in \eqref{AsyDist} is a diagonal quadratic form that has the same probability distribution as 
\begin{gather*}
-\frac{1}{50}\hat{\xi}_1(5\ell) 
- \frac{1}{25}\hat{\xi}_2\bigg(\frac{5\ell(\ell-1)}{2}\bigg)
+ \frac{1}{25}\hat{\xi}_3\bigg(\frac{5\ell(\ell-1)}{2}\bigg) 
+ \frac{1}{50} \hat{\xi}_4\bigg(\frac{5\ell(\ell-1)}{2}\bigg)
- \frac{1}{6}\hat{\xi}_5\bigg(\frac{3\ell(\ell-1)}{2}\bigg)  
\end{gather*}
where $\{\hat{\xi}_i(k_i):i =1,\ldots, 5\}$ denote  independent centred chi-squared random variables with $k_i\geq0$ degrees of freedom
with the convention $\hat{\xi}_i(0) \equiv 0$. In particular, this shows that for every $\ell\in [3]$, in the high-energy regime, the normalised nodal volume exhibits \textit{universal} and \textit{non-Gaussian} second-order fluctuations. 

\textbf{(b)} As already discussed, for $\ell=1$, Theorem \ref{MainThm2} coincides with known results on the nodal surface area on the three-dimensional torus: Indeed part (ii) gives 
\begin{equation*}
\V{L_n^{(1)}} \sim \frac{E_n}{\Nn^2}\cdot \frac{8}{375\pi^2} = 
\frac{n}{\Nn^2} \cdot \frac{32}{375} \ ,  \qquad n\to \infty, \notcon{n}{0,4,7}{8} \ ,
\end{equation*}
thus recovering the same order of magnitude as in Theorem 1.2 of \cite{BM17}, whereas our limit result (iii) is Theorem 1 of \cite{Cam17}: as $n \to \infty, \notcon{n}{0,4,7}{8}$,
\begin{equation*}
\widetilde{L_n^{(1)}} \Law \sqrt{250} \cdot Y^{(1)}M^{(1)}(Y^{(1)})^T \eqLaw 
\frac{1}{\sqrt{10}}\big(5-\xi(5)\big) \ ,
\end{equation*}
where $\xi(5) \eqLaw Y_1^2+ \ldots+ Y_5^2$.

\textbf{(c)} Note that the order of magnitude $E_n^{\ell}/\Nn^2$ of the asymptotic variance in each of the cases $\ell=1,2,3$ is consistent with what is observed in other models. As we will prove, this fact emerges from the vanishing of the second Wiener chaotic component of $L_n^{(\ell)}$. An abstract cancellation phenomenon for functionals of Gaussian fields, applicable to the setting of level sets of Laplacian eigenfunctions, is stated in Theorem \ref{MainThm1} (ii).

\medskip
We also point out that the statements (i) and (ii) of Theorem \ref{MainThm2} are sufficient to derive a \textit{universal weak law of large numbers}; it tells that the distribution of the normalised random variable $L_n^{(\ell)}/E_n^{\ell/2}$ is asymptotically concentrated around its mean: 
\begin{Cor}
For every $\delta>0$, as $n \to \infty, \notcon{n}{0,4,7}{8}$, we have 
\begin{eqnarray*}
\Prob\bigg[ \bigg|\frac{L_n^{(\ell)}}{E_n^{\ell/2}} -  \frac{\alpha(\ell,3)}{3^{\ell/2}(2\pi)^{\ell/2}}\bigg| > \delta\bigg] = o(1) \ .
\end{eqnarray*}
\end{Cor}
This immediately follows from Chebyshev's Inequality: as $n\to \infty,\notcon{n}{0,4,7}{8}$, 
\begin{eqnarray*} 
\Prob\bigg[ \bigg|\frac{L_n^{(\ell)}}{E_n^{\ell/2}} -  \frac{\alpha(\ell,3)}{3^{\ell/2}(2\pi)^{\ell/2}}\bigg| > \delta\bigg] 
\leq \frac{1}{\delta^2} \cdot \V{\frac{L_n^{(\ell)}}{E_n^{\ell/2}}} = \frac{c_{\ell}}{\delta^2 \Nn^2}(1+o(1)) = o(1) \ ,
\end{eqnarray*}
where $c_{\ell}$ is a constant only depending on $\ell$.
}

\begin{Rem} 
In the following two points listed below, we highlight further technical novelties appearing in the proof of Theorem \ref{MainThm2}.
\begin{enumerate}[label=\textbf{(\alph*)}] 
\item The chaos expansions of $L_n^{(\ell)}$ is obtained from the Area/Co-Area formula by an approximation argument similar to those used in \cite{KL}, where the authors discuss Gaussian limit theorems for general level functionals associated with stationary Gaussian fields with integrable covariance function. Our arguments for proving existence in $L^2(\Prob)$ rely on the use of an adequate partition of the torus into singular and non-singular regions, see for instance \cite{ORW08,KKW13}. To the best of our expertise, although such a route has already been effectively exploited for obtaining variance estimates for higher-order chaotic projections of nodal quantities (see \cite{PR16,DNPR16,NPR17}), 
this approach for proving existence results in $L^2(\Prob)$ for geometric functionals associated with multi-dimensional Gaussian fields appears for the first time in the literature. We also stress that the argument based on almost surely bounding the nodal length $L_n^{(1)}$ associated with a single ARW (see \cite{RW08} and \cite{Cam17}) does not apply in the case of more than one ARW, and therefore requires a different approach.
\item In order to derive the explicit expression of the fourth-order chaotic projection of $L_n^{(\ell)}$, we compute the Hermite projection coefficients associated with the mapping $\bX\mapsto  \det (\bX \bX^T)^{1/2}$, where $\bX$ is a  $\ell \times 3$ matrix. In order to do this, we tackle the more general task of computing these projection coefficients in the case where $\bX$ is a generic $\ell \times k$ matrix. Our techniques build on standard properties of the Gaussian distribution as well as Gramian determinants, and in particular recover the known results obtained in Lemma 3.3,  \cite{DNPR16}. \end{enumerate}
\end{Rem}

\subsection{Further connection with literature }
\underline{\textit{Berry's Random Wave Model.}}
In \cite{Ber77}, Berry introduced the so-called \textit{Berry Random Wave model} (BRW), that is, the unique translation-invariant centred Gaussian field $B_j=\{B_j(x):x\in \R^2\}$ on the plane with covariance function
\begin{eqnarray}\label{J0}
r_j(x,y)= \E{B_j(x) \cdot B_j(y) } = J_0(\sqrt{\lambda_j} \cdot \norm{x-y}) =:r_j(x-y) \ , \quad  (x,y) \in \R^2 \times \R^2 \ ,
\end{eqnarray}
with $J_0$ denoting the Bessel function of order $0$ of the first kind and $\norm{\cdot}$ the Euclidean norm in $\R^2$. Berry conjectured that local aspects of the geometry of zero sets of generic high-energy Laplace eigenfunctions on a two-dimensional manifold can be modelled by the BRW. More precisely, his observation proposes that  eigenfunctions of chaotic systems locally 'behave' like a random superposition of plane waves with fixed energy.  
Since Berry's publication \cite{Ber02}, the study of local and non-local features associated with the geometry of nodal and (non-zero) level sets of high-energy Gaussian Laplace eigenfunctions has gained substantial consideration and different models have been studied in recent years, the case of \textit{random spherical harmonics} on the $2$-sphere (see e.g. \cite{MRW17,Ros16,Wig10,MP11}) and \textit{arithmetic random waves} on the torus (see e.g. \cite{ORW08,RW08,KKW13,Cam17,DNPR16,MPRW16}) being of particular importance. 
The study of BRW on $\R^3$ has been initiated in \cite{DEL19}. Therein, the authors consider the nodal length restricted to growing cubes of the complex BRW and distinguish between isotropic and anisotropic covariance functions. In the isotropic case, they show that the limiting distribution of the nodal length is Gaussian whenever the underlying covariance function of the model is square-integrable. The proof of such a Central Limit Theorem, based on the Wiener chaos expansion of the nodal length, reveals in particular that, in this framework, \textit{all} the chaoses except the second contribute to the limit. As we will see, such an observation should be contrasted with our results, based on the dominance of the fourth Wiener chaos alone.
In \cite{CH16,Zel09}, the authors study  \textit{monochromatic random waves} on a general smooth compact manifold, that is, Gaussian linear combinations of eigenfunctions associated with eigenvalues ranging in a short interval. 

\medskip
\noindent\underline{\textit{Berry's Cancellation Phenomenon.}}
Berry's cancellation phenomenon was first observed in \cite{Ber02} for nodal sets of BRW. Using the notation introduced in \eqref{J0}, Berry considered the length $L_j(D)$ of the nodal lines of $B_j$ (Berry random wave for eigenvalue $\lambda_j$) and the number of nodal points $N_j(D)$ of the complex version of the BRW, i.e. the random field $\{B_j(x)+iB_j'(x):x\in \R^2\}$, with $B_j'$ denoting an independent copy of $B_j$, when both statistics are restricted to a compact domain $D$. For these observables, denoting $\mathcal{A}_D$ the area of $D$, Berry obtained 
\begin{eqnarray*} 
\E{L_j(D)} = \frac{\mathcal{A}_D}{2\sqrt{2}} \sqrt{\lambda_j} \ , \quad 
\E{N_j(D)} = \frac{\mathcal{A}_D}{4\pi} \lambda_j \ ; 
\end{eqnarray*}
as well as variance asymptotics, as $j \to \infty$
\begin{eqnarray}\label{Brw}
\V{L_j(D)} \sim \frac{\mathcal{A}_D}{256\pi} \log(\sqrt{\lambda_j}   \sqrt{\mathcal{A}_D}) \ , \quad 
\V{N_j(D)} = \frac{11\mathcal{A}_D}{64\pi^3} 
\lambda_j\log(\sqrt{\lambda_j}   \sqrt{\mathcal{A}_D})  \ .
\end{eqnarray}
In \cite{NPR17}, the authors recover these results and show that the properly scaled versions of $L_j(D)$ and $N_j(D)$ satisfy a central limit theorem in the high-energy regime. 
Berry's cancellation phenomenon essentially concerns the order of magnitude of the asymptotic variance in \eqref{Brw}: indeed, its \textit{logarithmic} order is unexpectedly smaller than a natural prediction. Loosely speaking, such a lower order of magnitude originates from the exact cancellation of the leading term in the  \textit{Kac-Rice formula} for the variance.  
A general explanation of such a cancellation, based on the use of Wiener-chaos expansions of the nodal volumes, distilling the main ideas introduced in \cite{MPRW16,DNPR16,NPR17} into a general principle, will be developed in the forthcoming sections. 

\subsection{Plan of the paper}
In Section \ref{AbsCan}, we provide a general result (see Theorem \ref{MainThm1}) leading to cancellation phenomena in the setting of geometric functionals associated with nodal sets of multiple independent Gaussian fields. The proof is deferred to Appendix \ref{AbsCanProof}.
The proof of Theorem \ref{MainThm2} on nodal sets of arithmetic random waves on the three-torus is the content of Section \ref{ARWProof}. 
Appendices \ref{Constants4}-\ref{AppSing} contain proofs of technical results needed for the proof of Theorem \ref{MainThm2}.

\subsection*{Acknowledgement}
The author thanks Professor  Giovanni Peccati for his guidance throughout this work and acknowledges support of the Luxembourg National Research Fund PRIDE15/10949314/GSM.

\section{Wiener Chaos and abstract cancellation phenomena}\label{AbsCan}
In this section, we present some general results about non-linear functionals of Gaussian fields that admit an integral representation in terms of Dirac masses and Jacobians. As discussed in Section \ref{Exs}, this contains as special cases exact and partial cancellations discovered in \cite{DNPR16,NPR17,MPRW16,MRW17}.

\subsection{Preliminaries on Wiener Chaos}\label{WCPrel}
We briefly recall standard facts from Gaussian analysis. For further details, the reader is referred to the monographs \cite{NP12,N06}.

\medskip
Let $\{H_k:k\geq0\}$ denote the family of Hermite polynomials on the real line given recursively by 
\begin{eqnarray*}
H_0(x)=1, \ H_k(x) = xH_{k-1}(x)- H'_{k-1}(x) \ , \quad 
k \geq 1 \ .
\end{eqnarray*}
The first few are then given by 
\begin{eqnarray*}
H_0(x)=1, \ H_1(x) = x, \ H_2(x)=x^2-1, \ H_3(x) = x^3-3x, \ 
H_4(x) = x^4-6x^2+3 \ . 
\end{eqnarray*}
Moreover, the following symmetry relation holds for every $k\geq 0$, and every $x\in \R$,
\begin{eqnarray}\label{Hsym}
H_k(-x)=(-1)^kH_k(x) \ . 
\end{eqnarray}
It is well-known that $\mathbb{H}:=\{H_k/\sqrt{k!}:k\geq 0\}$ forms a complete orthonormal system of $L^2(\gamma)=:L^2(\R,\mathscr{B}(\R),\gamma(x)dx)$, where $\gamma(x)$ denotes the standard Gaussian probability density function.

Let $G=\{G(u):u\in \mathscr{U}\}$ denote a centred Gaussian field on a generic set $\mathscr{U}$ and let $\mathbb{G}$ be the real Gaussian Hilbert space obtained as the $L^2(\Prob)$-closure of the vector space 
of all finite real linear combinations of elements of $G$. For an integer $q\geq 0$, we then denote by $\mathbf{C}_q^{\mathbb{G}}$ the $q$-th Wiener chaos associated with $\mathbb{G}$, that is, the $L^2(\Prob)$-closure of the vector space of all finite real linear combinations of elements of the form 
\begin{eqnarray*}
\prod_{j=1}^m H_{q_j}(X_j) \ , \quad m \geq 1 \ ,
\end{eqnarray*}
such that $q_1+\ldots+q_m = q $ and $(X_1,\ldots,X_m)$ is a standard $m$-dimensional Gaussian vector extracted from $\mathbb{G}$. 
In particular, $\mathbf{C}_0^{\mathbb{G}}= \R$ consists of all constant random variables.  
Since $\mathbb{H}$ is an orthonormal system of $L^2(\gamma)$, it follows that whenever $q\neq q'$, the spaces $\mathbf{C}_q^{\mathbb{G}}$ and $\mathbf{C}_{q'}^{\mathbb{G}}$ are orthogonal with respect to the inner product of $L^2(\Prob)$, and one has the following decomposition 
\begin{eqnarray*}
L^2(\Omega, \sigma(\mathbb{G}), \Prob) 
= \bigoplus_{q \geq 0} \mathbf{C}_q^{\mathbb{G}} \ , 
\end{eqnarray*}
that is, every $\sigma(\mathbb{G})$-measurable random variable $F$ can be uniquely written as series (converging in the $L^2(\Prob)$-sense)
\begin{eqnarray}\label{WCF}
F = \sum_{q\geq 0} \proj_{q}(F) \ , 
\end{eqnarray}
where for $q\geq 0, \proj_{q}(F) \in \mathbf{C}_q^{\mathbb{G}}$ denotes the projection of $F$ onto $\mathbf{C}_q^{\mathbb{G}}$. Moreover, since 
$\mathbf{C}_0^{\mathbb{G}}=\R$, it follows that $\proj_0(F) = \E{F}$.

\subsection{An abstract cancellation phenomenon}\label{SubAbsCan}
We consider a finite measurable space $(Z,\mathscr{Z},\mu)$ such that 
$\mu(Z)=1$. Let $G=\{G(z):z\in Z\}$ be a real-valued centred Gaussian field indexed by $Z$. For an integer $\ell \geq 1$, let $G^{(1)},\ldots, G^{(\ell)}$ be i.i.d. copies of $G$ and write $\mathbf{G}=\{\mathbf{G}(z)=(G^{(1)}(z),\ldots,G^{(\ell)}(z)):z\in Z\}$ to indicate the associated  $\ell$-dimensional Gaussian field. Additionally, let $W=\{W(z):z \in Z\}$ be a (not necessarily Gaussian) random field indexed by $Z$. We denote by $\delta_u$ the Dirac mass at $u \in \R$. We introduce the following definition. 
\begin{Def}\label{J}
For every $u^{(\ell)}:=(u_1,\ldots,u_{\ell}) \in \R^{\ell}$, we define the random variable 
\begin{eqnarray}\label{DefJ}
&&J(\mathbf{G},W;u^{(\ell)}) := \int_{Z} \prod_{i=1}^{\ell} \delta_{u_i}(G^{(i)}(z)) \cdot W(z) \ \mu(dz) \notag \\
&& := \lim_{\eps \to 0}  \int_Z 
(2\eps)^{-\ell } \prod_{i=1}^{\ell} \ind{[-\eps,\eps]}(G^{(i)}(z)-u_i) \cdot W(z) \ \mu(dz) 
\end{eqnarray}
whenever the limit exists $\Prob$-almost surely. In the case where the limit exists in $L^p(\Prob)$ for $p \geq 1$, we say that $J(\mathbf{G},W;u^{(\ell)})$ is \textit{well-defined} in $L^p(\Prob)$. 
\end{Def}
Our aim is to study the Wiener-It\^{o} chaos expansion of $J(\mathbf{G},W;u^{(\ell)})$. As we will prove later (see Lemma \ref{LemASL2}), the nodal volumes $L_n^{(\ell)}, \ell \in [3]$ defined in \eqref{defL} are obtained $\Prob$-a.s. and in $L^2(\Prob)$ as $L_{n}^{(\ell)}=J(\mathbf{G},W,(0,\ldots,0))$,
where $\mathbf{G}=\bT_n^{(\ell)}$ is as in \eqref{Tnl} and $W(z)$ is the square root of the Gramian determinant of the Jacobian matrix of $\bT_n^{(\ell)}$ computed at $z$.

\medskip
For integers $1\leq \ell \leq k$, we use the notation $
\bX = \big\{X^{(i)}_j: (i,j) \in [\ell] \times [k]\big\} $
to indicate a generic element of the class $\mathrm{Mat}_{\ell,k}(\R)$ of $\ell\times k$ matrices. The following definition generalizes the notion of Gramian determinants.
\begin{Def}\label{DefPhi}
We say that a map $\Phi_{\ell,k}: \mathrm{Mat}_{\ell,k}(\R) \to \R_+$ satisfies \textit{Assumption A} if it satisfies the following four requirements for every $\bX \in \mathrm{Mat}_{\ell,k}(\R) $:
\begin{enumerate}
\item[(A1)] $\Phi_{\ell,k}$ is invariant under permutations of columns and rows of $\bX$, that is, 
\begin{equation*}
\Phi_{\ell,k}(\bX) =\Phi_{\ell,k}\big(\{X^{(i)}_{\sigma(j)}: (i,j) \in [\ell]\times[k]\big\})
= \Phi_{\ell,k}\big(\{X^{(\pi(i))}_{j}: (i,j) \in [\ell]\times[k]\big\})
\end{equation*}
for every permutation $\sigma$ of $[k]$ and $\pi$ of $[\ell]$.
\item[(A2)]  $\Phi_{\ell,k}$ is positively homogeneous as a function of the rows of $\bX$, that is, for every $c \in \R$ and every $i\in [\ell]$, $|c| \Phi_{\ell,k}(\bX) = \Phi_{\ell,k}(\bX^*)$,
where $\bX^*$ denotes the matrix obtained from $\bX$ by multiplying the $i$-th row by $c$. 
\item[(A3)] $\Phi_{\ell,k}$ is invariant under sign changes in the columns of $\bX$, that is, for every $j\in [k]$, $\Phi_{\ell,k}(\bX) = \Phi_{\ell,k}(\bX^*) $,
where $\bX^*$ denotes the matrix obtained from $\bX$ by multiplying the $j$-th column by $-1$. 
\item[(A4)] If $\ell \geq 2, \Phi_{\ell,k}$ is invariant under row addition, that is, $\Phi_{\ell,k}(\bX) = \Phi_{\ell,k}(\bX^*)$, 
where $\bX^*$ denotes the matrix obtained from $\bX$ by replacing its $i_1$-th row by the sum of its $i_1$-th and $i_2$-th row for $i_1\neq i_2 \in [\ell]$.
\end{enumerate}
\end{Def}
A prototype example of a function satisfying \textit{Assumption A} above is given by the Gramian determinant $
\Phi_{\ell,k}^*(\bX):= \det (\bX \bX^T)^{1/2} $    
as proved in Lemma \ref{CauchyBinet} of Appendix \ref{Constants4}. 
We stress that, although in the proof of Theorem \ref{MainThm2} on ARW, we use the particular function $\Phi_{\ell,k}^*(\bX):= \det (\bX \bX^T)^{1/2}$ in order to derive the Wiener-It\^{o} chaos expansion of the nodal volumes, our main result about cancellation phenomena stated in Theorem \ref{MainThm1} here below holds for any function $\Phi$ verifying  \textit{Assumption A}.

\medskip
To state our  result, we introduce the following objects:
\begin{itemize}
\item For every $i\in [\ell]$, let 
\begin{equation*}
\bX^{(i)} = \left\{ \bX^{(i)}(z) :=(X_0^{(i)}(z),X_1^{(i)}(z),\ldots,X_k^{(i)}(z)):z \in Z\right\}
\end{equation*}
be a $(k+1)$-dimensional standard Gaussian field, i.e. $\bX^{(i)}$ is a Gaussian family and for every $z\in Z$, the vector $\bX^{(i)}(z)$ is a standard $(k+1)$-dimensional Gaussian vector, that is, its  coordinates $X^{(i)}_j(z), j=0,\ldots, k$ are independent standard Gaussian random random variables.
For $z\in Z$, we let $\bX_{\star}^{(i)}(z):=(X_1^{(i)}(z),\ldots,X_k^{(i)}(z))$ and write 
\begin{eqnarray}\label{mat}
\bX_{\star}(z):=\left\{ X_j^{(i)}(z):(i,j) \in [\ell]\times[k]\right\}
\end{eqnarray}
for the $\ell \times k$ matrix whose $i$-th row is given by $\bX_{\star}^{(i)}(z)$.  If $\ell \geq 2$, for every $i_1 \neq i_2 \in [\ell]$, we assume that the random fields 
$\bX^{(i_1)}$ and $\bX^{(i_2)}$ are stochastically independent.
\item For every $i \in [\ell]$, we define the quantities
\begin{eqnarray}
D^{(i)} &:=& \frac{1}{k}\sum_{j=1}^k \int_Z X_j^{(i)}(z)^2 \ \mu(dz) - \int_Z X_0^{(i)}(z)^2 \ \mu(dz) \ , \label{Di} \\
m^{(i)} &:=& \int_Z X_0^{(i)}(z) \ \mu(dz) \ . \label{mi}  
\end{eqnarray}
\item Consider a map $\Phi_{\ell,k}:\R^{\ell \times k} \to \R_+$ that satisfies \textit{Assumption A} of Definition \ref{DefPhi} and  such that for every $z\in Z$,
\begin{equation*}
\E{\Phi_{\ell,k}(\bX_{\star}(z))^2}<\infty \ ,
\end{equation*}
and set 
\begin{eqnarray}\label{notalk}
\E{\Phi_{\ell,k}(\bX_{\star}(z))} =: \alpha_{\ell,k} \ .
\end{eqnarray}
\end{itemize}

\sloppy
Our next result provides the chaotic projections 
onto the $q$-th Wiener chaos associated with $\{\bX^{(1)},\ldots,\bX^{(\ell)}\}$ of the random variable $J(\mathbf{G},W;u^{(\ell)})$ defined in Definition \ref{J} in the case where 
\begin{eqnarray}\label{GW}
\mathbf{G}=\left\{(X_0^{(1)}(z),\ldots,X_0^{(\ell)}(z)):z\in Z\right\} \ , \quad 
W=\left\{\Phi_{\ell,k}(\bX_{\star}(z)):z\in Z\right\} \ .
\end{eqnarray}
Note that, for every $z \in Z$, $W(z)$ as defined in \eqref{GW} is $\sigma(\mathbf{G})$-measurable and stochastically independent of $\mathbf{G}(z)$. Part (ii) contains a general version of the chaos cancellation phenomenon observed e.g. in \cite{Wig10,DR19,KKW13,DNPR16,MPRW16,NPR17,Cam17}. Its proof is deferred to Appendix \ref{AbsCanProof}.

\begin{Thm}\label{MainThm1}
Assume the above setting. Then, we have:
\begin{enumerate}[label=\rm{(\roman*)}]
\item (General projection formulae) Fix $u^{(\ell)}:=(u_1,\ldots, u_{\ell}) \in \R^{\ell}$ and assume that $J(\mathbf{G},W;u^{(\ell)})$ with $(\mathbf{G},W)$ as in \eqref{GW} is well-defined in $L^2(\Prob)$ in the sense of Definition \ref{J}. Writing $J=J(\mathbf{G},W;u^{(\ell)})$, we have, for every $q\geq0$, 
\begin{gather}\label{Pq}
\proj_q(J) = \sum_{\substack{j_1,\ldots,j_{\ell},r\geq0 \\ j_1+\ldots+j_{\ell}+r=q}} \frac{\beta^{(u_1)}_{j_1}\cdots \beta^{(u_{\ell})}_{j_{\ell}}}{j_1!\ldots j_{\ell}!} \int_Z \prod_{i=1}^{\ell} H_{j_i}(G^{(i)}(z)) \cdot \proj_r(W(z)) \mu(dz) , 
\end{gather}
where $\{\beta^{(u_i)}_{j}:j\geq0\}$ denote the coefficients associated with the formal Hermite expansion of the Dirac mass $\delta_{u_i}$, given by 
\begin{eqnarray*}
\beta^{(u)}_j = \int_{\R} \delta_u(y)H_j(y) \gamma(y)  dy = H_j(u) \gamma(u) \ .
\end{eqnarray*}
In particular, 
\begin{eqnarray}
&&\proj_0(J) = \E{J} = \alpha_{\ell,k} \cdot \prod_{i=1}^{\ell}
\gamma(u_i) \ , \label{P0i} \\
&&\proj_1(J) = \alpha_{\ell,k} \cdot \prod_{i=1}^{\ell} \gamma(u_i)  \cdot
\sum_{i=1}^{\ell}m^{(i)}u_i \ , \label{P1i} \\
&&\proj_2(J)= \frac{\alpha_{\ell,k}}{2} \cdot  \prod_{i=1}^{\ell} \gamma(u_i)  \cdot 
\sum_{i=1}^{\ell} \bigg( u_i^2 \int_Z (X_0^{(i)}(z)^2-1)\ \mu(dz) + D^{(i)} \bigg) \ . \label{P2i} 
\end{eqnarray}
\item (Abstract cancellation) If $u_i=D^{(i)}=0$ for every $i\in[\ell]$, then (using \eqref{notalk})
\begin{eqnarray}
&&\proj_0(J) = \E{J} = \frac{\alpha_{\ell,k}}{(2\pi)^{\ell/2}} \ , \label{P0ii} \\
&&\proj_{2q+1}(J) = \proj_2(J)= 0  \ , \quad  q \geq 0 \ . \label{Can} 
\end{eqnarray}
\end{enumerate}
\end{Thm}

As anticipated, we will apply Theorem \ref{MainThm1} to the study of nodal sets of Gaussian Laplace eigenfunctions. The following section deals with two such examples.

\subsection{Applications to nodal sets of Gaussian Laplace eigenfunctions} \label{Exs}
We provide two examples of applications of Theorem \ref{MainThm1} dealing with nodal volumes associated with (possibly multi-dimensional) stationary Gaussian random fields that are Laplace eigenfunctions. Example (i) deals with ARWs on the $d$-dimensional torus and is effectively used in the proof of Theorem \ref{MainThm2}, whereas (ii) is Berry's random wave model in $\R^d$.

\paragraph{(i) ARW on $\tor{d}$.}
Let $d \geq 2$ and $(Z,\mathscr{Z},\mu) = (\tor{d}, \mathcal{B}(\tor{d}), dx)$ with $dx$ denoting the Lebesgue measure on $\R^d$. For integers $1 \leq \ell \leq d$, consider independent ARWs $T_n^{(1)},\ldots,T_n^{(\ell)}$ on $\tor{d}$. By a straightforward computation, we have that, for every $i \in [\ell]$ and $j \in [d]$, the partial derivatives $\partial_{j} T_n^{(i)}(x)$ are centred Gaussian random variables with variance
\begin{gather}\label{VarT}
\V{\partial_{j} T_n^{(i)}(x)} = \frac{E_n}{d} \ , \quad n \in S_d , \quad x \in \tor{d} , 
\end{gather}
where $\partial_{j}:=\partial/\partial x_j$. 
Let $\mathbf{G} = \{ (T_n^{(1)}(x),\ldots, T_n^{(\ell)}(x)) :x \in \tor{d} \}$ and  write 
$\tilde{\partial_j}:=(E_n/d)^{-1/2}\partial_j$ for the normalised derivatives. Denote by $\mathbf{G}_{\star}(x) $  the normalised Jacobian $\ell \times d$ matrix of $\mathbf{G}$ computed at $x \in \tor{d}$ and
consider the random field $W= \{ \Phi_{\ell,d}^*(\mathbf{G}_{\star}(x) ):x\in \tor{d}\}$ 
where $\Phi_{\ell,d}^*(A) = \det(AA^T)^{1/2}$ for $A \in \mathrm{Mat}_{\ell,d}(\R)$.
Then, using the Area/Co-Area formula (see e.g. Propositions 6.1 and 6.13 in \cite{AW09}), the random variable 
\begin{gather*}
L_{n}^{(\ell)}(d):=\bigg(\frac{E_n}{d}\bigg)^{\ell/2}J(\mathbf{G},W,(0,\ldots,0))
\end{gather*}
represents the $(d-\ell)$-dimensional volume of the zero set of $\mathbf{G}$, where $J$ is defined according to Definition \ref{J}. Note that $L_n^{(\ell)}(3)=L_n^{(\ell)}, \ell=1,2,3$ as defined in \eqref{defL}. The continuity result in Theorem \ref{ThmCont} shows that the nodal volume is defined $\Prob$-a.s.  The fact that the random variable $L_n^{(1)}(d)$ is well-defined in $L^2(\Prob)$ for $d\geq2$ is proved in \cite{RW08}, whereas the case $(\ell,d)=(2,2)$ is proved in \cite{DNPR16}. The remaining cases on the three-dimensional torus corresponding to $(\ell,d)=(2,3),(3,3)$ will be proved in Lemma \ref{LemASL2}, the existence in $L^2(\Prob)$ of the nodal volume for arbitrary $\ell$ and $d$ can be proved by similar arguments, for which we omit the details.
Now, for every $i\in [\ell]$, the quantity $D^{(i)}$ in \eqref{Di} satisfies
\begin{eqnarray*}
D^{(i)} &=& \frac{1}{d}\sum_{j=1}^d \int_{\tor{d}} \tilde{\partial_j} T_{n}^{(i)}(x)^2 \ dx - \int_{\tor{d}} T_n^{(i)}(x)^2 \ dx   \\
&=& \frac{1}{d}\int_{\tor{d}} \norm{ \tilde{\nabla} T^{(i)}_n(x)}^2 \ dx 
- \int _{\tor{d}} T_n^{(i)}(x)^2 \ dx \\
&=& \frac{1}{d}\int_{\tor{d}} \scal{\tilde{\nabla}T_n^{(i)}(x)}{\tilde{\nabla}T_n^{(i)}(x)} \ dx - \int _{\tor{d}} T_n^{(i)}(x)^2 \ dx \\ 
&=& \frac{1}{E_n}  \int_{\tor{d}} \scal{\nabla T_n^{(i)}(x)}{\nabla T_n^{(i)}(x)} \ dx - \int _{\tor{d}} T_n^{(i)}(x)^2 \ dx \ .
\end{eqnarray*}
Using Green's first identity (see e.g. \cite{L97}, p.44) and the fact that $\Delta T_n^{(i)}(x) = - E_n T_n^{(i)}(x)$, gives
\begin{equation*}
D^{(i)} = -\frac{1}{E_n} \int_{\T} T_n^{(i)}(x) \Delta T_n^{(i)}(x) \ dx 
- \int _{\T} T_n^{(i)}(x)^2 \ dx
= 0 \ .
\end{equation*} 
In particular, we conclude from \eqref{Can} that the second chaotic projection of the nodal volume $L_n^{(\ell)}$ is identically zero.

\paragraph{(ii) BRW on $\R^d$.} Let $1 \leq \ell \leq d$ be as above. Consider a compact convex set $D \subset \R^d$ with $C^1$ boundary $\partial D$. Let $(Z,\mathscr{Z},\mu) = (D, \mathcal{B}(D), dx)$. Write $\{B_E(x):x \in D\}$ to indicate Berry's random wave with parameter $E>1$ restricted to $D$, that is, $B_E$ is the stationary centred Gaussian Laplace eigenfunction on $\R^d$ with covariance function (see e.g. Theorem 5.7.2 \cite{AT09})
\begin{gather*}
\E{B_E(x)\cdot B_E(y)} = \frac{J_{(d-2)/2}(2\pi\sqrt{E}\norm{x-y})}{(2\pi\sqrt{E}\norm{x-y})^{(d-2)/2}}, \quad x,y \in D,
\end{gather*}
with $J_m$ denoting the Bessel function of order $m$ of the first kind, and energy $4\pi^2E$. Consider $B_E^{(1)},\ldots, B_E^{(\ell)}$ i.i.d. copies of $B_E$ and $\mathbf{G}= \{(B_E^{(1)}(x),\ldots,B_E^{(\ell)}(x)) : x \in D  \}$.
One can show by a direct computation, that for every $i \in [\ell]$ and $j\in [d]$,
\begin{gather*}\label{VarB}
\V{\partial_jB_E^{(i)}(x)} = \frac{4\pi^2 E}{d},  \quad x \in D.
\end{gather*}
As in Example (i), we write $\tilde{\partial_j}:=(4\pi^2E/d)^{-1/2}\partial_j$ for the normalised derivatives and consider the 
random field $W= \{ \Phi_{\ell,d}^*(\mathbf{G}_{\star}(x) ):x\in D\}$ with $\Phi_{\ell,d}^*(A) = \det(AA^T)^{1/2}$ for $A \in \mathrm{Mat}_{\ell,d}(\R)$. Then, the random variable 
\begin{gather*}
L_E^{(\ell)}(d):=\bigg(\frac{4\pi^2E}{d}\bigg)^{\ell/2}J(\mathbf{G},W,(0,\ldots,0))
\end{gather*}
is the $(d-\ell)$-dimensional nodal volume of $\mathbf{G}$, where as previously, $J$ is as in Definition \ref{DefJ}. Again, an application of Theorem \ref{ThmCont}, shows that $L_E^{(\ell)}(d)$  is well-defined $\Prob$-a.s. The existence in $L^2(\Prob)$ is proved in the cases $(\ell,d)=(1,2),(2,2)$ in \cite{NPR17} and the arguments therein can be extended to the case of arbitrary integers $\ell$ and $d$.
Arguing as in the previous example, using Green's identity, the quantity $D^{(i)}$ in \eqref{Di} is equal to 
\begin{gather*}
D^{(i)} = \frac{1}{4\pi^2 E}   \int_{\partial D}  B_E^{(i)}(x) \scal{\nabla B_E^{(i)}(x)}{n(x)}  dx,
\end{gather*}
where $n(x)$ denotes the outward unit normal vector to $\partial D$ at $x$.
In particular, $D^{(i)}$ and hence the second chaotic component of $L_E^{(\ell)}(d)$ reduce to an integration over the boundary of $D$, thus recovering the exact expression of the second Wiener chaos of $L_E^{(1)}$(2) obtained in Lemma 4.1 \cite{NPR17} for $d=2$. As already pointed out, in \cite{DEL19}, the authors study among others the nodal length restricted to growing cubes of the complex BRW on $\R^3$ corresponding to the case $(\ell,d)=(2,3)$. In particular, applying Green's formula to the expression of the second chaotic component (see Lemma 8, \cite{DEL19}), one can proceed similarly as above to show that it reduces to a boundary integration.

\begin{Rem}
An analogous analysis as in example (i) for ARWs on $\tor{d}$ can be carried out for the related model of spherical harmonics on the $d$-sphere, see \cite{MRW17} for the case of the $2$-sphere. 
\end{Rem}

\section{Proof of Theorem \ref{MainThm2}}\label{ARWProof}
Section \ref{SecProof} contains the proof of Theorem \ref{MainThm2}: such a proof is based on a number of technical results, whose proofs and discussion are provided in Appendix \ref{AbsCanProof}-\ref{AppSing}. The only exception to this strategy of presentation is given by Proposition \ref{Var4} and \ref{Limit4}: indeed, since these results follow from direct probabilistic arguments, their full proofs will be immediately provided in the forthcoming Section \ref{Study4}. 
 
\subsection{The proof}\label{SecProof} 
\subsubsection{An integral representation of $L_n^{(\ell)}$} 
The proof of Theorem \ref{MainThm2} is based on the Wiener chaos expansion of the quantities $L_n^{(\ell)}$ defined in \eqref{defL}. 
In order to derive this expansion, we will rigorously prove that the nodal volume $L_n^{(\ell)}$ is formally obtained $\Prob$-almost surely and in $L^2(\Prob)$ 
as
\begin{gather*}
L_n^{(\ell)} = \int_{\T} \prod_{i=1}^{\ell} \delta_0(T_n^{(i)}(x)) \cdot \Phi^*_{\ell,3}(\Jac_{\bT_n^{(\ell)}}(x)) \ dx, 
\end{gather*}
where $\Phi_{\ell,3}^*(A) =  \det(AA^T)^{1/2}$ for $A \in \mathrm{Mat}_{\ell,3}(\R)$, and $\Jac_{\bT_n^{(\ell)}}(x)$ stands for the Jacobian matrix of $\bT_n^{(\ell)}$ evaluated at $x$.
More precisely, for $\eps >0$, we consider the $\eps$-approximations $L_{n,\eps}^{(\ell)}$ of 
$L_n^{(\ell)}$ given by (compare with Definition \ref{J})
\begin{equation*}
L_{n,\eps}^{(\ell)} := \int_{\T} (2\eps)^{-\ell}\prod_{i=1}^{\ell} \ind{[-\eps,\eps]}(T_n^{(i)}(x)) \cdot
\Phi_{\ell,3}^*(\Jac_{\bT_n^{(\ell)}}(x)) \ dx \ , \quad \eps>0 \ 
\end{equation*} 
and prove the following  statement. 
\begin{Lem}\label{LemASL2}
For $\ell\in[3]$ and $n \in S_3$, the random variable $L_{n,\eps}^{(\ell)}$ converges to $L_n^{(\ell)}$ $\Prob$-a.s and in $L^2(\Prob)$ as $\eps \to 0$.
\end{Lem}
The proof of Lemma \ref{LemASL2} is presented in Section \ref{SecProofASL2} of Appendix \ref{AppSing}.
Note that the case $\ell=1$ has been investigated in \cite{RW08} for arbitrary dimensions. To deal with the case $\ell=3$, one can directly adapt the proof of points (i)-(v) of Lemma 3.1 in \cite{NPR17} for the two-dimensional torus, based on universal bounds for the number of solutions of a system of trigonometric polynomials (see e.g. \cite{K91}). 

The proof of the almost sure convergence relies on a deterministic continuity result for nodal volumes restricted to compact sets on the torus associated with sequences of functions converging to a non-degenerate limit in the $C^1$-topology (see Appendix \ref{AppCont}).
Our proof of the $L^2(\Prob)$ convergence takes advantage of 
similar techniques as those that will be exposed in the forthcoming Section \ref{SectionCHO}, based on partitioning the torus into singular and non-singular subregions. We refer the reader to this part for an overview of our strategy. 

\subsubsection{Wiener-It\^{o} chaos decomposition of $L_n^{(\ell)}$.}
The statement of Lemma \ref{LemASL2} together with the fact that,
for every fixed $x \in \T$, the random variables $\bT_n^{(\ell)}(x)$ and $\Jac_{\bT_n^{(\ell)}}(x)$ are stochastically independent, justify the use of the general framework of Theorem \ref{MainThm1} to this precise setting, yielding in particular an explicit expression for the chaotic decomposition of $L_n^{(\ell)}$. 
In view of Example (i) of Section \ref{Exs} in the case $d=3$, the quantity $D^{(i)}$ in \eqref{Di} is zero for every $i\in [\ell]$. This together with the fact that we study nodal sets, implies that (in view of Theorem \ref{MainThm1} (ii)) the second-order as well as the odd-order chaoses identically vanish, yielding 
\begin{equation}\label{WCL}
L_n^{(\ell)} = \E{L_n^{(\ell)}} + \sum_{q \geq 2} \proj_{2q}(L_n^{(\ell)}) \ , \quad \ell \in [3] \ ,
\end{equation}
where we adopted the notation \eqref{WCF}.\\
\medskip
\underline{\textit{Normalised gradients.}}
Writing $T_n^{(i_1)}(x)= \Nn^{-1/2}\sum_{\lambda\in \Lambda_n} a_{i_1,\lambda}e_{\lambda}(x)$ for $i_1 \in [\ell]$, in view of \eqref{VarT}, we introduce the scaled partial derivatives having variance $1$,
\begin{eqnarray}\label{normalization}
T_{n,j}^{(i_1)}(x) := \tilde{\partial_j} T_n^{(i_1)}(x) := \sqrt{\frac{3}{E_n}} \partial_j T_n^{(i_1)}(x) 
= i\sqrt{\frac{3}{n\Nn}} \sum_{\lambda \in \Lambda_n } \lambda_{j} a_{i_1,\lambda} e_{\lambda}(x) \ ,\quad j\in[3]
\end{eqnarray}
and adopt the same notation as in \eqref{mat}, that is
\begin{equation*}
\bT_{n\star}^{(\ell)}(x)  
:= \left\{ T_{n,j}^{(i)}(x): (i,j) \in [\ell]\times[3] \right\}\in \mathrm{Mat}_{\ell,3}(\R)
\ .
\end{equation*}
Using the homogeneity property (A2) in Definition \ref{DefPhi} of the map $\Phi_{\ell,3}^*$, it follows that
\begin{equation}\label{intnorm}
L_{n,\eps}^{(\ell)} =  \bigg(\frac{E_n}{3}\bigg)^{\ell/2}
\int_{\T} (2\eps)^{-\ell} \prod_{i=1}^{\ell}\ind{[-\eps,\eps]}(T_n^{(i)}(x)) \cdot
\Phi_{\ell,3}^*(\bT_{n\star}^{(\ell)}(x)) \ dx \ , \quad \eps>0.
\end{equation}
Therefore, by virtue of the almost sure convergence stated in Lemma \ref{LemASL2}, 
we can write the nodal volume as (recall Definition \ref{J})
\begin{equation*}
L_n^{(\ell)} =  \bigg(\frac{E_n}{3}\bigg)^{\ell/2}J(\mathbf{G},W;u^{(\ell)}),
\end{equation*}
where 
\begin{eqnarray*}\label{GWu}
\mathbf{G} = \bT_n^{(\ell)} \ , \quad 
W = \{\Phi_{\ell,3}^*(\bT_{n\star}^{(\ell)}(x)):x\in \T\} \ , \quad
u^{(\ell)} = (0,\ldots,0) \in \R^{\ell} \ .
\end{eqnarray*}
The following proposition gives the Wiener-It\^{o} chaos expansion of $L_n^{(\ell)}$ and is a direct consequence of Theorem \ref{MainThm1}. 
\begin{Prop}[Wiener Chaos expansion of $L_n^{(\ell)}$]\label{WC23}
Fix $\ell \in [3]$. For $n \in S_3$, the chaotic projections of $L_n^{(\ell)}$ are given by 
\begin{equation}\label{P2odd}
\proj_2(L_n^{(\ell)})  = \proj_{2q+1}(L_n^{(\ell)})= 0 \ , \quad q\geq 0 \ , 
\end{equation}
while for $q = 0$ and $q\geq 2$, 
\begin{eqnarray}\label{Peven}
&&\proj_{2q}(L_n^{(\ell)})  
= \bigg(\frac{E_n}{3}\bigg)^{\ell/2}
\mathop{
\sum_{p^{(1)}_0,\ldots,p^{(1)}_3\geq0}
\ldots \sum_{p^{(\ell)}_0,\ldots,p^{(\ell)}_3\geq0}}_{p^{(1)}_0+\ldots+p^{(1)}_3+\ldots+p^{(\ell)}_0+\ldots+p^{(\ell)}_3=2q}
\frac{\beta_{p_0^{(1)}}\ldots \beta_{p_0^{(\ell)}}}{p_0^{(1)}! \ldots p_0^{(\ell)}!}
\alpha^{(\ell)}_3\left\{p^{(i)}_j: (i,j) \in [\ell]\times[3] \right\} \notag\\
&& \hspace{3cm}
\times \int_{\T}   \prod_{i=1}^{\ell} H_{p_{0}^{(i)}}(T_{n}^{(i)}(x)) \prod_{j=0}^3  H_{p_{j}^{(i)}}(T_{n,j}^{(i)}(x)) \ dx  \ ,\notag
\end{eqnarray} 
where $\{\beta_j:j\geq0\}$ and $\alpha^{(\ell)}_3\{\cdot\}$ are the Wiener chaos projection coefficients of $\delta_0$ and $\Phi_{\ell,3}^*$, that is 
\begin{equation*}
\beta_{2j+1} = 0 \ , \quad \beta_{2j} = \frac{H_{2j}(0)}{\sqrt{2\pi}} \ , \quad j \geq 0 \ , 
\end{equation*}
and 
\begin{equation*}\label{alpha}
\alpha^{(\ell)}_k\left\{p^{(i)}_j: (i,j) \in [\ell]\times [k] \right\} := 
\frac{1}{\prod_{i=1}^{\ell}\prod_{j=1}^{k}(p^{(i)}_j)! } 
\cdot \E{\Phi_{\ell,k}^*(\bX) \cdot \prod_{i=1}^{\ell} \prod_{j=1}^k  H_{p^{(i)}_j}(X^{(i)}_j)} \ , \quad k\geq \ell
\end{equation*}
respectively.
In particular, 
\begin{equation}\label{P0}
\proj_0(L_n^{(\ell)})  = \E{L_n^{(\ell)}} = \bigg(\frac{E_n}{3}\bigg)^{\ell/2}\frac{\alpha(\ell,3)}{(2\pi)^{\ell/2}} \ ,
\end{equation}
where 
\begin{equation*}  
\alpha(\ell,k)= \frac{(k)_{\ell}\kappa_k}{(2\pi)^{\ell/2}\kappa_{k-\ell}}   \ , 
\end{equation*}
is as in \eqref{defalk}.
\end{Prop}

\subsubsection{Analysis of the fourth chaotic projection}
Our main findings on the high-energy behaviour of the fourth-order chaotic projections $\proj_4(L_n^{(\ell)}), \ell \in [3]$ are contained in the next two propositions, whose proofs are presented in Section \ref{VarDist}:
\begin{Prop}\label{Var4} 
For $\ell\in [3]$, as $n \to \infty, \notcon{n}{0,4,7}{8}$,  
\begin{equation*}
\V{\proj_4(L_n^{(\ell)})} \sim \big(c_n^{(\ell)}\big)^2 \bigg(
\ell \cdot \frac{1}{250}+ \frac{\ell(\ell-1)}{2} \cdot \frac{76}{375}\bigg) \ , 
\end{equation*}
where the constant $c_n^{(\ell)}$ is given by
\begin{gather*}
c_n^{(\ell)} := \bigg(\frac{E_n}{3}\bigg)^{\ell/2} \frac{2}{ (2\pi)^{\ell/2}}\frac{\alpha(\ell,3)}{\Nn}  \ .
\end{gather*}
\end{Prop}
\begin{Prop}\label{Limit4}
For $\ell \in [3]$, we define the normalized  fourth-order chaotic component
\begin{equation*}
\left\{\widetilde{\proj_4(L_n^{(\ell)})} : n\in S_3\right\}
:=\left\{ \big(v^{(\ell)}_{n;4}\big)^{-1/2} \proj_4(L_n^{(\ell)}): n\in S_3\right\},
\end{equation*}
where $v^{(\ell)}_{n;4}:=\V{\proj_4(L_n^{(\ell)})}$.
As $n \to \infty, \notcon{n}{0,4,7}{8}$,
\begin{equation*}
\widetilde{\proj_4(L_n^{(\ell)})} \Law 
\bigg(\ell \cdot \frac{1}{250} + \frac{\ell(\ell-1)}{2} \cdot \frac{76}{375}\bigg)^{-1/2} 
Y^{(\ell)} M^{(\ell)} (Y^{(\ell)})^T \ ,
\end{equation*}
where $Y^{(\ell)} \sim \mathcal{N}_{\ell(9\ell-4)}(0,\Id_{\ell(9\ell-4)})$ and $M^{(\ell)} \in \mathrm{Mat}_{\ell(9\ell-4),\ell(9\ell-4)}(\R)$ is the deterministic matrix given by 
\begin{equation*}
M^{(\ell)} = \frac{-1}{50} \Id_{5\ell} 
\oplus \frac{-1}{25}\Id_{\frac{5\ell(\ell-1)}{2}}
\oplus \frac{1}{25}\Id_{\frac{5\ell(\ell-1)}{2}}
\oplus \frac{1}{50}\Id_{\frac{5\ell(\ell-1)}{2}}
\oplus \frac{-1}{6} \Id_{\frac{3\ell(\ell-1)}{2}} 
\ . 
\end{equation*}
\end{Prop}
Such results are proved as follows:
In Section \ref{subExplicit}, we provide an exact expression of the fourth-order chaotic projection of $L_n^{(\ell)}$. In order to achieve this, we compute the Fourier-Hermite coefficients of the function $\Phi_{\ell,3}^*$ on the fourth Wiener chaos (see Proposition \ref{ChaosPhi4}). We then use the orthogonality relation for complex exponentials on the torus
\begin{eqnarray}\label{ort}
\int_{\T} e_{\lambda}(x) \ dx = \ind{\lambda=0} \ ,
\end{eqnarray}
and rewrite each integral of multivariate Hermite polynomials evaluated at the arithmetic random waves and its gradient components by means of a useful summation rule over $4$-correlations $\mathcal{C}_n(4)$ and non-degenerate $4$-correlations $\mathcal{X}_n(4)$ (see  \eqref{Cm} and \eqref{Xm} for precise definitions). 

A subsequent asymptotic analysis of $\proj_4(L_n^{(\ell)})$ is presented in Section \ref{subAsy}. 
This analysis is based on a multivariate Central Limit Theorem (see Proposition \ref{PropJointLaw3}) for the summands composing the explicit expression of $\proj_4(L_n^{(\ell)})$. 
Such a Central Limit Theorem, already appearing in \cite{MPRW16,DNPR16} for the two-dimensional torus and \cite{Cam17} for the nodal surface on the three-dimensional torus, is obtained by verifying a suitable condition characterising normal convergence of the so-called Fourth Moment Theorem (see Theorem 5.2.7 \cite{NP12}). 
Among others, we use the following asymptotic estimate bounding non-degenerate $4$-correlations on $\T$ (see Theorem 1.6 \cite{BM17}):
\begin{eqnarray}\label{EstX4}
\mathrm{card}(\mathcal{X}_n(4)) =  
O(\Nn^{7/4+o(1)}) \ , \quad n \to \infty \ .
\end{eqnarray}

\subsubsection{Contribution of higher-order chaotic projections}\label{SectionCHO}
We show that the projection on the fourth Wiener chaos of $L_n^{(\ell)}$ dominates the series in \eqref{WCL}, in the sense that
\begin{eqnarray*}
\widetilde{L_n^{(\ell)}} = \widetilde{\proj_4(L_n^{(\ell)})}+o_{\Prob}(1) \ ,
\end{eqnarray*}
where $o_{\Prob}(1)$ denotes a sequence of random variables converging to zero in probability as $n\to \infty,\notcon{n}{0,4,7}{8}$. This is done by proving the following statement (see Appendix \ref{AppSing}): 
\begin{Prop}\label{VarHigherOrder} 
For $\ell \in [3]$, as $n \to \infty, \notcon{n}{0,4,7}{8}$,  
\begin{equation}\label{oVar4}
\V{\sum_{q\geq 3} \proj_{2q}(L_n^{(\ell)})} = o\bigg(\V{\proj_4(L_n^{(\ell)})}\bigg) \ .
\end{equation}
\end{Prop}
The arguments for the proof of Proposition \ref{VarHigherOrder} are based on the use of a suitable partition $\mathcal{P}(M)$ (where $M=M(n)$ is proportional to 
$\sqrt{E_n}$) of the torus into singular and non-singular pairs of subregions (see Definition \ref{DefSing}), following the route of \cite{ORW08} and, later, \cite{PR16,DNPR16}. We denote by $L_n^{(\ell)}(Q)$ the nodal volume restricted to a cube $Q$ and by $\proj_{6+}(L_n^{(\ell)}):= \sum_{q \geq 3} \proj_{2q}(L_n^{(\ell)})$ the chaotic projection of $L_n^{(\ell)}$ on Wiener chaoses of order at least $6$. This allows us to write the variance of higher-order chaoses as 
\begin{equation}\label{Varsum}
\V{\proj_{6+}(L_n^{(\ell)})}    
= \sum_{(Q,Q') \in \mathcal{P}(M)^2} 
\Cov{ 
\proj_{6+}(L_n^{(\ell)}(Q))} 
{   \proj_{6+}(L_n^{(\ell)}(Q'))}     \ ,
\end{equation}
where the summation is over all pairs of cubes $(Q,Q')$ of side length $1/M$. Splitting this sum into the singular part $\mathcal{S}$ and the non-singular part $\mathcal{S}^c$, we  bound each of the contributions separately. For the singular part, we prove the following bound (see Section \ref{ProofsLemS} of Appendix \ref{AppSing}):
\begin{Lem}\label{Singl}
For $\ell\in[3]$, as $n \to \infty,  \notcon{n}{0,4,7}{8}$, we have 
\begin{equation*}
\left|S_{n,1}^{(\ell)}\right|:= \left| \sum_{(Q,Q') \in \mathcal{S}}  \Cov{\mathrm{proj}_{6+}(L_n^{(\ell)}(Q))}{\mathrm{proj}_{6+}(L_n^{(\ell)}(Q'))} \right| = O(E_n^{\ell} \mathcal{R}_n(6)) \ .
\end{equation*}
\end{Lem} 
Here, $\mathcal{R}_n(6)$ denotes the integral $6$-th moment of the covariance function $r_n$, see formula \eqref{Rn6} below.
We give a brief overview of the proof of Lemma \ref{Singl}. We use the Cauchy-Schwarz inequality and translation-invariance of the model to write 
\begin{eqnarray}\label{Est1a}
\left|S_{n,1}^{(\ell)}\right|\leq E_n^3\mathcal{R}_n(6) \cdot\V{\mathrm{proj}_{6+}(L_n^{(\ell)}(Q_0))}   , 
\end{eqnarray}
where we used that the number of singular pairs of cubes in the summation index is bounded by $E_n^3\mathcal{R}_n(6)$ and where $Q_0$ denotes a small cube of side length $1/M$ around the origin.
In Lemma \ref{LemKR}, we justify the use of Kac-Rice formula in $Q_0$, so that, writing 
\begin{gather*}
\V{\proj_{6+}(L_n^{(\ell)}(Q_0))}   \leq \E{L_n^{(\ell)}(Q_0)^2},
\end{gather*}
one can use Kac-Rice formulae for moments (Theorem 6.2, 6.3 \cite{AW09} for $\ell=3$ and Theorem 6.9 \cite{AW09} for $\ell=1,2$). Doing so, we exploit stationarity to obtain
\begin{eqnarray}\label{intQ0}
\E{L_n^{(\ell)}(Q_0)^2} = \int_{Q_0\times Q_0} 
K^{(\ell)}(x,y;(0,\ldots,0)) \ dxdy 
+ \E{L_n^{(3)}(Q_0)}\ind{\ell=3} \notag \\
\leq \Leb(Q_0) \int_{2Q_0} K^{(\ell)}(z,0;(0,\ldots,0)) \ dz 
+ \frac{E_n^{3/2}}{M^3}\ind{\ell=3}\ ,
\end{eqnarray}
where $K^{(\ell)}$ is the two-point correlation function defined in \eqref{K} of Appendix C. 
Appendix C contains a self-contained study of the two-point correlation function; in particular, in \eqref{Kupperbound}, we derive an upper bound of $K^{(\ell)}$ in terms of the covariance function $r_n$ and its gradient, and subsequently perform a precise Taylor-type expansion near the origin of this expression (see Lemma \ref{Taylorq}). Using these results then yields the estimate
\begin{gather*}
\E{L_n^{(\ell)}(Q_0)^2}  \ll E_n^{-2}\ind{\ell=1}+E_n^{-1} \ind{\ell=2} + \ind{\ell=3},
\end{gather*}
which combined with \eqref{Est1a} establishes Lemma \ref{Singl}.

Concerning the contribution to the variance of the non-singular pairs of cubes, we prove the following proposition (see Section \ref{ProofsLemS} of Appendix \ref{AppSing}):
\begin{Lem}\label{NonSingl}
For  $\ell\in [3]$, as $n \to \infty,  \notcon{n}{0,4,7}{8}$, we have 
\begin{equation*}
\left|S_{n,2}^{(\ell)}\right| := \left| \sum_{(Q,Q') \in \mathcal{S}^c} 
\Cov{\mathrm{proj}_{6+}(L_n^{(\ell)}(Q))}
{\mathrm{proj}_{6+}(L_n^{(\ell)}(Q'))} \right| = O(E_n^{\ell}\mathcal{R}_n(6)) \ .
\end{equation*}
\end{Lem}
In order to prove Lemma \ref{NonSingl}, we take advantage of (i) the Wiener-It\^{o} chaos expansion of $L_n^{(\ell)}$ and (ii) a particular version of diagram formula for Hermite polynomials (see Proposition \ref{Diagram}) allowing us to handle covariances of products of Hermite polynomials. The desired bound is then obtained by exploiting the fact that the summation is over non-singular pairs of cubes.

Combining the decomposition of the variance in \eqref{Varsum} with Lemma \ref{Singl} and Lemma
\ref{NonSingl}, the proof of Proposition \ref{VarHigherOrder} is then concluded once we derive a bound for the integral $6$-th moment of $r_n$. 
In order to achieve this, we can again use the orthogonality relation for complex exponentials on the torus \eqref{ort}
in order to link moments of the covariance function $r_n$ to $m$-correlations, for $m\geq 1$, 
\begin{gather*}
\mathcal{R}_n(m):= \int_{\T} r_n(z)^m dz =
\frac{1}{\Nn^m}\sum_{(\lambda^{(1)},\ldots,\lambda^{(m)}) \in \Lambda_n^m}\int_{\T}
e_{\lambda^{(1)}+\ldots+\lambda^{(m)}}(z) dz 
= \frac{\mathrm{card}(\mathcal{C}_n(m))}{\Nn^m}  \ . 
\end{gather*}
Using this formula for $m=6$ together with the estimate bounding the number of $6$-correlations on $\T$
(Theorem 1.7 \cite{BM17}) 
\begin{equation*} 
\mathrm{card}(\mathcal{C}_n(6)) = O(\Nn^{11/3+o(1)}) \ , \qquad n \to \infty \ , 
\end{equation*}
yields 
\begin{equation}\label{Rn6}
\mathcal{R}_n(6)= \int_{\T} r_n(z)^6 \ dz = \frac{\mathrm{card}(\mathcal{C}_n(6))}{\Nn^6} 
= O(\Nn^{-7/3+o(1)}) \ , \qquad n \to \infty \ .
\end{equation}
Combining this with the content of Proposition \ref{Var4}, we conclude that $E_n^{\ell}\mathcal{R}_n(6)=o\left(\V{\proj_4(L_n^{(\ell)})}\right)$.

\subsubsection{Finishing the proof of Theorem \ref{MainThm2}}
The proof of Theorem \ref{MainThm2} is concluded as follows: Relation \eqref{meanL} follows from \eqref{P0}
and the distributional identity stated in formula \eqref{mean}.
The asymptotic variance in Proposition \ref{Var4} together with Proposition \ref{VarHigherOrder} prove \eqref{AsyVar}. Finally, \eqref{AsyDist} follows from the limiting distribution established in Proposition \ref{Limit4} combined with Proposition \ref{VarHigherOrder}. 

\subsection{Complete study of the fourth chaotic component of $L_n^{(\ell)}$}
\label{Study4}
In this section, we provide the exact expression of the fourth-order chaotic component of $L_n^{(\ell)}$. A subsequent asymptotic analysis of this expression serves as preparation to deriving the limiting distribution of the normalised version of $L_n^{(\ell)}$.
\subsubsection{Explicit form of $\proj_4(L_n^{(\ell)})$}\label{subExplicit}
In order to write the explicit expression of the fourth-order chaotic component of $L_n^{(\ell)}$, we introduce some auxiliary random variables. Fix $\ell \in [3]$. 
\begin{Def}\label{RV1}
For $i_1,i_2 \in [\ell], j,k \in [3]$ and $n\in S_3$, we define:
\begin{eqnarray*}
W^{(i_1)}(n) &:=& \frac{1}{\sqrt{\Nn}} \sum_{\lambda \in \Lambda_n} (|a_{i_1,\lambda}|^2-1) \ , \qquad
W^{(i_1)}_{jk}(n) := \frac{1}{n\sqrt{\Nn}} \sum_{\lambda \in \Lambda_n} 
\lambda_j \lambda_k (|a_{i_1,\lambda}|^2-1) \ , \\
M^{(i_1,i_2)}(n) &:=& \frac{1}{\sqrt{\Nn}} \sum_{\lambda \in \Lambda_n} a_{i_1,\lambda} \overline{a_{i_2,\lambda}} \ , \qquad i_1<i_2, \ \ell \in \{2,3\} \ , \\
M^{(i_1,i_2)}_{j}(n) &:=& \frac{i}{\sqrt{n\Nn}} \sum_{\lambda \in \Lambda_n} \lambda_j a_{i_1,\lambda} \overline{a_{i_2,\lambda}} \ , \qquad i_1<i_2, \ \ell \in \{2,3\}  \ , \\
M^{(i_1,i_2)}_{jk}(n) &:=& \frac{1}{n\sqrt{\Nn}} \sum_{\lambda \in \Lambda_n} \lambda_j \lambda_k a_{i_1,\lambda} \overline{a_{i_2,\lambda}} \  , \qquad i_1<i_2, \ \ell \in \{2,3\} \ , \\
R^{(i_1,i_2)}(n) &:=& \frac{1}{\Nn} \sum_{\lambda \in \Lambda_n}
|a_{i_1,\lambda}|^2|a_{i_2,\lambda}|^2 \ , \qquad
R^{(i_1,i_2)}_{jk}(n) := \frac{1}{n^2\Nn} \sum_{\lambda \in \Lambda_n} \lambda_j^2 \lambda_k^2
|a_{i_1,\lambda}|^2|a_{i_2,\lambda}|^2 \ , \\
S^{(i_1,i_2)}(n) &:=& \frac{1}{\Nn} \sum_{\lambda \in \Lambda_n}
a_{i_1,\lambda}^2 \overline{a_{i_2,\lambda}}^2 \ , \qquad
S^{(i_1,i_2)}_{jk}(n) := \frac{1}{n^2\Nn} \sum_{\lambda \in \Lambda_n}
\lambda_j^2 \lambda_k^2 a_{i_1,\lambda}^2 \overline{a_{i_2,\lambda}}^2 \ , \\
X^{(i_1,i_2)}(n) &:=& \frac{1}{\Nn} \sum_{(\lambda,\lambda',\lambda'',\lambda''') \in \mathcal{X}_n(4)} a_{i_1,\lambda}a_{i_1,\lambda'}a_{i_2,\lambda''}a_{i_2,\lambda'''} \ , \\
X^{(i_1,i_2)}_{kk}(n) &:=& \frac{1}{n\Nn} \sum_{(\lambda,\lambda',\lambda'',\lambda''') \in \mathcal{X}_n(4)} \lambda_k\lambda'_k a_{i_1,\lambda}a_{i_1,\lambda'}a_{i_2,\lambda''}a_{i_2,\lambda'''} \ , \\
X^{(i_1,i_2)}_{kkjj}(n) &:=& \frac{1}{n^2\Nn} \sum_{(\lambda,\lambda',\lambda'',\lambda''') \in \mathcal{X}_n(4)} 
\lambda_k\lambda'_k  \lambda''_j\lambda'''_j  
a_{i_1,\lambda}a_{i_1,\lambda'}a_{i_2,\lambda''}a_{i_2,\lambda'''} \ . 
\end{eqnarray*}
\end{Def}
Note that $\lambda_1^2+\lambda_2^2 + \lambda_3^2=n$ implies the relations 
\begin{equation*}
R^{(i_1,i_2)}(n) = \sum_{k,j=1}^3 R_{jk}^{(i_1,i_2)}(n) \ , \quad 
S^{(i_1,i_2)}(n) = \sum_{k,j=1}^3 S_{jk}^{(i_1,i_2)}(n) \ .
\end{equation*} 

\begin{Def}\label{RV2}
For $i_1 \in [\ell]$, and $n\in S_3$, we set 
\begin{eqnarray*}
a_1^{(i_1)}(n):= \int_{\T} H_4(T_n^{(i_1)}(x)) \ dx \ ,    && a_2^{(i_1)}(n) := \sum_{k=1}^3 \int_{\T}{H_2(T_n^{(i_1)}(x)) H_2(T_{n,k}^{(i_1)}(x)) \ dx} \ , \\
a_3^{(i_1)}(n):= \sum_{k=1}^3 \int_{\T}{H_4(T_{n,k}^{(i_1)}(x)) \ dx} \ ,   &&
a_4^{(i_1)}(n):= \sum_{k < j} \int_{\T}{H_2(T_{n,k}^{(i_1)}(x)) 
H_2(T_{n,j}^{(i_1)}(x))\ dx} \ , 
\end{eqnarray*}
and for $\ell \in \{2,3\}$ and $i_1<i_2 \in [\ell], n\in S_3$,
\begin{eqnarray*}
b_1^{(i_1,i_2)}(n)&:=& \int_{\T}{H_2(T_n^{(i_1)}(x)) H_2(T_n^{(i_2)}(x))\ dx} \ , \\
b_2^{(i_1,i_2)}(n)&:=& \sum_{k=1}^3 \int_{\T}{H_2(T_n^{(i_1)}(x)) 
H_2(T_{n,k}^{(i_2)}(x)) \ dx} \ , \\
{b_2'}^{(i_1,i_2)}(n)&:=& \sum_{k=1}^3 \int_{\T}{H_2(T_{n,k}^{(i_1)}(x)) 
H_2(T_n^{(i_2)}(x)) \ dx} \ , \\ 
b_3^{(i_1,i_2)}(n)&:=& \sum_{k\neq j=1}^3 \int_{\T}{H_2(T_{n,k}^{(i_1)}(x)) 
H_2(T_{n,j}^{(i_2)}(x)) \ dx} \ , \\
b_4^{(i_1,i_2)}(n)&:=& \sum_{k=1}^3 \int_{\T}{H_2(T_{n,k}^{(i_1)}(x)) 
H_2(T_{n,k}^{(i_2)}(x)) \ dx} \ , \\
b_5^{(i_1,i_2)}(n)&:=& \sum_{k< j}\int_{\T}{T_{n,k}^{(i_1)}(x) T_{n,j}^{(i_1)}(x)
T_{n,k}^{(i_2)}(x) T_{n,j}^{(i_2)}(x)\ dx} \ .
\end{eqnarray*}
\end{Def}

\noindent\underline{\textit{Spectral correlations on $\T$.}}
For $n \in S_3$ and an integer $m \geq 1$, we introduce the set of $m$-correlations on the torus, 
\begin{equation}\label{Cm}
\mathcal{C}_n(m) := \left\{ (\lambda^{(1)},\ldots,\lambda^{(m)}) \in \Lambda_n^m : \lambda^{(1)}+\ldots+\lambda^{(m)}=0\right\} 
\end{equation} 
and the set of non-degenerate $m$-correlations 
\begin{equation}\label{Xm}
\mathcal{X}_n(m) := \left\{ (\lambda^{(1)},\ldots,\lambda^{(m)}) \in \mathcal{C}_n(m) : \forall I \subsetneq [m], \sum_{i \in I} \lambda^{(i)} \neq 0\right\} \subsetneq \mathcal{C}_n(m) \ .
\end{equation}
Recall that $\mathrm{card}(\mathcal{C}_n(4)) = 3 \Nn^2 - 3 \Nn + \mathrm{card}(\mathcal{X}_n(4))$, which is in accordance with the following summation rule (see (3.6) in \cite{Cam17})
\begin{eqnarray}\label{sumC}
&&\sum_{(\lambda,\lambda',\lambda'',\lambda''') \in \mathcal{C}_n(4) } =
\sum_{\substack{\lambda=-\lambda' \\ \lambda''=-\lambda'''}}
+ \sum_{\substack{\lambda=-\lambda'' \\ \lambda'=-\lambda'''}}
+ \sum_{\substack{\lambda=-\lambda''' \\ \lambda'=-\lambda''}} \notag\\ 
&& - \sum_{\lambda=-\lambda'=\lambda''=-\lambda'''} 
- \sum_{\lambda=\lambda'=-\lambda''=-\lambda'''} 
- \sum_{\lambda=-\lambda'=-\lambda''=\lambda'''} 
+ \sum_{(\lambda,\lambda',\lambda'',\lambda''') \in \mathcal{X}_n(4)} \ .
\end{eqnarray}
In the sequel, we will write $(\lambda,\lambda',\lambda'',\lambda''') = (\lambda^{(1)},\lambda^{(2)},\lambda^{(3)},\lambda^{(4)})$ for elements in $\mathcal{C}_n(4)$ and $\mathcal{X}_n(4)$ and use the following abbreviations 
\begin{gather*}
\sum_{\lambda}:= \sum_{\lambda \in \Lambda_n}, \quad
\sum_{\mathcal{C}_{n}(4)} := \sum_{(\lambda,\lambda',\lambda'',\lambda''')\in\mathcal{C}_{n}(4)}, \quad
\sum_{\mathcal{X}_{n}(4)} := \sum_{(\lambda,\lambda',\lambda'',\lambda''')\in\mathcal{X}_{n}(4)}.
\end{gather*}

The following lemma is a generalization of Lemma 4.5 in \cite{Cam17} (obtained for $\ell=1$) applying to the setting of multiple independent arithmetic random waves. These formulae follow by carefully applying the summation rule \eqref{sumC}.
\begin{Lem}\label{LemSums}
Fix $\ell \in [3]$. For every $i_1,i_2 \in [\ell]$ and every $j,k \in [3]$, the following formulae hold:
\begin{eqnarray}\label{ForA}
\sum_{\mathcal{C}_n(4)} a_{i_1,\lambda}a_{i_1,\lambda'}a_{i_2,\lambda''}a_{i_2,\lambda'''} =
\sum_{\lambda} |a_{i_1,\lambda}|^2 \sum_{\lambda}|a_{i_2,\lambda}|^2 + 
2 \bigg(\sum_{\lambda}a_{i_1,\lambda} \overline{a_{i_2,\lambda}} \bigg)^2 \notag \\
- 2 \sum_{\lambda} |a_{i_1,\lambda}|^2 |a_{i_2,\lambda}|^2  
- \sum_{\lambda} a_{i_1,\lambda}^2 \overline{a_{i_2,\lambda}}^2 + \sum_{\mathcal{X}_n(4)}a_{i_1,\lambda}a_{i_1,\lambda'}a_{i_2,\lambda''}a_{i_2,\lambda'''} \ , 
\end{eqnarray}
\begin{eqnarray}\label{ForB}
\sum_{\mathcal{C}_n(4)} \lambda''_k \lambda'''_k a_{i_1,\lambda}a_{i_1,\lambda'}a_{i_2,\lambda''}a_{i_2,\lambda'''} 
= -\sum_{\lambda}  |a_{i_1,\lambda}|^2 \sum_{\lambda} \lambda^2_k 
|a_{i_2,\lambda}|^2 
+ 2 \bigg(\sum_{\lambda} \lambda_k  a_{i_1,\lambda}\overline{a_{i_2,\lambda}}\bigg)^2 \notag \\
+ 2 \sum_{\lambda} \lambda^2_k |a_{i_1,\lambda}|^2 |a_{i_2,\lambda}|^2  
- \sum_{\lambda} \lambda^2_k a_{i_1,\lambda}^2 \overline{a_{i_2,\lambda}}^2  
+ \sum_{\mathcal{X}_n(4)} \lambda''_k \lambda'''_k a_{i_1,\lambda}a_{i_1,\lambda'}a_{i_2,\lambda''}a_{i_2,\lambda'''} \ , 
\end{eqnarray}
\begin{eqnarray}\label{ForC}
\sum_{\mathcal{C}_n(4)} \lambda_k \lambda'_k\lambda''_j \lambda'''_j a_{i_1,\lambda}a_{i_1,\lambda'}a_{i_2,\lambda''}a_{i_2,\lambda'''} 
= \sum_{\lambda} \lambda^2_k |a_{i_1,\lambda}|^2 
\sum_{\lambda} \lambda^2_j |a_{i_2,\lambda}|^2 
+ 2 \bigg(\sum_{\lambda} \lambda_k \lambda_j  a_{i_1,\lambda}\overline{a_{i_2,\lambda}} \bigg)^2  \notag\\
- 2 \sum_{\lambda} \lambda^2_k \lambda_j^2 |a_{i_1,\lambda}|^2 |a_{i_2,\lambda}|^2  
- \sum_{\lambda} \lambda^2_k\lambda^2_j a_{i_1,\lambda}^2 \overline{a_{i_2,\lambda}}^2  
+ \sum_{\mathcal{X}_n(4)} 
\lambda_k \lambda'_k\lambda''_j \lambda'''_j
a_{i_1,\lambda}a_{i_1,\lambda'}a_{i_2,\lambda''}a_{i_2,\lambda'''} \ ,
\end{eqnarray}
\begin{eqnarray}\label{ForD}
\sum_{\mathcal{C}_n(4)} \lambda_k \lambda'_j\lambda''_k \lambda'''_j a_{i_1,\lambda}a_{i_1,\lambda'}a_{i_2,\lambda''}a_{i_2,\lambda'''} 
= \sum_{\lambda} \lambda_k \lambda_j  |a_{i_1,\lambda}|^2
\sum_{\lambda} \lambda_k \lambda_j  |a_{i_2,\lambda}|^2\notag\\
+  \sum_{\lambda} \lambda^2_k   a_{i_1,\lambda}\overline{a_{i_2,\lambda}} 
\sum_{\lambda} \lambda^2_j   a_{i_1,\lambda}\overline{a_{i_2,\lambda}}
+ \bigg(\sum_{\lambda} \lambda_k \lambda_j   a_{i_1,\lambda} \overline{a_{i_2,\lambda}}  \bigg)^2 \notag \\
- 2 \sum_{\lambda} \lambda^2_k \lambda^2_j |a_{i_1,\lambda}|^2 |a_{i_2,\lambda}|^2  
- \sum_{\lambda} \lambda^2_k \lambda^2_j a_{i_1,\lambda}^2 \overline{a_{i_2,\lambda}}^2 \notag\\ 
+ \sum_{\mathcal{X}_n(4)} 
\lambda_k \lambda'_j\lambda''_k \lambda'''_j
a_{i_1,\lambda}a_{i_1,\lambda'}a_{i_2,\lambda''}a_{i_2,\lambda'''} \ .
\end{eqnarray}
\end{Lem}

The next two lemmas express the random variables introduced in Definition \ref{RV2} in terms of the quantities defined in Definition \ref{RV1}. 
The following expansions have been proved in Lemma 4.4 of \cite{Cam17}.  
\begin{Lem}\label{ExpA}
Fix $\ell \in [3]$. For every $i_1  \in [\ell]$, we have 
\begin{enumerate}[label=\rm{(\roman*)}]
\item $a_1^{(i_1)}(n) = \frac{3}{\Nn}\big(W^{(i_1)}(n)^2-R^{(i_1,i_1)}(n) + \frac{1}{3}X^{(i_1,i_1)}(n)  \big) $
\item $a_2^{(i_1)}(n) =  \frac{3}{\Nn}\big(W^{(i_1)}(n)^2 - R^{(i_1,i_1)}(n) - \sum_{k=1}^3 X_{kk}^{(i_1,i_1)}(n)\big)$
\item $a_3^{(i_1)}(n) = \frac{27}{\Nn}\sum_{k=1}^3 \big(  W^{(i_1)}_{kk}(n)^2 -  R_{kk}^{(i_1,i_1)}(n) +   \frac{1}{3}X_{kkkk}^{(i_1,i_1)}(n)\big)$
\item $a_4^{(i_1)}(n) = \frac{9}{\Nn}\sum_{k<j}\big( W^{(i_1)}_{kk}(n)W^{(i_1)}_{jj}(n) + 2  W^{(i_1)}_{kj}(n)^2 - 3R_{kj}^{(i_1,i_1)}(n) + X_{kkjj}^{(i_1,i_1)}(n)\big)$ 
\end{enumerate}
\end{Lem}
The next lemma deals with mixed expressions containing indices $i_1<i_2$. 
\begin{Lem}\label{ExpB} 
Fix $\ell \in \{2,3\}$. For every $i_1 < i_2 \in [\ell]$, we have 
\begin{enumerate}[label=\rm{(\roman*)}]
\item $b_1^{(i_1,i_2)}(n) = \frac{1}{\Nn}\big( W^{(i_1)}(n) W^{(i_2)}(n) 
+ 2 M^{(i_1,i_2)}(n)^2 - 2R^{(i_1,i_2)}(n)-S^{(i_1,i_2)}(n) + X^{(i_1,i_2)}(n)\big)$ 
\item $b_2^{(i_1,i_2)}(n)  = {b_2'}^{(i_2,i_1)}(n)  
=  \frac{3}{\Nn} \big(W^{(i_1)}(n) W^{(i_2)}(n) +
2\sum_{k=1}^3 M_k^{(i_1,i_2)}(n)^2 
- 2R^{(i_1,i_2)}(n) + S^{(i_1,i_2)}(n)
- \sum_{k=1}^3 X^{(i_1,i_2)}_{kk}(n) \big)$
\item $b_3^{(i_1,i_2)}(n) =\frac{9}{\Nn}\sum_{k \neq j=1}^3
\big(W^{(i_1)}_{kk}(n) W^{(i_2)}_{jj}(n) 
+ 2 M^{(i_1,i_2)}_{kj}(n)^2 
-2 R^{(i_1,i_2)}_{kj}(n) 
- S^{(i_1,i_2)}_{kj}(n)
+  X^{(i_1,i_2)}_{kkjj}(n)\big) $
\item $b_4^{(i_1,i_2)}(n) =\frac{9}{\Nn}\sum_{k=1}^3
\big(W^{(i_1)}_{kk}(n) W^{(i_2)}_{kk}(n) 
+ 2 M^{(i_1,i_2)}_{kk}(n)^2 
-2 R^{(i_1,i_2)}_{kk}(n) 
- S^{(i_1,i_2)}_{kk}(n)
+  X^{(i_1,i_2)}_{kkkk}(n) \big) $
\item $b_5^{(i_1,i_2)}(n) =\frac{9}{\Nn}\sum_{k< j}\big(W^{(i_1)}_{kj}(n) W^{(i_2)}_{kj}(n) 
+M^{(i_1,i_2)}_{kk}(n) M^{(i_1,i_2)}_{jj}(n) 
+M^{(i_1,i_2)}_{kj}(n)^2 
-2 R^{(i_1,i_2)}_{kj}(n) -S^{(i_1,i_2)}_{kj}(n)
+X^{(i_1,i_2)}_{kkjj}(n)\big)$
\end{enumerate}
\end{Lem}

\begin{proof}  
Let $\ell \in \{2,3\}$ be fixed. For (i), by \eqref{ForA}, we have
\begin{eqnarray*}
&& b_1^{(i_1,i_2)}(n) = \int_{\T}{H_2(T_n^{(i_1)}(x)) H_2(T_n^{(i_2)}(x))\ dx} \\
&=& \int_{\T}{\big( T_n^{(i_1)}(x)^2 T_n^{(i_2)}(x)^2 - T_n^{(i_1)}(x)^2 - T_n^{(i_2)}(x)^2 + 1\big) \ dx} \\
&=& \frac{1}{\Nn^2} \sum_{ \mathcal{C}_n(4)} a_{i_1,\la} a_{i_1,\la'} a_{i_2, \la''} a_{i_2,\la'''}
- \frac{1}{\Nn}\sum_{\la}|a_{i_1,\la}|^2
-\frac{1}{\Nn}\sum_{\la}|a_{i_2,\la}|^2 +1\\
&=& \frac{1}{\Nn^2}\sum_{\lambda} |a_{i_1,\lambda}|^2 \sum_{\lambda}|a_{i_2,\lambda}|^2 + 
\frac{2}{\Nn^2} \bigg(\sum_{\lambda}a_{i_1,\lambda} \overline{a_{i_2,\lambda}} \bigg)^2
- \frac{2}{\Nn^2} \sum_{\lambda} |a_{i_1,\lambda}|^2 |a_{i_2,\lambda}|^2 \\
&& \qquad - \frac{1}{\Nn^2} \sum_{\lambda} a_{i_1,\lambda}^2 \overline{a_{i_2,\lambda}}^2 
+ \frac{1}{\Nn^2}\sum_{\mathcal{X}_n(4)}a_{i_1,\lambda}a_{i_1,\lambda'}a_{i_2,\lambda''}a_{i_2,\lambda'''} \\ 
&&\qquad- \frac{1}{\Nn}\sum_{\la}|a_{i_1,\la}|^2  -\frac{1}{\Nn}\sum_{\la}|a_{i_2,\la}|^2 +1 \ .
\end{eqnarray*}
Now using the relation
\begin{eqnarray}\label{rel}
&&\frac{1}{\Nn^2}\sum_{\la}(|a_{i_1,\la}|^2-1)
\sum_{\la}(|a_{i_2,\la}|^2-1) \notag \\
&=&  \frac{1}{\Nn^2} \sum_{\la}|a_{i_1,\la}|^2\sum_{\la}|a_{i_2,\la'}|^2 - \frac{1}{\Nn}\sum_{\la}|a_{i_1,\la}|^2
-\frac{1}{\Nn}\sum_{\la}|a_{i_2,\la}|^2 +1  \ ,
\end{eqnarray}
we can rewrite $b_1^{(i_1,i_2)}(n)$ as
\begin{eqnarray*}
\frac{1}{\Nn^2}\sum_{\la}(|a_{i_1,\la}|^2-1)\sum_{\la}(|a_{i_2,\la}|^2-1) +
\frac{2}{\Nn^2} \bigg(\sum_{\lambda}a_{i_1,\lambda} \overline{a_{i_2,\lambda}} \bigg)^2 \\
\qquad - \frac{2}{\Nn^2} \sum_{\lambda} |a_{i_1,\lambda}|^2 |a_{i_2,\lambda}|^2 
- \frac{1}{\Nn^2} \sum_{\lambda} a_{i_1,\lambda}^2 \overline{a_{i_2,\lambda}}^2 
+ \frac{1}{\Nn^2}\sum_{\mathcal{X}_n(4)}a_{i_1,\lambda}a_{i_1,\lambda'}a_{i_2,\lambda''}a_{i_2,\lambda'''} \\ 
= \frac{1}{\Nn}\big(W^{(i_1)}(n)W^{(i_2)}(n)
+ 2 M^{(i_1,i_2)}(n)^2 - 2 R^{(i_1,i_2)}(n) - S^{(i_1,i_2)}(n) + X^{(i_1,i_2)}(n)\big) \ . 
\end{eqnarray*}
Let us now prove (ii). We start by computing $\int_{\T}{H_2(T_n^{(i_1)}(x)) H_2(T_{n,k}^{(i_2)}(x)) dx}$ for fixed $k \in [3]$. Bearing in mind that 
\begin{equation*}
T_{n,k}^{(i_2)}(x) = i\sqrt{\frac{3 }{ n \Nn}}
\sum_{\la} \la_k a_{i_2,\la} e_{\la}(x)
\end{equation*}
and using \eqref{ForB}, we have
\begin{eqnarray*}
&&\int_{\T}{H_2(T_n^{(i_1)}(x)) H_2(T_{n,k}^{(i_2)}(x)) \ dx}   
=  \int_{\T}{\big( T_n^{(i_1)}(x)^2 T_{n,k}^{(i_2)}(x)^2 - T_n^{(i_1)}(x)^2 - T_{n,k}^{(i_2)}(x)^2 + 1\big) \ dx} \\
&&= \frac{3}{n\Nn^2}\sum_{\lambda} |a_{i_1,\lambda}|^2 \sum_{\lambda} \lambda_k^2|a_{i_2,\lambda'}|^2 
- \frac{6}{n\Nn^2} \bigg(\sum_{\lambda} \lambda_k a_{i_1,\lambda}\overline{a_{i_2,\lambda}}\bigg)^2 \\
&&\qquad -\frac{6}{n\Nn^2} \sum_{\lambda} \lambda^2_k |a_{i_1,\lambda}|^2 |a_{i_2,\lambda}|^2  + \frac{3}{n\Nn^2}\sum_{\lambda} \lambda^2_k a_{i_1,\lambda}^2 \overline{a_{i_2,\lambda}}^2 \\
&&\qquad -\frac{3}{n\Nn^2}\sum_{\mathcal{X}_n(4)} \lambda''_k \lambda'''_k a_{i_1,\lambda}a_{i_1,\lambda'}a_{i_2,\lambda''}a_{i_2,\lambda'''} \\
&& \qquad - \frac{1}{\Nn}\sum_{\la} |a_{i_1,\la}|^2 
- \frac{3}{n\Nn}\sum_{\la} |a_{i_2,\la}|^2 \lambda_k^2 + 1  \ .
\end{eqnarray*}
Hence, summing over $k$ and using the fact that $\la_1^2+\la_2^2+\la_3^2=n$ for $\la=(\la_1,\la_2,\la_3) \in \Lambda_n$ yields
\begin{eqnarray*}
b_2^{(i_1,i_2)}(n) &=& \frac{3}{\Nn^2}\sum_{\lambda} 
|a_{i_1,\lambda}|^2 \sum_{\lambda} |a_{i_2,\lambda}|^2 
-\frac{6}{n\Nn^2} \sum_{k=1}^3\bigg(\sum_{\lambda} \lambda_k  a_{i_1,\lambda}\overline{a_{i_2,\lambda}} \bigg)^2 
-\frac{6}{\Nn^2} \sum_{\lambda}  |a_{i_1,\lambda}|^2 |a_{i_2,\lambda}|^2 \\
&& \quad + \frac{3}{\Nn^2}\sum_{\lambda} a_{i_1,\lambda}^2 \overline{a_{i_2,\lambda}}^2 
-\frac{3}{n\Nn^2}\sum_{k=1}^3\sum_{\mathcal{X}_n(4)} \lambda''_k \lambda'''_k a_{i_1,\lambda}a_{i_1,\lambda'}a_{i_2,\lambda''}a_{i_2,\lambda'''} \\
&& \quad - \frac{3}{\Nn}\sum_{\la}|a_{i_1,\la}|^2 
- \frac{3}{\Nn}\sum_{\lambda} |a_{i_2,\lambda}|^2  + 3 \ .
\end{eqnarray*}
Note that we can rewrite the second term as 
\begin{eqnarray*}
-\frac{6}{n\Nn^2} \sum_{k=1}^3\bigg(\sum_{\lambda} \lambda_k  a_{i_1,\lambda}\overline{a_{i_2,\lambda}} \bigg)^2  
= \frac{6}{\Nn}\sum_{k=1}^3\bigg(\frac{i}{\sqrt{n\Nn}}\sum_{\la} \la_k a_{i_1,\la} \overline{a_{i_2,\la}}\bigg)^2
= \frac{6}{\Nn}\sum_{k=1}^3 M_k^{(i_1,i_2)}(n)^2 \ .
\end{eqnarray*}
Substituting \eqref{rel} in the computation above shows that $b_2^{(i_1,i_2)}(n)$ is equal to
\begin{eqnarray*} 
\frac{3}{\Nn} \left(W^{(i_1)}(n) W^{(i_2)}(n) +
2\sum_{k=1}^3 M_k^{(i_1,i_2)}(n)^2 
- 2R^{(i_1,i_2)}(n) + S^{(i_1,i_2)}(n)
- \sum_{k=1}^3 X^{(i_1,i_2)}_{kk}(n) \right) , 
\end{eqnarray*}
which is the desired equality.  Let us now prove (iii). First, by \eqref{ForC}, we have for $k\neq j$, 
\begin{eqnarray*}
&&\int_{\T}{H_2(T_{n,k}^{(i_1)}(x)) H_2(T_{n,j}^{(i_2)}(x))\ dx}    
= \int_{\T} \big( T_{n,k}^{(i_1)}(x)^2 T_{n,j}^{(i_2)}(x)^2 - 
T_{n,k}^{(i_1)}(x)^2 - T_{n,j}^{(i_2)}(x)^2 + 1 \big) \ dx \notag 
\\
&& = \frac{9}{n^2\Nn^2}\sum_{\lambda} \lambda^2_k   
|a_{i_1,\lambda}|^2 \sum_{\lambda} \lambda_j^2 |a_{i_2,\lambda}|^2 
+ \frac{18}{n^2\Nn^2} \bigg(\sum_{\lambda} \lambda_k \lambda_j    a_{i_1,\lambda} \overline{a_{i_2,\lambda}} \bigg)^2 \notag \\
&&\qquad - \frac{18}{n^2\Nn^2} \sum_{\lambda} \lambda^2_k \lambda^2_j |a_{i_1,\lambda}|^2 |a_{i_2,\lambda}|^2 
- \frac{9}{n^2\Nn^2}\sum_{\lambda} \lambda^2_k \lambda^2_j a_{i_1,\lambda}^2 \overline{a_{i_2,\lambda}}^2 
\notag \\
&&\qquad + \frac{9}{n^2\Nn^2}\sum_{\mathcal{X}_n(4)} 
\lambda_k \lambda'_k\lambda''_j \lambda'''_j
a_{i_1,\lambda}a_{i_1,\lambda'}a_{i_2,\lambda''}a_{i_2,\lambda'''} 
\notag \\
&&\qquad - \frac{3}{n\Nn} \sum_{\la} \la_k^2 |a_{i_1,\la}|^2 - \frac{3}{n\Nn} \sum_{\la} \la_j^2 |a_{i_2,\la}|^2 +1 \notag \\
&&= \frac{9}{\Nn} \left(W^{(i_1)}_{kk}(n) W^{(i_2)}_{jj}(n) 
+ 2 M^{(i_1,i_2)}_{kj}(n)^2 
-2 R^{(i_1,i_2)}_{kj}(n) 
- S^{(i_1,i_2)}_{kj}(n)
+  X^{(i_1,i_2)}_{kkjj}(n)\right) \ , \notag
\end{eqnarray*}
where in the last line we used the relation
\begin{eqnarray*}
&&\frac{9}{n^2\Nn^2}\sum_{\lambda} \lambda_k^2 (|a_{i_1,\lambda}|^2-1) 
\sum_{\lambda} \lambda_j^2 (|a_{i_2,\lambda}|^2-1) \\
&&=\frac{9}{n^2\Nn^2}\sum_{\lambda} \lambda^2_k   
|a_{i_1,\lambda}|^2 \sum_{\lambda} \lambda_j^2 |a_{i_2,\lambda}|^2
- \frac{3}{n\Nn} \sum_{\la} \la_k^2 |a_{i_1,\la}|^2 - \frac{3}{n\Nn} \sum_{\la} \la_j^2 |a_{i_2,\la}|^2 +1 \ ,
\end{eqnarray*}
in view of \eqref{sumkj}.  The formula in (iii) then follows 
summing over all $j,k$ such that $j\neq k$.   Relations (iv) and (v) are proved similarly.  
\end{proof} 

\noindent\underline{\textit{Explicit expression of $\proj_4(L_n^{(\ell)})$.}}
We are now in position to provide the precise expression of the fourth-order chaotic component of $L_n^{(\ell)}$.   
Concerning the coefficients 
$\alpha\{p^{(i)}_j: (i,j) \in [\ell]\times[3]\} $ appearing in the Wiener chaos expansion of $L_n^{(\ell)}$ in \eqref{Peven}, we introduce the following notation: 
We write $\mathbf{0}_{\ell} \in \mathrm{Mat}_{\ell,3}(\R)$ for the zero-matrix; for an integer $m \geq 1$, we consider the mapping $s_m^{(\ell)}:([\ell]\times [3])^m \to \mathrm{Mat}_{\ell,3}(\R)$ defined by 
\begin{equation*}
s_m^{(\ell)}((i_1,j_1),\ldots,(i_m,j_m)) := 
\left\{\ind{(i,j) \in \{(i_1,j_1),\ldots, (i_m,j_m)\}} : (i,j) \in [\ell]\times[3]\right\} \ ,
\end{equation*}
that is, $s_m^{(\ell)}((i_1,j_1),\ldots,(i_m,j_m))$ is the $\ell\times3$ matrix whose entry is $1$ at positions $(i_1,j_1),\ldots,(i_m,j_m)$ and $0$ elsewhere. 
From Proposition \ref{ChaosPhi4} applied with $\bX=\bT_{n\star}^{(\ell)}(z), z\in\T$, we compute the constants in the three cases $(\ell,k)\in \{(1,3),(2,3),(3,3)\}$; these are entirely given once we compute (see \eqref{defalk})
\begin{equation*}
\alpha(1,3) =\frac{4}{\sqrt{2\pi}} \ , \quad \alpha(2,3) = 2 \ , \quad \alpha(3,3) =\frac{4}{\sqrt{2\pi}} \ .
\end{equation*} 
\begin{Prop}\label{coeffs}
For every $\ell \in [3]$ and every collection 
$I = \{(i_1,j_1)\neq(i_2,j_2)\neq(i_3,j_3)\neq(i_4,j_4) \in [\ell]\times [3]\} $, we have  
\begin{eqnarray*}  
&&\alpha_3^{(\ell)}\{ \mathbf{0}_{\ell} \} = \alpha(\ell,3)  \ , \\
&&\alpha_3^{(\ell)}\{2s^{(\ell)}_1((i_1,j_1))\} = 
\frac{1}{2!}\frac{1}{3}\alpha(\ell,3) = \frac{1}{6}\alpha(\ell,3) \ , \\
&& \alpha_3^{(\ell)}\{ 4 s^{(\ell)}_1((i_1,j_1))\} =  
-\frac{1}{4!}\frac{1}{5}\alpha(\ell,3) = -\frac{1}{120}\alpha(\ell,3) \ , \\
&&\alpha_3^{(\ell)}\{2s_2^{(\ell)}((i_1,j_1),(i_2,j_2))\} = 
-\frac{1}{60}\alpha(\ell,3) \ind{i_1=i_2} \\
&&\hspace{4.5cm} - \frac{1}{60}\alpha(\ell,3) \ind{i_1\neq i_2,j_1=j_2}\ind{\ell \in \{2,3\}} \\
&&\hspace{4.5cm} +\frac{1}{20}\alpha(\ell,3) \ind{i_1\neq i_2,j_1\neq j_2}\ind{\ell \in \{2,3\}}  \ , \\
&& \alpha_3^{(\ell)}\{ s^{(\ell)}_4((i_1,j_1), (i_2,j_2), (i_3,j_3), (i_4,j_4)) \} = -\frac{2}{15}\alpha(\ell,k)  
\ind{I \in S}\ind{\ell \in \{2,3\}} \ ,
\end{eqnarray*}
where $
S=\{ \{(i_1,j_1),(i_1,j_2),(i_2,j_1),(i_2,j_2)\}: i_1 \neq i_2, j_1 \neq j_2\}$. 
\end{Prop}
In particular, from Proposition \ref{ChaosPhi4}, it becomes clear that the fourth-order chaotic component of $L_n^{(\ell)}$ does not involve (i) any non-linear interaction of the three ARWs simultaneously (for $\ell=3$), and (ii) any product of odd Hermite polynomials except expressions of the form $H_1(\cdot)H_1(\cdot)H_1(\cdot)H_1(\cdot)$. 

\medskip
Recalling the random variables introduced in Definition \ref{RV2}, we define the following two quantities: for $\ell \in [3]$ and $i_1 \in [\ell]$,  
\begin{eqnarray}\label{A}
A^{(i_1)}_{n,\ell}&:=&   
\frac{\beta_4 \beta_0^{\ell-1}}{4!} \alpha^{(\ell)}_3\{\mathbf{0}_{\ell}\} \cdot a_1^{(i_1)}(n) 
+  \frac{\beta_0^{\ell-1}\beta_2}{2!} 
\alpha^{(\ell)}_3\{2s^{(\ell)}_1((1,1))\} \cdot a_2^{(i_1)}(n) \notag \\
&& \quad
+  \quad \beta_0^\ell \alpha^{(\ell)}_3\{4s_1^{(\ell)}((1,1))\} \cdot a_3^{(i_1)}(n)  \notag \\
&& \quad
+   \quad \beta_0^\ell \alpha^{(\ell)}_3\{2s_2^{(\ell)}((1,1),(1,2))\} \cdot a_4^{(i_1)}(n) \ ;  
\end{eqnarray}
and for $\ell \in \{2,3\}$ and $i_1<i_2\in [\ell]$, 
\begin{eqnarray}\label{B}
B^{(i_1,i_2)}_{n,\ell} &:=& 
\left(\frac{\beta_2}{2!}\right)^2 \beta_0^{\ell-2}  \alpha^{(\ell)}_3\{\mathbf{0}_{\ell}\}  \cdot b_1^{(i_1,i_2)}(n)
+  \frac{\beta_2 \beta_0^{\ell-1}}{2!} 
\alpha^{(\ell)}_3\{2s_1^{(\ell)}((1,1))\} \cdot b_2^{(i_1,i_2)}(n) \notag\\
&&  \quad
+ \quad \frac{\beta_0^{\ell-1} \beta_2}{2!} \alpha^{(\ell)}_3\{2s_1^{(\ell)}((1,1))\} \cdot {b_2'}^{(i_1,i_2)}(n)  \notag \\
&&  \quad
+ \quad \beta_0^{\ell}\alpha^{(\ell)}_3\{2s^{(\ell)}_2((1,1),(2,2))\} \cdot b_3^{(i_1,i_2)}(n) \notag \\
&& \quad
+ \quad \beta_0^{\ell} \alpha^{(\ell)}_3\{2s_2^{(\ell)}((1,1),(2,1)) \} \cdot b_4^{(i_1,i_2)}(n) \notag \\
&& \quad 
+ \quad    \beta_0^{\ell} \alpha^{(\ell)}_3\{s^{(\ell)}_4((1,1),(1,2),(2,1),(2,2))\} \cdot b_5^{(i_1,i_2)}(n) \ . 
\end{eqnarray}
Then, the fourth-order chaotic component of $L_n^{(\ell)}$ is given by (recall \eqref{intnorm})
\begin{equation}\label{Ln4S}
\proj_4(L_n^{(\ell)})  = \left(\frac{E_n}{3}\right)^{\ell/2}   
\left( \sum_{i_1\in [\ell]} A_{n,\ell}^{(i_1)} + 
\sum_{i_1<i_2 \in [\ell]} B_{n,\ell}^{(i_1,i_2)}\right)  =: 
\left(\frac{E_n}{3}\right)^{\ell/2} \cdot   S_n^{(\ell)} \ ,
\end{equation}
with the convention that $\sum_{i_1<i_2 \in [\ell]} = 0$ if $\ell=1$.
Using \eqref{beta024} and Proposition \ref{coeffs}, the expressions in \eqref{A} and \eqref{B} simplify to
\begin{equation*}
A_{n,\ell}^{(i_1)} = \frac{2}{ (2\pi)^{\ell/2}} \alpha(\ell,3) \left( 
\frac{1}{16} a_1^{(i_1)}(n) - \frac{1}{24} a_2^{(i_1)}(n) - \frac{1}{240} a_3^{(i_1)}(n) - \frac{1}{120}a_4^{(i_1)}(n)\right) 
\end{equation*} 
and 
\begin{eqnarray*}
&&B_{n,\ell}^{(i_1,i_2)} = \frac{2}{ (2\pi)^{\ell/2}} \alpha(\ell,3) \bigg( 
\frac{1}{8} b_1^{(i_1,i_2)}(n) - \frac{1}{24} b_2^{(i_1,i_2)}(n) - \frac{1}{24} {b_2'}^{(i_1,i_2)}(n)+ \frac{1}{40}b_3^{(i_1,i_2)}(n) \\
&&\hspace{3cm} \qquad -  \frac{1}{120}b_4^{(i_1,i_2)}(n)-  \frac{1}{15}b_5^{(i_1,i_2)}(n) \bigg) \ .
\end{eqnarray*} 
Using the expansions in Lemma \ref{ExpA} and the fact that $W^{(i_1)}(n)=\sum_{k=1}^3W^{(i_1)}_{kk}(n)$, we compute
\begin{gather}\label{Amu}
A_{n,\ell}^{(i_1)} = 
\frac{2}{ (2\pi)^{\ell/2}} \frac{\alpha(\ell,3)}{\Nn}
\bigg(  
-\frac{1}{40} \sum_{k<j} \big( W^{(i_1)}_{kk}(n)-W^{(i_1)}_{jj}\big)^2
-\frac{3}{20}\sum_{k<j} W^{(i_1)}_{kj}(n)^2
+ \mu^{(i_1)}(n) \bigg)   
\end{gather}
where $\mu^{(i_1)}(n)$ is given by 
\begin{eqnarray}\label{mun}
\mu^{(i_1)}(n) = \frac{1}{20}R^{(i_1,i_1)}(n) 
+ \frac{1}{16} X^{(i_1,i_1)}(n) 
+ \frac{1}{8}\sum_{k=1}^3 X^{(i_1,i_1)}_{kk}(n)
- \frac{3}{80} \sum_{k,j=1}^3 X^{(i_1,i_1)}_{kkjj}(n) \ .
\end{eqnarray}
Similarly, if $\ell \in \{2,3\}$, using Lemma \ref{ExpB} together with the fact that $M^{(i_1,i_2)}(n) = \sum_{k=1}^3M_{kk}^{(i_1,i_2)}(n)$, yields
\begin{eqnarray}\label{Beta}
B_{n,\ell}^{(i_1,i_2)} &=& \frac{2}{ (2\pi)^{\ell/2}} \frac{\alpha(\ell,3)}{\Nn}
\bigg( 
-\frac{1}{10}
\sum_{k<j}\big(W^{(i_1)}_{kk}(n)-W^{(i_1)}_{jj}(n)\big)\big(W^{(i_2)}_{kk}(n)-W^{(i_2)}_{jj}(n)\big) \notag \\
&&  
-\frac{3}{5}\sum_{k<j} W^{(i_1)}_{kj}(n)W^{(i_2)}_{kj}(n) 
+ \frac{1}{10}\sum_{k=1}^3 M^{(i_1,i_2)}_{kk}(n)^2 
-\frac{1}{20}\sum_{k\neq j}M^{(i_1,i_2)}_{kk}(n)M^{(i_1,i_2)}_{jj}(n) \notag \\
&&   
-\frac{1}{2} \sum_{k=1}^3 M^{(i_1,i_2)}_k(n)^2 
+\frac{3}{10}\sum_{k<j} M^{(i_1,i_2)}_{kj}(n)^2 
+ \eta^{(i_1,i_2)}(n) \bigg) 
\end{eqnarray}
where $\eta^{(i_1,i_2)}(n)$ is given by 
\begin{eqnarray}\label{etan}
\eta^{(i_1,i_2)}(n) = \frac{2}{5}R^{(i_1,i_2)}(n) 
- \frac{3}{10} S^{(i_1,i_2)}(n) 
+ \frac{1}{8}  X^{(i_1,i_2)}(n)
+ \frac{1}{4} \sum_{k=1}^3 X_{kk}^{(i_1,i_2)}(n)   
- \frac{3}{40} \sum_{k,j=1}^3 X^{(i_1,i_2)}_{kkjj}(n) .
\end{eqnarray}

\subsubsection{Asymptotic simplification of $\proj_4(L_n^{(\ell)})$}\label{subAsy}
We will now lead an asymptotic study of the fourth chaotic component of 
$\proj_4(L_n^{(\ell)})$ obtained in \eqref{Ln4S}. This analysis is based on a multivariate Central Limit Theorem for the summands composing the expressions of $A_{n,\ell}^{(i_1)}$ and $B_{n,\ell}^{(i_1,i_2)}$.

\medskip
We start by recalling the following formulae (see Lemma 3.3 and Appendix C in \cite{Cam17}), which are a consequence of the asymptotic equidistribution of lattice points projected to the unit two-sphere. 
\begin{Lem}\label{LemSumsLambda}
For every $j,k,l,m \in [3]$, we have
\begin{gather}
\frac{1}{n\Nn}\sum_{\lambda \in \Lambda_n} \lambda_k \lambda_j =  \frac{1}{3}\ind{k=j} \ ,\label{sumkj}\\
\frac{1}{n^2 \Nn} \sum_{\lambda \in \Lambda_n} \lambda_k \lambda_l \lambda_j \lambda_m
= \frac{1}{5} \ind{k=l=j=m} + 
\frac{1}{15}\bigg(\ind{\substack{k=l\\j=m\\k\neq j}}
+ \ind{\substack{k=j\\l=m\\k\neq l}} 
+ \ind{\substack{k=m\\l=j\\k\neq l}}\bigg) + \eps_n \ ,\label{sumkjs} 
\end{gather}
where $\eps_n=O(n^{-1/28+o(1)})$, as $n \to \infty, \notcon{n}{0,4,7}{8}$. 
\end{Lem}
For the random variables in Definition \ref{RV1}, we prove the following asymptotic relations.  
\begin{Lem}\label{LLN}
Fix $\ell \in [3]$. For every $i_1,i_2 \in [\ell]$, the following  holds as $n \to \infty, \notcon{n}{0,4,7}{8}$: 
\begin{gather}
R^{(i_1,i_2)}(n) \toP 2 \ind{i_1 = i_2} + \ind{i_1 \neq i_2} \ , \label{R}    \\
S^{(i_1,i_2)}(n) \toP 2 \ind{i_1=i_2} \ , \label{S}\\
X^{(i_1,i_2)}(n),X^{(i_1,i_2)}_{kk}(n),X^{(i_1,i_2)}_{kkjj}(n) \toLtwo 0 \ . \label{X} 
\end{gather}
\end{Lem}
\begin{proof}
We introduce the equivalence relation $\sim$ on $\Lambda_n$ defined by $\lambda \sim \lambda' $ if and only if $\lambda = - \lambda'$ and write $\Lambda_n/_{\sim}$ for the set of representatives of the equivalence classes under $\sim$. Then, it follows that $\mathrm{card}(\Lambda_n/_{\sim})=\Nn/2$ and the collections
$\{|a_{i_1,\lambda}|^2|a_{i_2,\lambda}|^2: \lambda \in \Lambda_n/_{\sim}\}$ resp. $\{a_{i_1,\lambda}^2\overline{a_{i_2,\lambda}}^2: \lambda \in \Lambda_n/_{\sim}\}$ are families of i.i.d. random variables with respective means
\begin{gather*}
\E{|a_{i_1,\lambda}|^2|a_{i_2,\lambda}|^2} = 2\ind{i_1=i_2}+ \ind{i_1\neq i_2} \ , \quad
\E{a_{i_1,\lambda}^2\overline{a_{i_2,\lambda}}^2} = 2\ind{i_1=i_2} \ . 
\end{gather*}
Thus, relations \eqref{R} and \eqref{S} follow from the Law of Large Numbers:  as $n \to \infty, \notcon{n}{0,4,7}{8}$, we have 
\begin{equation*}
R^{(i_1,i_2)}(n) = \frac{1}{\Nn/2}\sum_{\lambda \in \Lambda_n/_{\sim}} |a_{i_1,\lambda}|^2 |a_{i_2,\lambda}|^2 \toP 2\ind{i_1=i_2}+\ind{i_1\neq i_2} \ ,
\end{equation*}
and 
\begin{equation*}
S^{(i_1,i_2)}(n) = \frac{1}{\Nn/2}\sum_{\lambda \in \Lambda_n/_{\sim}} a_{i_1,\lambda}^2 \overline{a_{i_2,\lambda}}^2 \toP
2\ind{i_1=i_2} \ .
\end{equation*}
The convergences in \eqref{X} have been proved in \cite{Cam17} in the case $i_1=i_2$. Using independence and the fact that $a_{i_1,\lambda}=\overline{a_{i_1,-\lambda}}$ for every $i_1\in [\ell]$ and $\lambda \in \Lambda_n$ yields
\begin{eqnarray*}
&& \E{|X^{(i_1,i_2)}(n)|^2}  = 
\E{X^{(i_1,i_2)}(n) \overline{X^{(i_1,i_2)}(n)}} \\
&=& 
\frac{1}{\Nn^2}\sum_{  \mathcal{X}_n(4) } \sum_{  \mathcal{X}_n(4) }
\E{a_{i_1,\lambda} a_{i_1,\lambda'} 
a_{i_1,-\mu}  a_{i_1,-\mu'} }
\E{ a_{i_2,\lambda''} a_{i_2,\lambda'''} 
a_{i_2,-\mu''} a_{i_2,-\mu'''} } \\
&=:& \frac{1}{\Nn^2} \sum_{  \mathcal{X}_n(4) }
\sum_{  \mathcal{X}_n(4) }
\E{z^{(i_1)}_{\lambda,\lambda',\mu,\mu'}}
\E{z^{(i_2)}_{\lambda'',\lambda''',\mu'',\mu'''}} \ .
\end{eqnarray*} 
Let us consider the random variable $z^{(i_1)}_{\lambda,\lambda',-\mu,-\mu'}$.
Denote by $N$ the number of pairs of vectors that are equal in absolute value among $\{\lambda,\lambda',\mu,\mu'\}$. 
Since we consider vectors of $\mathcal{X}_n(4)$, we have that $\lambda+\lambda'\neq 0$ and $\mu+\mu'\neq 0$. Conditional to this observation, we claim that the only non-zero contributions of $\E{z^{(i_1)}_{\lambda,\lambda',-\mu,-\mu'}}$ arise when $N=2$ or $N=4$. 
Indeed, if $N=0$, all the vectors are distinct, so that by independence, $\E{z^{(i_1)}_{\lambda,\lambda',-\mu,-\mu'}}=0$. If $N=1$, then 
$\E{z^{(i_1)}_{\lambda,\lambda',-\mu,-\mu'}}$ takes one of the forms 
\begin{gather*}
\E{|a_{i_1,s}|^2}\E{a_{i_1,t}}\E{a_{i_1,t'}}=0 \ , \quad 
\E{a_{i_1,s}^2}\E{a_{i_1,t}}\E{a_{i_1,t'}}=0 \ , \quad   s \neq \pm t \neq \pm t' \ . 
\end{gather*}
If $N=2$, $\E{z^{(i_1)}_{\lambda,\lambda',-\mu,-\mu'}}$ is of the form
\begin{gather*}
\E{|a_{i_1,s}|^2}\E{a_{i_1,t}^2} = 0  \ , \quad 
\E{|a_{i_1,s}|^2}\E{|a_{i_1,t}|^2} = 1  \ , \quad 
\E{a_{i_1,s}^2}\E{a_{i_1,t}^2} = 0 \ , \quad  s \neq \pm t \ .
\end{gather*}
If $N=3$, then $\E{z^{(i_1)}_{\lambda,\lambda',-\mu,-\mu'}}$ is of the form
\begin{gather*}
\E{a_{i_1,s}^3}\E{a_{i_1,t}} = 0 \ , \quad \E{|a_{i_1,s}|^2\overline{a_{i_1,s}}}\E{a_{i_1,t}} = 0 \ , \quad s \neq \pm t \ .
\end{gather*}
Finally, if $N=4$, the elements $\lambda,\lambda',\mu,\mu'$ are all the same in absolute value, so that 
$\E{z^{(i_1)}_{\lambda,\lambda',-\mu,-\mu'}}$ is of the form 
$\E{|a_{i_1,s}|^4}=2$ or $ \E{a_{i_1,s}^4} = 0$. 
The same arguments hold for $\E{z^{(i_2)}_{\lambda'',\lambda''',\mu'',\mu'''}}$. Therefore, in every non-zero contributions, the vector $(\lambda, \lambda',\lambda'',\lambda''')$ determines the choices of $(\mu,\mu',\mu'',\mu''')$, so that 
\begin{eqnarray*}
\E{|X^{(i_1,i_2)}(n)|^2}  \ll \frac{\mathrm{card}(\mathcal{X}_n(4))}{\Nn^2} \ll \frac{\Nn^{7/4+o(1)}}{\Nn^2} = o(1) \ ,
\end{eqnarray*}
as $n\to \infty, \notcon{n}{0,4,7}{8}$ in view of \eqref{EstX4}.
\end{proof}

\medskip
\noindent\underline{\textit{A multivariate Central Limit Theorem.}} 
Recalling the random variables defined in Definition \ref{RV1}, we define the following two random vectors for $n\in S_3$: for every $\ell \in [3]$ and $i_1 \in [\ell]$, 
\begin{gather*}
\mathbb{W}^{(i_1)}(n):=\big(W^{(i_1)}_{11}(n),W^{(i_1)}_{12}(n),
W^{(i_1)}_{13}(n), W^{(i_1)}_{22}(n), W^{(i_1)}_{23}(n), W^{(i_1)}_{33}(n)\big) \in \R^6 \ ,  
\end{gather*}
and, for every $\ell\in \{2,3\}$ and $i_1<i_2 \in [\ell]$,
\begin{gather*}
\mathbb{M}^{(i_1,i_2)}(n):= \big(M^{(i_1,i_2)}_1(n),M^{(i_1,i_2)}_2(n),M^{(i_1,i_2)}_3(n),M^{(i_1,i_2)}_{11}(n),M^{(i_1,i_2)}_{12}(n),M^{(i_1,i_2)}_{13}(n), \\
M^{(i_1,i_2)}_{22}(n),M^{(i_1,i_2)}_{23}(n),M^{(i_1,i_2)}_{33}(n)\big) \in \R^{9} \ .
\end{gather*}
The covariance matrix of the vectors $\mathbb{W}^{(i_1)}(n)$ and
$\mathbb{M}^{(i_1,i_2)}(n)$ above is computed in the following lemmas.

\begin{Lem}\label{LemCovW}
For every $n \in S_3, \ell \in [3]$ and every $i_1\in [\ell]$, the covariance matrix of $\mathbb{W}^{(i_1)}(n)$ is 
\begin{equation}\label{Wn}
\Sigma_{\mathbb{W}(n)} =  
\begin{pmatrix}
\frac{2}{5}+\eps_n & 0 & 0 &  \frac{2}{15}+\eps_n & 0 & \frac{2}{15}+\eps_n \\[0.1cm]
0 & \frac{2}{15}+ \eps_n & 0 &  0 & 0 & 0 \\[0.1cm]
0 & 0 & \frac{2}{15}+\eps_n &  0 & 0 & 0 \\[0.1cm]
\frac{2}{15}+\eps_n & 0 & 0 & \frac{2}{5}+\eps_n & 0 & \frac{2}{15}+\eps_n \\[0.1cm]
0 & 0 & 0 & 0 & \frac{2}{15}+\eps_n & 0\\[0.1cm]
\frac{2}{15}+\eps_n & 0 & 0 & \frac{2}{15}+\eps_n & 0  & \frac{2}{5} + \eps_n 
\end{pmatrix} \ ,
\end{equation}
where $\eps_n = O(n^{-1/28+o(1)})$, as $n \to \infty, \notcon{n}{0,4,7}{8}$.
\end{Lem}
\begin{proof} 
The proof mainly follows from the relations in Lemma \ref{LemSumsLambda}, together with the fact 
\begin{equation*}
\E{(|a_{i_1,\lambda}|^2-1)(|a_{i_1,\lambda'}|^2-1)} = \ind{\lambda=\pm
\lambda'} \ .
\end{equation*} 
The covariances of $W^{(i_1)}_{jk}$ for $j,k \in [3]$ have been computed in \cite{Cam17}, Appendix C. 
\end{proof}

\begin{Lem}\label{LemCovM}
For every $n \in S_3, \ell \in \{2,3\}$ and every $i_1 < i_2 \in [\ell]$, the covariance matrix of $\mathbb{M}^{(i_1,i_2)}(n)$ is 
\begin{equation}\label{Mn}
\Sigma_{\mathbb{M}(n)} = 
\begin{pmatrix}
\frac{1}{3} & 0 & 0 &  0 & 0 & 0 & 0 & 0 & 0\\[0.1cm]
0 & \frac{1}{3} & 0 &  0 & 0 & 0 & 0 & 0 & 0\\[0.1cm]
0 & 0 & \frac{1}{3} &  0 & 0 & 0 & 0 & 0 & 0\\[0.1cm]
0 & 0 & 0 & \frac{1}{5}+\eps_n & 0 & 0 & \frac{1}{15}+\eps_n & 0 & \frac{1}{15}+\eps_n \\[0.1cm]
0 & 0 & 0 & 0 & \frac{1}{15}+\eps_n & 0 & 0 & 0 & 0 \\[0.1cm]
0 & 0 & 0 & 0 & 0 & \frac{1}{15}+\eps_n & 0 & 0 & 0 \\[0.1cm] 
0 & 0 & 0 & \frac{1}{15}+\eps_n & 0 & 0 & \frac{1}{5}+\eps_n & 0 & \frac{1}{15}+\eps_n \\[0.1cm]
0 & 0 & 0 & 0 & 0 & 0 & 0 & \frac{1}{15}+\eps_n & 0 \\[0.1cm]
0 & 0 & 0 & \frac{1}{15}+\eps_n & 0 & 0 & \frac{1}{15}+\eps_n & 0 & \frac{1}{5}+\eps_n 
\end{pmatrix} \ ,
\end{equation}
where $\eps_n = O(n^{-1/28+o(1)})$, as $n \to \infty, \notcon{n}{0,4,7}{8}$.
\end{Lem}
\begin{proof}
Similarly as in the proof of Lemma \ref{LemCovW}, we use Lemma \ref{LemSumsLambda} and the fact that, by independence
\begin{equation*}
\E{a_{i_1,\lambda}\overline{a_{i_2,\lambda}}a_{i_1,\lambda'}\overline{a_{i_2,\lambda'}}} 
= \E{a_{i_1,\lambda}a_{i_1,\lambda'}} \E{\overline{a_{i_2,\lambda}}\ \overline{a_{i_2,\lambda'}}}
= \ind{\lambda =- \lambda'} 
\ .
\end{equation*}
Using this identity, it follows that  
\begin{equation*}
\Cov{M^{(i_1,i_2)}_{j}(n)}{M^{(i_1,i_2)}_{k}(n)} 
=\E{M^{(i_1,i_2)}_{j}(n)M^{(i_1,i_2)}_{k}(n)} 
= \frac{1}{n \Nn}\sum_{\lambda} \lambda_j \lambda_k
= \frac{1}{3}\ind{j=k} \ ,
\end{equation*}
and 
\begin{equation*}
\Cov{M^{(i_1,i_2)}_{j}(n)}{M^{(i_1,i_2)}_{lm}(n)}
=\E{M^{(i_1,i_2)}_{j}(n)M^{(i_1,i_2)}_{lm}(n)} = \frac{i}{n\sqrt{n}\Nn}
\sum_{\lambda} \lambda_j \lambda_l\lambda_m = 0 \ .
\end{equation*}
Moreover,  
\begin{eqnarray*}
&&\Cov{M^{(i_1,i_2)}_{jk}(n)}{M^{(i_1,i_2)}_{lm}(n)} 
=\E{M^{(i_1,i_2)}_{jk}(n)M^{(i_1,i_2)}_{lm}(n)} = \frac{1}{n^2\Nn}\sum_{\lambda}
\lambda_j \lambda_k \lambda_l \lambda_m\\
&&=  \frac{1}{5} \ind{k=l=j=m} + 
\frac{1}{15}\bigg(\ind{\substack{k=l\\j=m\\k\neq j}}
+ \ind{\substack{k=j\\l=m\\k\neq l}} 
+ \ind{\substack{k=m\\l=j\\k\neq l}}\bigg) + \eps_n \ ,
\end{eqnarray*}
which finishes the proof.
\end{proof}
The following proposition plays a central role in the study of the fourth chaotic component of the nodal volume $L_n^{(\ell)}$ in the high-frequency regime. 
We define the limiting matrices obtained from \eqref{Wn} and \eqref{Mn}:
\begin{equation*}
\Sigma_{\mathbb{W}} := \lim_{n\to \infty} \Sigma_{\mathbb{W}(n)} \ , \quad 
\Sigma_{\mathbb{M}} := \lim_{n\to \infty} \Sigma_{\mathbb{M}(n)} \ ,
\end{equation*} 
where for a square matrix $M_n=(m_{ij}(n))$, we set $\lim_n M_n:= (\lim_n m_{ij}(n))$.

\begin{Prop}\label{PropJointLaw3}
As $n \to \infty, \notcon{n}{0,4,7}{8}$, the random vector 
\begin{equation*}
\mathbf{V}_{1,2,3}(n):=\big(\mathbb{W}^{(1)}(n),\mathbb{W}^{(2)}(n), \mathbb{W}^{(3)}(n),\mathbb{M} ^{(1,2)}(n),\mathbb{M} ^{(1,3)}(n),\mathbb{M} ^{(2,3)}(n)\big) \in \R^{45}
\end{equation*}
converges in distribution to 
\begin{equation*} 
\mathbf{G}_{1,2,3}:=\big(\mathbb{G}^{(1)},\mathbb{G}^{(2)},\mathbb{G}^{(3)},\mathbb{G}^{(1,2)},
\mathbb{G}^{(1,3)},\mathbb{G}^{(2,3)}\big) 
\sim \mathcal{N}_{45}(\mathbf{0},\Sigma_{\mathbf{G}_{1,2,3}}) \ , 
\end{equation*}
where  
\begin{equation*}
\Sigma_{\mathbf{G}_{1,2,3}} = 
\Sigma_{\mathbb{W}}  \oplus \Sigma_{\mathbb{W}}  \oplus 
\Sigma_{\mathbb{W}}  \oplus
\Sigma_{\mathbb{M}}  \oplus \Sigma_{\mathbb{M}} 
\oplus \Sigma_{\mathbb{M}} 
\in \mathrm{Mat}_{45,45}(\R) \ .
\end{equation*}
\end{Prop}
\begin{proof}
We start by showing that the covariance matrix of the vector 
$\mathbf{V}_{1,2,3}(n)$ has the block diagonal form 
\begin{equation*}
\Sigma_{\mathbf{V}_{1,2,3}(n)} = 
\Sigma_{\mathbb{W}(n)}  \oplus \Sigma_{\mathbb{W}(n)}  \oplus 
\Sigma_{\mathbb{W}(n)}  \oplus
\Sigma_{\mathbb{M}(n)}  \oplus \Sigma_{\mathbb{M}(n)} 
\oplus \Sigma_{\mathbb{M}(n)}  \ .
\end{equation*}
From Lemmas \ref{LemCovW} and \ref{LemCovM} and by independence, we have
\begin{equation*}
\E{(|a_{i_1,\lambda}|^2-1)a_{i_1,\lambda'} \overline{a_{i_2,\lambda'}}}
= \E{(|a_{i_1,\lambda}|^2-1)a_{i_1,\lambda'}} \E{\overline{a_{i_2,\lambda'}}} 
= 0 \ ,
\end{equation*}
and therefore $\Cov{\big((\mathbb{W}^{(i_1)}(n)\big)_l}{\big(\mathbb{M}^{(i_1,i_2)}(n)\big)_m}=0$ for every $l=1,\ldots,6$ and $m=1,\ldots,9$. Similarly, since for every $i_2 \neq i_3$, 
\begin{eqnarray*}
\E{a_{i_1,\la} \overline{a_{i_2,\la}} a_{i_1,\la'} \overline{a_{i_3,\la'}}}
= \E{a_{i_1,\la} a_{i_1,\la'}} \E{\overline{a_{i_2,\la}}}\E{\overline{a_{i_3,\la'}}} = 0 \ , 
\end{eqnarray*}
we have that $\Cov{\big((\mathbb{M}^{(i_1,i_2)}(n)\big)_l}{\big(\mathbb{M}^{(i_1,i_3)}(n)\big)_m}=0$ for every $l,m=1,\ldots,9$. Thus, $\mathbf{V}_{1,2,3}(n)$ is of the desired form. 
Furthermore, we notice that all the components $\{\big(\mathbf{V}_{1,2,3}(n)\big)_l:l=1,\ldots,45\}$ of $\mathbf{V}_{1,2,3}(n)$ belong to the second Wiener chaos and that 
$\Sigma_{\mathbf{V}_{1,2,3}(n)} \to
\Sigma_{\mathbf{G}_{1,2,3}}$ entry-wise as $n \to \infty, \notcon{n}{0,4,7}{8}$. Thus, Theorem 6.2.3 of \cite{NP12} implies that, in order to prove the joint convergence to the Gaussian vector $\mathbf{G}_{1,2,3}$, it suffices to prove that the convergence holds component-wise, that is 
\begin{equation*}
\big(\mathbf{V}_{1,2,3}(n)\big)_l \Law \mathcal{N}(0, (\Sigma_{\mathbf{G}_{1,2,3}})_{ll}) \ , \qquad n\to \infty, \notcon{n}{0,4,7}{8} \ ,
\end{equation*}
for every $l = 1, \ldots, 45$. Using the Fourth Moment Theorem (Theorem 5.2.7, \cite{NP12}), this can be shown by proving that the fourth cumulant of $\big(\mathbf{V}_{1,2,3}(n)\big)_l$ converges to zero for every $l=1,\ldots,45$. For the sake of completeness, we include the computations for $W^{(i_1)}_{jk}(n)$ with $j\neq k$ and $M^{(i_1,i_2)}(n)$: writing $\Lambda_n/_{\sim}$ for the set of all the representatives of the equivalence class of $\Lambda_n$ under the symmetry $\lambda\mapsto -\lambda$ and using the fact that $j\neq k$, we have 
\begin{equation*}
W^{(i_1)}_{jk}(n) = \frac{1}{n \sqrt{\Nn}} \sum_{\lambda \in \Lambda_n} \lambda_j \lambda_k (|a_{i_1,\lambda}|^2-1)
= \frac{2}{n \sqrt{\Nn}} \sum_{\lambda \in \Lambda_n/_{\sim}} \lambda_j \lambda_k |a_{i_1,\lambda}|^2 \ ,
\end{equation*}
that is, $W^{(i_1)}_{jk}(n)$ is a sum of i.i.d. random variables. Moreover, for $\lambda \in \Lambda_n/_{\sim}$, 
\begin{equation*}
|a_{i_1,\lambda}|^2 \eqLaw 
\frac{u_{\lambda}^2}{2}+\frac{v_{\lambda}^2}{2} \ ,
\end{equation*}
where $u_{\lambda} \eqLaw v_{\lambda}$ are independent real $\mathcal{N}(0,1)$ random variables. Thus, using homogeneity and independence properties of cumulants (see e.g. \cite{PT11}), we have, as $n \to \infty, \notcon{n}{0,4,7}{8}$
\begin{eqnarray*}
\kappa_4\left(W^{(i_1)}_{jk}(n)\right) &=& 
\kappa_4\left(\frac{2}{n\sqrt{\Nn}} 
\sum_{\lambda \in \Lambda_n/\sim} \lambda_j \lambda_k 
\left(\frac{u_{\lambda}^2}{2}+\frac{v_{\lambda}^2}{2}\right)\right) \\
&=& \frac{2^4}{n^4\Nn^2}\sum_{\lambda \in \Lambda_n/_{\sim}}
\lambda_j^4 \lambda_k^4 \left(2^{-4}\kappa_4(u_{\lambda}^2)+2^{-4}\kappa_4(v_{\lambda}^2)\right) \\
&\leq& \frac{1}{\Nn^2}\sum_{\lambda \in \Lambda_n/_{\sim}}
\left( \kappa_4(u_{\lambda}^2)+ \kappa_4(v_{\lambda}^2)\right) 
\ll \frac{1}{\Nn} = o(1) \ ,
\end{eqnarray*}
where we used that
$\lambda_k^2\leq n$ for every $k=1,2,3$, which implies that $\lambda_j^4\lambda_k^4 \leq n^4$.  Concerning $M^{(i_1,i_2)}(n)$, we write 
\begin{eqnarray*}
M^{(i_1,i_2)}(n) =  \frac{1}{\sqrt{\Nn}} \sum_{\lambda \in \Lambda_n} a_{i_1,\lambda} \overline{a_{i_2,\lambda}} =
\frac{2}{\sqrt{\Nn}} \sum_{\lambda \in \Lambda_n/_{\sim}} a_{i_1,\lambda} \overline{a_{i_2,\lambda}} .
\end{eqnarray*} 
Noting that for every $\lambda \in \Lambda_n/_{\sim}$,
\begin{eqnarray*}
a_{i_1,\lambda} \overline{a_{i_2,\lambda}} \eqLaw 
\frac{(a_{i_1,\lambda} +\overline{a_{i_2,\lambda}})(a_{i_1,\lambda} - \overline{a_{i_2,\lambda}})}{2}
= \frac{a_{i_1,\lambda}^2-\overline{a_{i_2,\lambda}}^2}{2}
\end{eqnarray*}
and using independence, we infer
\begin{eqnarray*}
\kappa_4\left(M^{(i_1,i_2)}(n)\right) = 
\frac{2^4}{\Nn^2}\sum_{\lambda \in \Lambda_n/_{\sim}} 
\kappa_4(a_{i_1,\lambda} \overline{a_{i_2,\lambda}})
= \frac{1}{\Nn^2}\sum_{\lambda \in \Lambda_n/_{\sim}}
\kappa_4(a_{i_1,\lambda}^2-\overline{a_{i_2,\lambda}}^2)\\
= \frac{1}{\Nn^2}\sum_{\lambda \in \Lambda_n/_{\sim}}
\big(\kappa_4(a_{i_1,\lambda}^2)+\kappa_4(\overline{a_{i_2,\lambda}}^2)\big)  
\ll \frac{1}{\Nn}  = o(1) \ ,
\end{eqnarray*} 
as $n \to \infty, \notcon{n}{0,4,7}{8}$.
The other computations are done similarly. 
\end{proof}
The following corollary follows immediately: 
\begin{Cor}\label{SubVector}
For $\ell \in \{2,3\}$ and  $i_1 < i_2 \in [\ell]$, as $n \to \infty, \notcon{n}{0,4,7}{8}$, the random vector 
\begin{equation*}
\mathbf{V}_{i_1,i_2}(n):=\big(\mathbb{W}^{(i_1)}(n),\mathbb{W}^{(i_2)}(n), \mathbb{M} ^{(i_1,i_2)}(n)\big) \in \R^{21}
\end{equation*}
converges in distribution to 
\begin{equation*} 
\mathbf{G}_{i_1,i_2}:=\big(\mathbb{G}^{(i_1)},\mathbb{G}^{(i_2)},\mathbb{G}^{(i_1,i_2)}\big) 
\sim \mathcal{N}_{21}(\mathbf{0},\Sigma_{\mathbf{G}_{i_1,i_2}}) \ , 
\end{equation*}
where  
\begin{equation*}
\Sigma_{\mathbf{G}_{i_1,i_2}} = 
\Sigma_{\mathbb{W}}  \oplus \Sigma_{\mathbb{W}}  \oplus 
\Sigma_{\mathbb{M}}  \in \mathrm{Mat}_{21,21}(\R) \ .
\end{equation*}
\end{Cor}
We use the above established CLT in order to derive the limiting distribution of the fourth-order chaotic component of $L_n^{(\ell)}$. From Lemma \ref{LLN}, it follows that, as $n \to \infty, \notcon{n}{0,4,7}{8}$, the sequences in \eqref{mun} and \eqref{etan} satisfy
\begin{equation}\label{mueta}
\mu^{(i_1)}(n)  = \frac{1}{10} + o_{\Prob}(1) \ , \quad
\eta^{(i_1,i_2)}(n) = \frac{2}{5}+o_{\Prob}(1) \ ,
\end{equation} 
where $o_{\Prob}(1)$ denotes a sequence converging to zero in probability.
Now, bearing in mind the expressions \eqref{Amu} and \eqref{Beta}, we define
\begin{eqnarray*}
F(\mathbb{W}^{(i_1)}) :=
-\frac{1}{40} \sum_{k<j} \big( W^{(i_1)}_{kk}(n)-W^{(i_1)}_{jj}(n)\big)^2
-\frac{3}{20}\sum_{k<j} W^{(i_1)}_{kj}(n)^2 \ , \quad i_1 \in [\ell]
\end{eqnarray*}
and 
\begin{eqnarray*}
G(\mathbf{V}_{i_1,i_2}) &:=&
-\frac{1}{10}
\sum_{k<j}\big(W^{(i_1)}_{kk}(n)-W^{(i_1)}_{jj}(n)\big)\big(W^{(i_2)}_{kk}(n)-W^{(i_2)}_{jj}(n)\big) \notag \\
&&  
-\frac{3}{5}\sum_{k<j} W^{(i_1)}_{kj}(n)W^{(i_2)}_{kj}(n) 
+ \frac{1}{10}\sum_{k=1}^3 M^{(i_1,i_2)}_{kk}(n)^2 
-\frac{1}{20}\sum_{k\neq j}M^{(i_1,i_2)}_{kk}(n)M^{(i_1,i_2)}_{jj}(n) \notag \\
&&   
-\frac{1}{2} \sum_{k=1}^3 M^{(i_1,i_2)}_k(n)^2 
+\frac{3}{10}\sum_{k<j} M^{(i_1,i_2)}_{kj}(n)^2 \ , \quad i_1<i_2 \in [\ell].
\end{eqnarray*}
Combining these definitions with \eqref{mueta}, leads to the asymptotic relations 
\begin{eqnarray}
A_{n,\ell}^{(i_1)} &=& \frac{2}{ (2\pi)^{\ell/2}}
\frac{\alpha(\ell,3)}{\Nn}\cdot
\big[f(\mathbb{W}^{(i_1)}(n))+o_{\Prob}(1)\big] , \quad i_1 \in [\ell] \label{Af}\\
B_{n,\ell}^{(i_1,i_2)} &=& \frac{2}{ (2\pi)^{\ell/2}} \frac{\alpha(\ell,3)}{\Nn}\cdot
\big[g(\mathbf{V}_{i_1,i_2}(n))+ o_{\Prob}(1) \big] , \quad i_1<i_2 \in [\ell] \label{Bg}
\end{eqnarray}
where 
\begin{gather}\label{fg}
f(\mathbb{W}^{(i_1)}(n)):=F(\mathbb{W}^{(i_1)}(n))+\frac{1}{10} \ , \quad 
g(\mathbf{V}_{i_1,i_2}(n)):= G(\mathbf{V}_{i_1,i_2}(n))+\frac{2}{5}.
\end{gather}
Plugging  \eqref{Af} and \eqref{Bg} into  \eqref{Ln4S}  and using the CLT in Corollary \ref{SubVector}, we obtain that, as $n \to \infty, \notcon{n}{0,4,7}{8}$,
\begin{equation}\label{Limit}
\big(c_n^{(\ell)}\big)^{-1} \cdot \proj_4(L_n^{(\ell)})  \Law 
\sum_{i_1\in [\ell]} f(\mathbb{G}^{(i_1)}) + \sum_{i_1<i_2\in [\ell]} 
g(\mathbf{G}_{i_1,i_2})
=: L^{(\ell)}\  , 
\end{equation}
where 
\begin{equation}\label{cnl}
c_n^{(\ell)} := \bigg(\frac{E_n}{3}\bigg)^{\ell/2} \frac{2}{ (2\pi)^{\ell/2}}\frac{\alpha(\ell,3)}{\Nn}  \ .
\end{equation}



\subsubsection{Proofs of Proposition \ref{Var4} and \ref{Limit4}}\label{VarDist}
From the convergence in distribution stated in \eqref{Limit}, we conclude that the sequence 
$\{Y_n^{(\ell)}:=(c_n^{(\ell)})^{-1} \proj_4(L_n^{(\ell)}): n \in S_3\} $ living in the fourth Wiener chaos, is tight and thus  bounded in $L^p(\Prob)$ for any $p>0$ by virtue of the hypercontractivity property of Wiener chaoses (see e.g. Lemma 2.1 \cite{NR14}). This implies that the sequence $\{(Y_n^{(\ell)})^2: n\in S_3\}$ is uniformly integrable. By Skorohod's Representation Theorem (see e.g. Theorem 25.6 \cite{B77}), there exist random variables $\{Y_n^{(\ell)*}:n\in S_3\}$ and $L^{(\ell)*}$ defined on some auxiliary probability space $(\Omega^*,\mathscr{F}^*,\Prob^*)$, such that (i) $Y_n^{(\ell)*}\eqLaw Y_n^{(\ell)}$ for every $n\in S_3$ and $L^{(\ell)*}\eqLaw L^{(\ell)}$ and (ii)
$Y_n^{(\ell)*} \to L^{(\ell)*}, \Prob^*$-a.s. as $n\to \infty, \notcon{n}{0,4,7}{8}$.
Therefore we conclude that the sequence $\{(Y_n^{(\ell)*})^2:n\in S_3\}$ is uniformly integrable. In particular, we infer that $\norm{Y_n^{(\ell)}}_{L^2(\Prob)} = \norm{Y_n^{(\ell)*}}_{L^2(\Prob^*)} \to 
\norm{L^{(\ell)*}}_{L^2(\Prob^*)} = \norm{L^{(\ell)}}_{L^2(\Prob)}$, i.e.
\begin{equation*}
\big(c_n^{(\ell)}\big)^{-2} \V{\proj_4(L_n^{(\ell)})} \to \V{L^{(\ell)}} \ , 
\end{equation*}
as $n \to \infty, \notcon{n}{0,4,7}{8}$, or equivalently
\begin{equation}\label{V4asy}
\V{\proj_4(L_n^{(\ell)})} \sim \big(c_n^{(\ell)}\big)^2 \cdot \V{L^{(\ell)}} \ , 
\end{equation}
as $n \to \infty, \notcon{n}{0,4,7}{8}$.
Therefore, the asymptotic variance of $\proj_4(L_n^{(\ell)})$ in Proposition \ref{Var4} and its asymptotic distribution in Proposition \ref{Limit4} follow respectively from the
variance and distribution of $L^{(\ell)}$, given in the following statement.
\begin{Prop}\label{PropL}
For the random variable $L^{(\ell)}$ appearing in \eqref{Limit}, we have
\begin{gather*}
L^{(\ell)} \eqLaw 
-\frac{1}{50}\hat{\xi}_1(5\ell) 
- \frac{1}{25}\hat{\xi}_2\bigg(\frac{5\ell(\ell-1)}{2}\bigg)
+ \frac{1}{25}\hat{\xi}_3\bigg(\frac{5\ell(\ell-1)}{2}\bigg) 
+ \frac{1}{50} \hat{\xi}_4\bigg(\frac{5\ell(\ell-1)}{2}\bigg)
- \frac{1}{6}\hat{\xi}_5\bigg(\frac{3\ell(\ell-1)}{2}\bigg) \ ,
\end{gather*}
where $\{\hat{\xi}(k_i):i=1,\ldots,5\}$ is a family of independent centered chi-squared random variables, and therefore
\begin{gather*}
\V{L^{(\ell)}} = \ell \cdot \frac{1}{250} + \frac{\ell(\ell-1)}{2} \cdot
\frac{76}{375} \ .
\end{gather*}
\end{Prop}
\begin{proof}
The proof  
is based on lengthy but standard computations involving covariances of Gaussian random variables. We provide a sketch of the proof for the sake of readability. 
From relation \eqref{Limit} and the structure of the covariance matrix of $\mathbf{G}_{i_1,i_2}$ in Corollary \ref{SubVector}, it follows that 
\begin{equation*}
\V{L^{(\ell)}} = \ell \cdot\V{f(\mathbb{G}^{(1)})}+
\frac{\ell(\ell-1)}{2}\cdot\V{g(\mathbf{G}_{1,2})} \ .
\end{equation*}
The variances of $f(\mathbb{G}^{(1)})$ and $g(\mathbf{G}_{1,2})$ are then computed using the explicit expressions of $f$ and $g$ as well as the covariance matrix $\Sigma_{\mathbf{G}_{123}}$ in \eqref{PropJointLaw3}. The probability distribution of $L^{(\ell)}$ is obtained by a standard diagonalization argument in order to express the latter in terms of independent standard Gaussian random variables, implying in particular the formula for its variance.
\end{proof}
The proof of Propositon \ref{Limit4} is concluded, once we note that 
the distribution of $L^{(\ell)}$ in Proposition \ref{PropL} can be written in the form 
$Y^{(\ell)} M^{(\ell)} (Y^{(\ell)})^T$, where
$Y^{(\ell)} \sim \mathcal{N}_{\ell(9\ell-4)}(0,\Id_{\ell(9\ell-4)})$ and $M^{(\ell)} \in \mathrm{Mat}_{\ell(9\ell-4),\ell(9\ell-4)}(\R)$ is the deterministic matrix given by 
\begin{gather*}
M^{(\ell)} = \frac{-1}{50} \Id_{5\ell} 
\oplus \frac{-1}{25}\Id_{\frac{5\ell(\ell-1)}{2}}
\oplus \frac{1}{25}\Id_{\frac{5\ell(\ell-1)}{2}}
\oplus \frac{1}{50}\Id_{\frac{5\ell(\ell-1)}{2}}
\oplus \frac{-1}{6} \Id_{\frac{3\ell(\ell-1)}{2}} 
\ ,
\end{gather*}
with the convention that, $A \oplus 0=A$ for any matrix $A$.


\appendix
\section{Proof of Theorem \ref{MainThm1} and chaos expansion of level functionals}\label{AbsCanProof}

\subsection{Wiener-It\^{o} chaos expansion of $J(\mathbf{G},W;u^{(\ell)})$} 
We now provide the chaotic decomposition of the random variable $J(\mathbf{G},W;u^{(\ell)})$ introduced in Definition \ref{DefJ}. Informally, the latter is obtained by multiplying the respective chaotic expansions of $\prod_{i=1}^{\ell} \delta_{u_i}$ and $W$ and then integrating the obtained expression over $Z$.

\medskip
\noindent\underline{\textit{Formal chaotic expansion of the Dirac mass.}}
For $u\in \R$, denote by $\{\beta^{(u)}_j:j \geq 0\}$ the Hermite coefficients of the formal expansion in Hermite polynomials of $\delta_u$, that is
\begin{eqnarray*}
\delta_u(x) = \sum_{j\geq 0} \frac{\beta_j^{(u)}}{j!} H_j(x) \ , \quad x \in \R
\end{eqnarray*}
where 
\begin{eqnarray}\label{beta}
\beta^{(u)}_j = \int_{\R} \delta_u(y)H_j(y) \gamma(y)  dy = H_j(u) \gamma(u) \ .
\end{eqnarray} 
Approximating the Dirac mass by indicators $(2\eps)^{-1} \ind{[-\eps,\eps]}(x - u)$ for $\eps>0$ and denoting by
$\{\beta^{(u)}_j(\eps):j \geq 0\}$ their associated Fourier-Hermite coefficients, the following lemma (roughly corresponding to \cite{MPRW16}, Lemma 3.4) shows that the coefficients $\{\beta^{(u)}_j:j \geq 0\}$ in \eqref{beta} are obtained from $\{\beta^{(u)}_j(\eps):j \geq 0\}$ by letting $\eps \to 0$. 

\begin{Lem}\label{WCDirac}
For every $u\in \R$ and $\eps > 0$, the following expansion holds in $L^2(\gamma)$:
\begin{equation*}
\frac{1}{2\eps} \ind{[-\eps,\eps]}(x-u) = \sum_{j \geq 0} \frac{\beta_{j}^{(u)}(\eps)}{j!} H_j(x) \ ,  \quad x \in \R
\end{equation*}
where 
\begin{equation*}
\beta_0^{(u)}(\eps) = \frac{1}{2\eps}\int_{u-\eps}^{u+\eps} \gamma(y) dy \ ,
\end{equation*}
and for $j \geq 1$, 
\begin{equation}\label{betaepsfor}
\beta_j^{(u)}(\eps) = -\frac{1}{2\eps}\big(H_{j-1}(u+\eps)\gamma(u+\eps)-H_{j-1}(u-\eps)\gamma(u-\eps)\big) \ .
\end{equation}
In particular, for every $j \geq 0$, as $\eps \to 0$, 
\begin{eqnarray}\label{limbetaeps}
\beta^{(u)}_j(\eps) \to \beta^{(u)}_j \ .
\end{eqnarray}
\end{Lem}
For the nodal case corresponding to $u=0$, we write $\beta^{(0)}_j=:\beta_j$, and compute
\begin{equation*}
\beta_{2j+1} = 0 \ , \quad \beta_{2j} = \frac{H_{2j}(0)}{\sqrt{2\pi}} \ , \quad j \geq 0 \ , 
\end{equation*}
where the first equality is a consequence of the symmetry relation \eqref{Hsym}. In particular, we have 
\begin{equation}\label{beta024}
\beta_0 = \frac{1}{\sqrt{2\pi}} \ , \quad
\beta_2 = -\frac{1}{\sqrt{2\pi}} \ , \quad
\beta_4 = \frac{3}{\sqrt{2\pi}} \ .
\end{equation} 
The following standard proposition gives the Wiener-It\^{o} chaos expansion of $J(\mathbf{G},W;u^{(\ell)})$ defined in Definition \ref{DefJ}.
Its proof is based on the expansion of  $(2\eps)^{-\ell} \prod_{i=1}^{\ell}  \ind{[-\eps,\eps]}(\cdot-u_i)$ into Hermite polynomials by means of Lemma \ref{WCDirac} and then letting $\eps \to 0$. We omit the details.

\begin{Prop}\label{WCJProp}
Let the above setting prevail. Assume that the random field 
$W=\{W(z):z\in Z\}$ is such that (i) $\sup_{z\in Z} \E{W(z)^2}<\infty$, (ii) $W(z)$ is $\sigma(\mathbf{G})$-measurable for every $z\in Z$, and (iii) $W(z)$ is stochastically independent of $(G^{(1)}(z),\ldots,G^{(\ell)}(z))$ for every $z\in Z$. Then, the random variable
\begin{equation*}
J_{\eps}(\mathbf{G},W;u^{(\ell)}) :=
\int_Z (2\eps)^{-\ell} \prod_{i=1}^{\ell}  \ind{[-\eps,\eps]}(G^{(i)}(z)-u_i) \cdot W(z) \  \mu(dz)  
\end{equation*} 
is an element of $L^2(\Prob)$ for every $\eps>0$. Moreover, if $J(\mathbf{G},W;u^{(\ell)})$ as in \eqref{DefJ} is well-defined in $L^2(\Prob)$, then for every $q\geq 0$, 
\begin{eqnarray}\label{WCJ} 
\proj_q(J(\mathbf{G},W;u^{(\ell)})) 
= \sum_{\substack{j_1,\ldots,j_{\ell},r \geq 0 \\ j_1+\ldots+j_{\ell}+r=q} } \frac{\beta_{j_1}^{(u_1)}\cdots \beta_{j_{\ell}}^{(u_{\ell})}}{j_1! \cdots j_{\ell}!} \int_Z \prod_{i=1}^{\ell} H_{j_i}(G^{(i)}(z))\cdot \proj_r(W(z)) \ \mu(dz) \ ,
\end{eqnarray}
where $\{\beta^{(u)}_{j}: j \geq 0\}$ denote the coefficients of the formal Hermite expansion of  
$ \delta_{u}$ given in \eqref{beta}. 
\end{Prop}

\subsubsection{Some elementary facts}
Let $k\geq 1$ be an integer and $X=(X_1,\ldots,X_k)$  a standard $k$-dimensional Gaussian vector. We write $\norm{\cdot}_k$ 
to indicate the Euclidean norm in $\R^k$.  
We will need the following standard fact, whose proof is omitted.
\begin{Lem}\label{LemRU}
The random variable $\norm{X}_k$ is stochastically independent of $X/\norm{X}_k$.
\end{Lem}
For integers $1\leq \ell \leq  k $, we recall the notation introduced in \eqref{defalk}
\begin{equation*}  
\alpha(\ell,k):= \frac{(k)_{\ell}\kappa_k}{(2\pi)^{\ell/2}\kappa_{k-\ell}}   \ , 
\end{equation*}
where $(k)_{\ell}:=k!/(k-\ell)!$ and $\kappa_k:=\frac{\pi^{k/2}}{\Gamma(1+k/2)}$ stands for the volume of the unit ball in $\R^k$. 
The following lemma contains an expression of the moments of the Euclidean norm of a standard $k$-dimensional Gaussian vector.
\begin{Lem}\label{LemNorms}
For all integers $k\geq 1$ and $n\geq 1$, we have 
\begin{eqnarray}\label{moment}
\E{\norm{X}_k^n} 
= 2^{n/2} \frac{\Gamma((k+n)/2)}{\Gamma(k/2)} \ .
\end{eqnarray}
In particular, 
\begin{eqnarray} 
&&\E{\norm{X}_k} = \alpha(1,k) \ , \label{M1}\\
&&\E{\norm{X}_k^2} = k\ , \label{M2}\\
&&\E{\norm{X}_k^3} = \alpha(1,k)(k+1) \ ,\label{M3}\\
&&\E{\norm{X}_k^4} = k(k+2) \ , \label{M4}\\
&&\E{\norm{X}_k^5} = \alpha(1,k)(k+1)(k+3) \ ,\label{M5}
\end{eqnarray}
so that  
\begin{eqnarray}\label{Quotient}
\frac{\E{\norm{X}_k^3} }{\E{\norm{X}_k} } = k+1 \ .
\end{eqnarray}
\end{Lem}
\begin{proof} 
The law of the random variable $\norm{X}_k$ is the chi-distribution with $k$ degrees of freedom, whose density is given by (see e.g. \cite{W96} p.43)
\begin{equation*}
f(x) = \frac{1}{2^{k/2-1}\Gamma(k/2)} x^{k-1}e^{-x^2/2} \ , \qquad x > 0 \ .
\end{equation*}
Thus, it follows that, for $n \geq 1$, 
\begin{equation*}
\E{\norm{X}_k^n} = \int_{0}^{\infty} x^n f(x) dx 
= \frac{1}{2^{k/2-1}\Gamma(k/2)} 
\int_{0}^{\infty} x^{k+n-1} e^{-x^2/2}  dx \ .
\end{equation*}
Performing the change of variables $y = x^2/2$ yields 
\begin{equation*}
\E{\norm{X}_k^n} = \frac{1}{2^{k/2-1}\Gamma(k/2)} 
\cdot 2^{(k+n)/2-1}\Gamma((k+n)/2) 
= 2^{n/2} \frac{\Gamma((k+n)/2)}{\Gamma(k/2)} \ ,
\end{equation*}
which proves \eqref{moment}. The identities \eqref{M1}-\eqref{M5} are obtained from \eqref{moment} for $n = 1,\ldots,6$ respectively, together with the relations $\Gamma(z+1)=z\Gamma(z)$ and the definition in \eqref{defalk}.
\end{proof}

\subsubsection{Wiener-It\^{o} chaos expansion of 
$\Phi_{\ell,k}$ }
For integers $1\leq \ell \leq k$, we consider a generic map 
$\Phi_{\ell,k}$ as in Definition \ref{DefPhi} and a matrix $\bX=\big\{X_j^{(i)}:(i,j)\in [\ell]\times[k]\big\} \in \mathrm{Mat}_{\ell,k}(\R)$ with independent standard normal entries. 

\medskip
The next lemma provides a characterization of the second chaotic projection associated with 
$\bX$ and $\Phi_{\ell,k}(\bX)$, where we assume that $\E{\Phi_{\ell,k}(\bX)^2}<\infty$. 
As before, we set 
$\E{\Phi_{\ell,k}(\bX)}=: \alpha_{\ell,k}$.
\begin{Lem}\label{LemWC}
Let the above assumptions and notation prevail. Then, the following properties hold:
\begin{enumerate}[label=\rm{(\roman*)}]
\item for every $m\geq 1, (i_1,j_1),\ldots,(i_m,j_m)\in [\ell] \times [k]$ and $p_1,\ldots,p_m \in \N$ such that $p_1+\ldots+p_m$ is odd, we have 
\begin{equation*}
\E{\Phi_{\ell,k}(\bX)\prod_{a=1}^m H_{p_a}(X^{(i_a)}_{j_a})} = 0 \ ;
\end{equation*}
\item for every $(i_1,j_1) \neq (i_2,j_2) \in [\ell] \times [k]$, we have 
\begin{equation*}
\E{\Phi_{\ell,k}(\bX)X^{(i_1)}_{j_1}X^{(i_2)}_{j_2}} = 0 \ ;
\end{equation*}
\item for every $(i,j) \in [\ell] \times [k]$, we have 
\begin{equation*}
\E{\Phi_{\ell,k}(\bX)H_2(X^{(i)}_j)} = \frac{1}{k}\alpha_{\ell,k}\ .
\end{equation*}
\end{enumerate}
\end{Lem}
\begin{proof} 
Let us prove (i). Writing $p_1+\ldots+p_m=r$ and using the fact that $\bX \eqLaw -\bX$ together with property (A3) and the symmetry relation \eqref{Hsym}, we have
\begin{eqnarray*}
\E{\Phi_{\ell,k}(\bX)\prod_{a=1}^m H_{p_a}(X^{(i_a)}_{j_a})} = 
\E{\Phi_{\ell,k}(-\bX)\prod_{a=1}^m H_{p_a}(-X^{(i_a)}_{j_a})} 
= (-1)^{r}\E{\Phi_{\ell,k}(\bX)\prod_{a=1}^m H_{p_a}(X^{(i_a)}_{j_a})} ,
\end{eqnarray*}
which implies the claim. 
Let us now prove (ii). Assume first that $\ell \geq 2$ and $i_1 \neq i_2$. Let $\bX^*$ be the matrix obtained from $\bX$ by multiplying the $i_1$-th row by $-1$. Then, $\bX \eqLaw \bX^*$ together with (A2) applied with $c=-1$ imply  
\begin{eqnarray*}
J:=\E{\Phi_{\ell,k}(\bX)X^{(i_1)}_{j_1}X^{(i_2)}_{j_2}}  =
\E{\Phi_{\ell,k}(\bX^*)X^{*(i_1)}_{j_1}X^{*(i_2)}_{j_2}}   
= \E{\Phi_{\ell,k}(\bX)(-X^{(i_1)}_{j_1})X^{(i_2)}_{j_2}} = -J ,
\end{eqnarray*}
and therefore $J=0$. Assume now that $i_1=i_2$ (and therefore that $j_1\neq j_2$). Let $\bX^{**}$ be the matrix obtained from 
$\bX$ by multiplying the $j_1$-th column of $\bX$ by $-1$. Then, 
$\bX \eqLaw \bX^{**}$ together with (A3) imply  
\begin{eqnarray*}
J:=\E{\Phi_{\ell,k}(\bX)X^{(i_1)}_{j_1}X^{(i_2)}_{j_2}}  =
\E{\Phi_{\ell,k}(\bX^{**})X^{**(i_1)}_{j_1}X^{**(i_2)}_{j_2}}  
= \E{\Phi_{\ell,k}(\bX)(-X^{(i_1)}_{j_1})X^{(i_2)}_{j_2}} = -J ,
\end{eqnarray*}
which yields the desired conclusion. 
In order to prove (iii), let $\bX^*$ be the matrix obtained from $\bX$ by multiplying the $i$-th row by $c=1/\norm{X^{(i)}}_k$. Then, according to Lemma \ref{LemRU}, the $i$-th row of $\bX^*$ is stochastically independent of $\norm{X^{(i)}}_k$. We have
\begin{equation*}
\E{\Phi_{\ell,k}(\bX)H_2(X^{(i)}_j)}   
= \frac{1}{k}\E{\Phi_{\ell,k}(\bX)\norm{X^{(i)}}_k^2}-\E{\Phi_{\ell,k}(\bX)} \ ,
\end{equation*}
so that, using (A2) and the independence mentioned above, yields
\begin{eqnarray*}
\E{\Phi_{\ell,k}(\bX)H_2(X^{(i)}_j)}&=& \frac{1}{k}\E{\Phi_{\ell,k}(\bX^*)\norm{X^{(i)}}_k^3} -\E{\Phi_{\ell,k}(\bX)} \\
&=& \frac{1}{k}\frac{\E{\Phi_{\ell,k}(\bX)}}{\E{\norm{X^{(i)}}_k}} \E{\norm{X^{(i)}}_k^3} - \E{\Phi_{\ell,k}(\bX)} = \frac{1}{k}\E{\Phi_{\ell,k}(\bX)} = \frac{1}{k}\alpha_{\ell,k}\ ,
\end{eqnarray*}
where the last equality follows from \eqref{Quotient}. 
\end{proof}

The following proposition combines Lemma \ref{LemWC} with the classical general formula for the chaotic projections of all order of $\Phi_{\ell,k}(\bX)$.

\begin{Prop}\label{WCPhi}
Let $\Phi_{\ell,k}:\mathrm{Mat}_{\ell,k}(\R)  \to \R_+$ be as in the previous lemma. Then, for $q\geq 0$, the projection of $\Phi_{\ell,k}(\bX)$ onto the $q$-th Wiener chaos associated with $\bX$ is given by  
\begin{eqnarray*}
\proj_q(\Phi_{\ell,k}(\bX))   
= 
\mathop{\sum_{p^{(1)}_1,\ldots,p^{(1)}_k\geq0}
\ldots \sum_{p^{(\ell)}_1,\ldots,p^{(\ell)}_k\geq0}}_{p^{(1)}_1+\ldots+p^{(1)}_k+\ldots+p^{(\ell)}_1+\ldots+p^{(\ell)}_k=q}
\alpha^{(\ell)}_k\left\{p^{(i)}_j: (i,j) \in [\ell]\times [k] \right\}\cdot
\prod_{i=1}^{\ell} \prod_{j=1}^k  H_{p^{(i)}_j}(X^{(i)}_j)  \ , \end{eqnarray*}
where the coefficients $\alpha^{(\ell)}_k\left\{p^{(i)}_j: (i,j) \in [\ell]\times [k] \right\}$ are given by
\begin{equation}\label{alpha}
\alpha^{(\ell)}_k\left\{p^{(i)}_j: (i,j) \in [\ell]\times [k] \right\} := 
\frac{1}{\prod_{i=1}^{\ell}\prod_{j=1}^{k}(p^{(i)}_j)! } 
\cdot \E{\Phi_{\ell,k}(\bX) \cdot \prod_{i=1}^{\ell} \prod_{j=1}^k  H_{p^{(i)}_j}(X^{(i)}_j)} \ .
\end{equation}
In particular, we have
\begin{eqnarray}
&& \proj_0(\Phi_{\ell,k}(\bX)) = \E{\Phi_{\ell,k}(\bX)} = \alpha_{\ell,k} \ , \label{E1} \\
&& \proj_2(\Phi_{\ell,k}(\bX)) = 
\frac{\alpha_{\ell,k} }{2} \cdot \frac{1}{k} \sum_{i=1}^{\ell} \sum_{j=1}^k 
\big((X^{(i)}_j)^2-1\big) \ , \label{E2} \\ 
&& \proj_{2q+1}(\Phi_{\ell,k}(\bX)) = 0 \ , \quad q\geq0 \ . \label{E3}
\end{eqnarray}
\end{Prop}
\begin{proof}
The formula for $\proj_q(\Phi_{\ell,k}(\bX))$ follows from the orthogonal decomposition of $L^2(\Prob)$. For $q=0$, we have $p^{(i)}_j=0$ for every $ (i,j) \in [\ell]\times [k]$, so that $\proj_0(\Phi_{\ell,k}(\bX))=\E{\Phi_{\ell,k}(\bX)}$. For $q=2$, in view of Lemma \ref{LemWC}, only the tuples 
$(p^{(i)}_j: (i,j) \in [\ell]\times [k])$ involving exactly one $2$ contribute to the projection on the second chaos and the conclusion then follows from Lemma \ref{LemWC} (iii). Finally, the projections onto Wiener chaoses of odd order vanish in view of Lemma \ref{LemWC}  (i). 
\end{proof}

\subsection{Proof of Theorem \ref{MainThm1} }\label{genstat}
Part (i) follows from the form of the $q$-th chaotic projection of $J$ provided in \eqref{WCJ} and Proposition \ref{WCPhi} where the random matrix $\mathbf{X}$ is replaced with $\bX_{\star}(z)$. Indeed, by \eqref{E1} and the fact that $\mu(Z)=1$, we have 
\begin{equation*}
\proj_0(J)= \beta_0^{(u_1)}\cdots \beta_0^{(u_{\ell})} \int_Z 
\prod_{i=1}^{\ell} H_0(X_0^{(i)}(z)) \cdot \proj_0(\Phi_{\ell,k}(\bX_{\star}(z))) \ \mu(dz) = \prod_{i=1}^{\ell} \gamma(u_i) \cdot \alpha_{\ell,k} \ .
\end{equation*}
This proves \eqref{P0i}. For \eqref{P1i}, since $\proj_1(\Phi_{\ell,k}(\bX_{\star}(z)))=0$ by \eqref{E3}, we have (recalling the definition of $m^{(i)}$ in \eqref{mi})
\begin{eqnarray*}
\proj_1(J) &=& \sum_{i=1}^{\ell} \beta_1^{(u_i)} \prod_{\substack{j=1 \\ j \neq i}}^{\ell}
\beta_0^{(u_j)} \int_Z H_0(X_0^{(j)}(z))H_1(X_0^{(i)}(z)) \cdot \proj_0(\Phi_{\ell,k}(\bX_{\star}(z))) \ \mu(dz) \\
&=& \sum_{i=1}^{\ell} \prod_{\substack{j=1 \\ j \neq i}}^{\ell}
\gamma(u_j)\gamma(u_i)u_i \int_Z X_0^{(i)}(z) \cdot \alpha_{\ell,k} \ \mu(dz) = \prod_{j=1}^{\ell} \gamma(u_j) \cdot \alpha_{\ell,k} \cdot \sum_{i=1}^{\ell} m^{(i)} u_i \ .
\end{eqnarray*}
Let us now turn to \eqref{P2i}. We have
\begin{eqnarray*}
\proj_2(J)&=& \sum_{i=1}^{\ell} \frac{\beta_2^{(u_i)}}{2!}\prod_{\substack{j=1 \\ j \neq i}}^{\ell}
\beta_0^{(u_j)} \int_Z H_0(X_0^{(j)}(z))H_2(X_0^{(i)}(z)) \cdot \proj_0(\Phi_{\ell,k}(\bX_{\star}(z))) \ \mu(dz) \\
&& \hspace{2cm} \quad + \prod_{i=1}^{\ell} \beta_0^{(u_i)}
\int_Z H_0(X_0^{(i)}(z)) \cdot \proj_2(\Phi_{\ell,k}(\bX_{\star}(z))) \ \mu(dz) \ . 
\end{eqnarray*}
Now, using $\beta_2^{(u_i)}=\gamma(u_i)(u_i^2-1)$ and \eqref{E2} yields
\begin{eqnarray*}
\proj_2(J) &=&\frac{\alpha_{\ell,k}}{2} \cdot \prod_{j=1}^{\ell} \gamma(u_j) \cdot \sum_{i=1}^{\ell} (u_i^2-1) 
\int_Z (X_0^{(i)}(z)^2-1) \ \mu(dz) \\
&& \hspace{2cm} \quad + \frac{\alpha_{\ell,k}}{2} \cdot \prod_{i=1}^{\ell} \gamma(u_i)
\int_Z  \frac{1}{k}\sum_{i=1}^{\ell} \sum_{j=1}^k (X_j^{(i)}(z)^2-1) \ \mu(dz) \\ 
&=& \frac{\alpha_{\ell,k}}{2} \cdot \prod_{i=1}^{\ell} \gamma(u_i) \cdot 
\sum_{i=1}^{\ell} \biggl\{ (u_i^2 -1) \int_Z (X_0^{(i)}(z)^2-1)    
+ \frac{1}{k} \sum_{j=1}^k (X_j^{(i)}(z)^2-1) \ \mu(dz) \biggr\} \\
&=& \frac{\alpha_{\ell,k}}{2} \cdot \prod_{i=1}^{\ell} \gamma(u_i) \cdot 
\sum_{i=1}^{\ell} \biggl\{ u_i^2  \int_Z (X_0^{(i)}(z)^2-1) \ \mu(dz) + D^{(i)} \biggr\} \ ,
\end{eqnarray*}
where we used the definition of $D^{(i)}$ in \eqref{Di}.

\noindent
For part (ii), set $u_i=D^{(i)}=0$ for every $i\in[\ell]$. Then, \eqref{P0ii} follows since $\gamma(0)=1/\sqrt{2\pi}$. By \eqref{P2i}, we have that $\proj_2(J) = 0$. It remains to show that $\proj_{2q+1}(J)=0$ for $q\geq 0$. The fact that $\beta^{(0)}_{2k+1}=0$ for every $k\geq 0$ implies that the expansion in \eqref{WCJ} runs over indices $j_1,\ldots, j_{\ell}$ that are all even. The projection of $J$ onto Wiener chaoses of odd order is therefore of the form 
\begin{equation*} 
\proj_{2q+1}(J)= \sum_{\substack{j_1,\ldots,j_{\ell},r \geq 0 \\ j_1+\ldots+j_{\ell}+r=2q+1} } \frac{\beta_{j_1}^{(0)}\cdots \beta_{j_{\ell}}^{(0)}}{j_1! \cdots j_{\ell}!} \int_Z \prod_{i=1}^{\ell} H_{j_i}(X_0^{(i)}(z))\cdot \proj_r(\Phi_{\ell,k}(\bX_{\star}(z))) \ \mu(dz) \ ,
\end{equation*}
where $j_1,\ldots,j_{\ell}$ are all even and $r$ is odd. The conclusion then follows from \eqref{E3}.

\section{Fourier-Hermite coefficients of Gramian determinants on the fourth Wiener chaos}\label{Constants4}
For integers $1\leq \ell \leq k$ and a $\ell \times k$ matrix $\bX$ with i.i.d. standard normal entries, we consider the function
\begin{equation}\label{detXX} 
\Phi_{\ell,k}^*: \mathrm{Mat}_{\ell,k}(\R) \to \R_+ \ , \quad \bX \mapsto  \det(\bX \bX^T)^{1/2} \ . 
\end{equation} 
The following lemma shows that $\Phi_{\ell,k}^*$  defined in \eqref{detXX} satisfies \textit{Assumption A} of Definition \ref{DefPhi}. In order to prove this, we recall Cauchy-Binet's identity:
\begin{equation}\label{CB}
\Phi_{\ell,k}^*(\bX) =\left[\sum_{j_1<\ldots<j_{\ell} \in [k]} \det (\bX_{j_1,\ldots,j_{\ell}})^2 \right]^{1/2} \ ,
\end{equation}
where, for $j_1<\ldots<j_{\ell} \in [k]$, we denote by $\bX_{j_1,\ldots,j_{\ell}} \in \mathrm{Mat}_{\ell,\ell}(\R)$ the matrix obtained from $\bX$ by only keeping columns labeled $j_1,\ldots,j_{\ell}$. We refer to $\det(\mathbf{X}_{j_1,\ldots,j_{\ell}})$ as the \textit{minors} of $\mathbf{X}$.

\begin{Lem}\label{CauchyBinet}
The function $\Phi_{\ell,k}^*$  in \eqref{detXX} satisfies Assumption A of Definition \ref{DefPhi}.
\end{Lem}
\begin{proof}
\begin{enumerate}[label=\rm{(\roman*)}]
\item[(A1)] Permuting two columns multiplies some of the minors by $-1$, which is absorbed by taking its square. Permuting two rows multiplies each minor by $-1$, which is again absorbed by taking its square.
\item[(A2)] Let $\bX^*$ denote the matrix obtained from $\bX$ by multiplying the $i$-th row by $c\in \R$. Then, for every $j_1<\ldots<j_{\ell}\in [k]$, we have
$\det (\bX_{j_1,\ldots,j_{\ell}}^*)^2 = c^2\det (\bX_{j_1,\ldots,j_{\ell}})^2 $, so that \eqref{CB} implies  $\Phi_{\ell,k}^*(\bX^*) = |c|\Phi_{\ell,k}^*(\bX) $.
\item[(A3)] Let $\bX^*$ denote the matrix obtained from $\bX$ by multiplying its $j$-th column by $-1$. Then, $\bX^*(\bX^*)^T = \bX \bX^T$, so that trivially $\Phi_{\ell,k}^*(\bX) = \Phi_{\ell,k}^*(\bX^*)$. 
\item[(A4)] Let $\bX^*$ denote the matrix obtained from $\bX$ by replacing its $i_1$-th row with the sum of its $i_1$-th and $i_2$-th row for $i_1 \neq i_2$. Then, the invariance of the determinant under this operation implies that for every $j_1<\ldots<j_{\ell}\in [k],\det(\bX_{j_1,\ldots,j_{\ell}}^*) = \det(\bX_{j_1,\ldots,j_{\ell}})$, so that $\Phi_{\ell,k}^*(\bX^*) = \Phi_{\ell,k}^*(\bX)$. 
\end{enumerate}
\end{proof}

\subsection{A representation of the Gramian determinant}
In the forthcoming discussion, our goal is to 
compute the Fourier-Hermite coefficients within the fourth Wiener chaos associated with the function $\Phi_{\ell,k}^*$ in \eqref{detXX}. 

\medskip
We first prove a deterministic result. Let $v^{(1)},\ldots,v^{(\ell)} \in \R^k$ be linearly independent vectors and $\bX$ the $\ell\times k$ matrix whose $i$-th row is $v^{(i)}$. For $s=0,\ldots, \ell-1$, we write $\mathscr{V}_s:=\mathrm{span}\{v^{(1)},\ldots,v^{(s)}\}$ to indicate the $s$-dimensional linear subspace generated by the first $s$ rows of $\bX$ with the convention $\mathscr{V}_0:=\{0\}$ and denote by $p_s$ the projection operator onto $\mathscr{V}_s$. Furthermore, we set
\begin{equation*}
d(k-s) := \norm{v^{(s+1)}-p_s(v^{(s+1)})}_k \ , \qquad 
s=0, \ldots, \ell-1 \ , 
\end{equation*}
that is, $d(k-s)$ is the Euclidean distance in $\R^k$ between $v^{(s+1)}$ and $\mathscr{V}_s$. 
The next lemma yields a useful representation of Gramian determinants.
\begin{Lem}\label{det}
Let the above notation prevail. 
Then, the map $\Phi_{\ell,k}^*$  in \eqref{detXX} admits the representation
\begin{equation}\label{DetProd}
\Phi_{\ell,k}^*(\bX) = \prod_{s=0}^{\ell-1} d(k-s) \ . 
\end{equation}
\end{Lem}
\begin{proof}
Applying the Gram-Schmidt orthogonalization process to the vectors $\{v^{(1)},\ldots,v^{(\ell)}\}$ gives rise to a family of orthogonal vectors $\{w^{(1)}, \ldots, w^{(\ell)}\}$ such that 
$\mathrm{span}\{w^{(1)},\ldots, w^{(\ell)}\} = \mathrm{span}\{v^{(1)},\ldots,v^{(\ell)}\}$. These are given recursively by $w^{(1)} = v^{(1)}$ and for $s=1,\ldots,\ell-1$,
\begin{equation*}
w^{(s+1)} = v^{(s+1)} - \sum_{i=1}^{s} 
\frac{\scal{v^{(s+1)}}{w^{(i)}}}{\norm{w^{(i)}}_k^2}w^{(i)} 
= v^{(s+1)} -p_s(v^{(s+1)})  \ ,
\end{equation*}
where $\scal{\cdot}{\cdot}$ denotes the canonical inner product in $\R^d$.
Denote by $\mathbf{W}$ the $\ell\times k$ matrix with rows $w^{(1)},\ldots,w^{(\ell)}$. There exists an orthogonal $\ell \times \ell$ matrix $P$ such that $\mathbf{W}=P\bX$, which implies that 
$\mathbf{W}\mathbf{W}^T = P\bX \bX^T P^T$, so that $\Phi_{\ell,k}^*(\mathbf{W}) = \Phi_{\ell,k}^*(\bX)$. 
As the rows of $\mathbf{W}$ are mutually orthogonal, we have that 
\begin{eqnarray*}
\mathbf{W}\mathbf{W}^T = \mathrm{diag}\left(\norm{w^{(1)}}_k^2, \ldots, \norm{w^{(\ell)}}_k^2\right)= 
\mathrm{diag}\left(d(k)^2,\ldots,d(k-(\ell-1))^2\right)  ,
\end{eqnarray*}
and therefore,  
\begin{equation*}
\Phi_{\ell,k}^*(\mathbf{W}) = \prod_{s=0}^{\ell-1} d(k-s)  , 
\end{equation*}
which is formula \eqref{DetProd}. 
\end{proof}
We will now pass to the probabilistic setting and replace each of the deterministic vectors $v^{(1)},\ldots,v^{(\ell)}$ by independent standard Gaussian vectors $X^{(1)},\ldots,X^{(\ell)}$. The following lemma characterizes the probability distribution of the random variables $d(k-s)$.
\begin{Lem}\label{indep}
Let the above setting prevail. For every $s=0,\ldots, \ell-1$, the random variable $d(k-s)$ is chi-distributed with $k-s$ degrees of freedom and stochastically independent of $(X^{(1)},\ldots,X^{(s)})$.  In particular, 
\begin{equation}\label{mean}
\alpha_{\ell,k}:=\E{\Phi_{\ell,k}^*(\bX)} = \prod_{s=0}^{\ell-1} \E{d(k-s)} = \alpha(\ell,k) \ ,
\end{equation}
where $\alpha(\ell,k)$ is  defined in \eqref{defalk}. 
\end{Lem}
\begin{proof}
Let $\{\mathbf{e}_1,\ldots, \mathbf{e}_k\}$ denote the canonical basis of $\R^k$. Since $d(k)=\norm{X^{(1)}}_k$, the random variable $d(k)$ is clearly chi-distributed with $k$ degrees of freedom. Now fix $s \in \{1,\ldots, \ell-1\}$. By the rotational invariance of the Gaussian distribution, the conditional distribution of $d(k-s)$ given $\{X^{(1)},\ldots,X^{(s)}\}$ is precisely the same as the distribution of the distance from $X^{(s+1)}$ to $\R^s$, that is 
\begin{equation*}
d(k-s)|\{X^{(1)},\ldots,X^{(s)}\} \eqLaw \bigg(\sum_{j=s+1}^k \scal{X^{(s+1)}}{\mathbf{e}_j}^2 \bigg)^{1/2} \ .
\end{equation*}
Since the coefficients $\scal{X^{(s+1)}}{\mathbf{e}_j}=X^{(s+1)}_j$ are i.i.d. standard Gaussian, we infer that $d(k-s)|\{X^{(1)},\ldots,X^{(s)}\}$ is chi-distributed with $k-s$ degrees of freedom. Thus the characteristic function of $d(k-s)^2|\{X^{(1)},\ldots,X^{(s)}\}$ is
\begin{equation*}
\phi_{d(k-s)^2|\{X^{(1)},\ldots,X^{(s)}\}}(t) = \E{e^{itd(k-s)^2}|X^{(1)},\ldots,X^{(s)}} = (1-2it)^{-(k-s)/2} \ , \quad t \in \R.
\end{equation*}
Therefore, taking expectation
\begin{equation*}
\phi_{d(k-s)^2}(t) = \E{e^{itd(k-s)^2}} = \E{\E{e^{itd(k-s)^2}|X^{(1)},\ldots,X^{(s)}}} = (1-2it)^{-(k-s)/2} \ ,
\end{equation*}
from which we conclude that $d(k-s)$ is also chi-distributed with $k-s$ degrees of freedom. Moreover, since $d(k-s)|\{X^{(1)},\ldots,X^{(s)}\} \eqLaw d(k-s)$, we deduce that $d(k-s)$ is independent of $\{X^{(1)},\ldots,X^{(s)}\}$.  The identity in \eqref{mean} follows from 
independence, and the fact that by \eqref{M1}, $\E{d(k-s)} = \alpha(1,k-s)$:
\begin{equation*} 
\alpha_{\ell,k} = \E{\Phi_{\ell,k}^*(\bX)}  
= \prod_{s=0}^{\ell-1}\E{d(k-s)}
= \prod_{s=0}^{\ell-1} \alpha(1,k-s) 
= \prod_{s=0}^{\ell-1} \frac{(k-s)\kappa_{k-s}}{\sqrt{2\pi} \kappa_{k-s-1}} = \alpha(\ell,k) \ ,
\end{equation*}
which finishes the proof.
\end{proof}


\subsection{Technical computations}
The following result entirely characterizes the fourth chaotic component of the function $\Phi_{\ell,k}^*(\bX)$ defined in \eqref{detXX} where $\bX$ is a $\ell \times k$ matrix with i.i.d. standard Gaussian entries.
\begin{Lem}\label{LemMonomials}
Let the above notations prevail. The following properties hold:
\begin{enumerate}[label=\rm{(\roman*)}]
\item for every  $(i,j) \in [\ell]\times [k]$, we have 
\begin{equation*}
\E{\Phi_{\ell,k}^*(\bX) (X^{(i)}_j)^4} = 3\alpha(\ell,k) \frac{(k+1)(k+3)}{k(k+2)} \ ;
\end{equation*}

\item for every $(i_1,j_1) \neq (i_2,j_2)  \in [\ell] \times [k]$, we have 
\begin{equation*}
\E{\Phi_{\ell,k}^*(\bX) (X^{(i_1)}_{j_1})^3 X^{(i_2)}_{j_2} } = 0 \ , 
\end{equation*}

\item for every $(i_1,j_1) \neq (i_2,j_2) \neq (i_3,j_3) \in [\ell] \times [k]$, we have 
\begin{equation*}
\E{\Phi_{\ell,k}^*(\bX) (X^{(i_1)}_{j_1})^2 X^{(i_2)}_{j_2} X^{(i_3)}_{j_3}} = 0 \ , 
\end{equation*}

\item 
for every  $(i_1,j_1)\neq(i_2,j_2)\in [\ell]\times [k]$, we have 
\begin{eqnarray*}
\E{\Phi_{\ell,k}^*(\bX)(X^{(i_1)}_{j_1})^2(X^{(i_2)}_{j_2})^2}  =
\alpha(\ell,k)\frac{(k+1)(k+3)}{k(k+2)} \ind{i_1=i_2} \\
+ \ \alpha(\ell,k) \frac{(k+1)(k+3)}{k(k+2)}\ind{i_1\neq i_2,j_1=j_2}\ind{\ell\geq2} \\
+ \ \alpha(\ell,k) (k+1) \frac{(k+1)(k+2)-(k+3)}{k(k-1)(k+2)}  \ind{i_1\neq i_2,j_1 \neq j_2} \ind{\ell\geq2}\ ; 
\end{eqnarray*}

\item for every collection 
$I = \{(i_1,j_1)\neq(i_2,j_2)\neq(i_3,j_3)\neq(i_4,j_4) \in [\ell]\times [k]\} $, we have 
\begin{equation*}
\E{\Phi_{\ell,k}^*(\bX)X^{(i_1)}_{j_1}X^{(i_2)}_{j_2}X^{(i_3)}_{j_3}X^{(i_4)}_{j_4}} 
=   -\alpha(\ell,k)\frac{k+1}{k(k-1)(k+2)}  \ind{I \in S}\ind{\ell\geq2} \ ,
\end{equation*}
where $
S=\{ \{(i_1,j_1),(i_1,j_2),(i_2,j_1),(i_2,j_2)\}: i_1 \neq i_2, j_1 \neq j_2\}$.

\end{enumerate}
\end{Lem}
\begin{proof}
We prove (i). By (A1), without loss of generality, we can assume that $i=1$. Using the representation in \eqref{DetProd}, the fact that $\norm{X^{(1)}}_k=d(k)$, as well Lemma \ref{LemRU} and \eqref{mean}, we have for every $j \in [k]$, 
\begin{eqnarray*}
\E{\Phi_{\ell,k}^*(\bX) (X^{(1)}_j)^4} &=& \E{d(k)\prod_{s=1}^{\ell-1} d(k-s) \frac{ (X^{(1)}_j)^4}{\norm{X^{(1)}}^4_k} \norm{X^{(1)}}^4_k} 
= \E{d(k)^5 \prod_{s=1}^{\ell-1}d(k-s) \frac{ (X^{(1)}_j)^4}{d(k)^4} } 
\\
&=& \frac{\E{d(k)^5}}{\E{d(k)^4}} \prod_{s=1}^{\ell-1}\E{d(k-s)} \E{ (X^{(1)}_j)^4} 
= 3\alpha(\ell,k) \frac{(k+1)(k+3)}{k(k+2)} \ ,
\end{eqnarray*}
where the last equality follows from Lemma \ref{LemNorms}. 

\noindent
We now prove (ii). Assume $i_1=i_2$ (so that $j_1\neq j_2$). Multiplying column $j_2$ by $-1$ and using (A3) then yields the desired conclusion. If $i_1 \neq i_2$ and $j_1=j_2$, the result follows from (A2). The case $i_1\neq i_2,j_1\neq j_2$ follows either from (A3) or (A2). 

\noindent
The result in (iii) is obtained by arguments similar those  in (ii).

\noindent
For (iv), let us assume that $i_1=i_2$ (so that $j_1\neq j_2$).  Denote by $\bX^*$ the matrix obtained from $\bX$ by multiplying the $i_1$-th row by $1/\norm{X^{(i_1)}}_k$. Then, we first observe that by  (A2)  and Lemma \ref{LemRU},
\begin{eqnarray*}
\E{\Phi_{\ell,k}^*(\bX)\norm{X^{(i_1)}}_k^4} 
&= & \E{\Phi_{\ell,k}^*(\bX^*)\norm{X^{(i_1)}}_k^5 } = \E{\Phi_{\ell,k}^*(\bX^*)} \E{\norm{X^{(i_1)}}_k^5} \\
&=& \frac{\E{\Phi_{\ell,k}^*(\bX)}}{\E{\norm{X^{(i_1)}}_k}} \E{\norm{X^{(i_1)}}_k^5} 
= \alpha(\ell,k) (k+1)(k+3) \ ,
\end{eqnarray*}
where we used Lemma \ref{LemNorms}. 
On the other hand, we can write 
\begin{eqnarray*}
\E{\Phi_{\ell,k}^*(\bX)\norm{X^{(i_1)}}_k^4} &=& 
\E{\Phi_{\ell,k}^*(\bX) \sum_{j,j'=1}^k (X^{(i_1)}_{j})^2 (X^{(i_1)}_{j'})^2 }\\
&=& \sum_{j=1}^k \E{\Phi_{\ell,k}^*(\bX) (X^{(i_1)}_{j })^4} 
+ \sum_{j \neq j' \in [k]} \E{\Phi_{\ell,k}^*(\bX) (X^{(i_1)}_{j})^2 (X^{(i_1)}_{j'})^2}\\
&=& k \ \E{\Phi_{\ell,k}^*(\bX) (X^{(i_1)}_{j_1})^4} + k(k-1) \E{\Phi_{\ell,k}^*(\bX) 
(X^{(i_1)}_{j_1})^2 (X^{(i_1)}_{j_2})^2} \\
&=&  3\alpha(\ell,k) \frac{(k+1)(k+3)}{k+2} 
+ k(k-1) \E{\Phi_{\ell,k}^*(\bX) 
(X^{(i_1)}_{j_1})^2 (X^{(i_1)}_{j_2})^2}\ ,
\end{eqnarray*}
for every $j_1\neq j_2$,
where for the last equality we used the formula proved in (i). Therefore, it follows that for every $j_1\neq j_2$,
\begin{eqnarray*}
\E{\Phi_{\ell,k}^*(\bX) (X^{(i_1)}_{j_1})^2 (X^{(i_1)}_{j_2})^2} &=& 
\frac{1}{k(k-1)}\left( \E{\Phi_{\ell,k}^*(\bX) \norm{X^{(i_1)}}_k^4} 
- 3\alpha(\ell,k)\frac{(k+1)(k+3)}{k+2} \right)\\
&=& \frac{1}{k(k-1)}\left( \alpha(\ell,k) (k+1)(k+3) 
- 3\alpha(\ell,k)\frac{(k+1)(k+3)}{k+2} \right) \\
&=& \alpha(\ell,k)\frac{(k+1)(k+3)}{k(k+2)}  .
\end{eqnarray*}
\noindent
Let us now deal with the case $i_1 \neq i_2$ and $j_1 = j_2$, for $\ell\geq 2$. Denote by $\bX_{\pm}$ the matrix obtained from $\bX$ as follows:
\begin{gather*}
(X_{\pm})^{(i_1)} = \frac{1}{\sqrt{2}}(X^{(i_1)}+X^{(i_2)})  \ , \\
(X_{\pm})^{(i_2)} = \frac{1}{\sqrt{2}}\left(-2X^{(i_2)}+(X^{(i_1)}+X^{(i_2)})\right) = 
\frac{1}{\sqrt{2}}(X^{(i_1)}-X^{(i_2)}) \ , \\
(X_{\pm})^{(i)} = X^{(i)} \ ,\quad  i \in [\ell] \setminus \{i_1,i_2\}  \ .
\end{gather*}
By construction, the rows $(X_{\pm})^{(i_1)}$ and $(X_{\pm})^{(i_2)}$ are stochastically independent standard Gaussian vectors, so that    $\bX\eqLaw \bX_{\pm}$. Hence, we have on the one hand
\begin{equation*}
\E{\Phi_{\ell,k}^*(\bX_{\pm})\big((X_{\pm})^{(i_1)}_{j_1}\big)^4} 
= \E{\Phi_{\ell,k}^*(\bX) (X^{(i_1)}_{j_1})^4} = 3\alpha(\ell,k)\frac{(k+1)(k+3)}{k(k+2)} \ ,
\end{equation*}
in view of (i), and on the other hand, since
$\Phi_{\ell,k}^*(\bX_{\pm}) = \frac{(\sqrt{2})^2}{2} \Phi_{\ell,k}^*(\bX) = \Phi_{\ell,k}^*(\bX)$, we conclude
\begin{gather*}
\E{\Phi_{\ell,k}^*(\bX_{\pm})\big((X_{\pm})^{(i_1)}_{j_1}\big)^4} 
=
\E{\Phi_{\ell,k}^*(\bX)\bigg( \frac{X^{(i_1)}_{j_1} + X^{(i_2)}_{j_1} }{\sqrt{2}} \bigg)^4} \\
= \frac{1}{4}\left( 
2 \E{\Phi_{\ell,k}^*(\bX)( X^{(i_1)}_{j_1} )^4}
+ 6 \E{\Phi_{\ell,k}^*(\bX)( X^{(i_1)}_{j_1} )^2 ( X^{(i_2)}_{j_1} )^2} 
\right) \ ,
\end{gather*}
where we used that $\E{\Phi_{\ell,k}^*(\bX) X^{(i_1)}_{j_1} (X^{(i_2)}_{j_1})^3} 
= \E{\Phi_{\ell,k}^*(\bX) (X^{(i_1)}_{j_1})^3 X^{(i_2)}_{j_1}}=0$ in view of (ii). Therefore, 
\begin{eqnarray*}
\E{\Phi_{\ell,k}^*(\bX)( X^{(i_1)}_{j_1} )^2 ( X^{(i_2)}_{j_1} )^2}  
&=& \frac{1}{6} \left( 4\E{\Phi_{\ell,k}^*(\bX_{\pm})\big((X_{\pm})^{(i_1)}_{j_1}\big)^4} - 2\E{\Phi_{\ell,k}^*(\bX)( X^{(i_1)}_{j_1} )^4} \right) \\
&=& \alpha(\ell,k)\frac{(k+1)(k+3)}{k(k+2)} \ . 
\end{eqnarray*}

\noindent
Let us now treat the case $i_1 \neq i_2$ and $j_1 \neq j_2$. Let $\bX^*$ be the matrix obtained from $\bX$ by multiplying rows $X^{(i_1)}$ resp. $X^{(i_2)}$ by $1/\norm{X^{(i_1)}}_k$ resp. $1/\norm{X^{(i_2)}}_k$. Then, by independence, we infer  
\begin{eqnarray*}
&&\E{\Phi_{\ell,k}^*(\bX) \norm{X^{(i_1)}}_k^2 \norm{X^{(i_2)}}_k^2} 
= \E{\Phi_{\ell,k}^*(\bX^*)\norm{X^{(i_1)}}_k^3 \norm{X^{(i_2)}}_k^3} 
\\
&=& \E{\Phi_{\ell,k}^*(\bX^*)} \E{\norm{X^{(i_1)}}_k^3} \E{\norm{X^{(i_2)}}_k^3} = 
\frac{\E{\Phi_{\ell,k}^*(\bX)}}{\E{\norm{X^{(i_1)}}_k}^2}\E{\norm{X^{(i_1)}}_k^3}^2 \\
&=& \alpha(\ell,k)(k+1)^2 \ .
\end{eqnarray*}
Expanding the product of the norms, we can write 
\begin{eqnarray*}
&&\E{\Phi_{\ell,k}^*(\bX)\norm{X^{(i_1)}}_k^2 \norm{X^{(i_2)}}_k^2}\\ 
&=& k \ \E{\Phi_{\ell,k}^*(\bX) (X^{(i_1)}_{j_1})^2 (X^{(i_2)}_{j_1})^2} 
+ k(k-1) \E{\Phi_{\ell,k}^*(\bX) (X^{(i_1)}_{j_1})^2 (X^{(i_2)}_{j_2})^2} \\
&=&  \alpha(\ell,k) \frac{(k+1)(k+3)}{k+2} 
+ k(k-1) \E{\Phi_{\ell,k}^*(\bX) (X^{(i_1)}_{j_1})^2 (X^{(i_2)}_{j_2})^2}\ ,
\end{eqnarray*}
where we used the formula proved just before. Hence, we have that for every $j_1 \neq j_2$, 
\begin{eqnarray*}
\E{\Phi_{\ell,k}^*(\bX) (X^{(i_1)}_{j_1})^2 (X^{(i_2)}_{j_2})^2} 
&=& \frac{1}{k(k-1)}\left(\alpha(\ell,k)(k+1)^2 - \alpha(\ell,k) \frac{(k+1)(k+3)}{k+2}\right) \\
&=& \alpha(\ell,k) (k+1) \frac{(k+1)(k+2)-(k+3)}{k(k-1)(k+2)} \ ,
\end{eqnarray*}
which is the desired formula. The other cases do not contribute as one can mutliply a row or column by $-1$.

\noindent
We finally prove (v). First, note that if $I  \notin S$, then the expectation is zero. Indeed, we notice that if $I \notin S$, there is at least one row or column of $\bX$ that contains only one element corresponding to one of the four pairs of indices of $I$. Multiplying this row resp. column by $-1$ and using (A2) gives the desired conclusion. 
Let us now assume $I \in S$ and denote $E(I):=\E{\Phi_{\ell,k}^*(\bX)X^{(i_1)}_{j_1}X^{(i_2)}_{j_2}X^{(i_3)}_{j_3}X^{(i_4)}_{j_4}} 
\ind{I\in S}$. 
Since $I \in S$, we can write 
\begin{equation*}
E(I)= \E{\Phi_{\ell,k}^*(\bX) X^{(i_1)}_{j_1}X^{(i_1)}_{j_2}X^{(i_2)}_{j_1}X^{(i_2)}_{j_2}}\ , \quad i_1 \neq i_2, j_1 \neq j_2 .
\end{equation*} 
Let us again consider the matrix $\bX_{\pm}$ used in part (iv). From  formula (iv) in the case $i_1=i_2$, it follows that 
\begin{equation}\label{eta1}
\E{ \Phi_{\ell,k}^*(\bX_{\pm}) \big( (X_{\pm})^{(i_1)}_{j_1} \big)^2 
\big( (X_{\pm})^{(i_1)}_{j_2} \big)^2 }   =
\alpha(\ell,k)\frac{(k+1)(k+3)}{k(k+2)} \ .
\end{equation}
On the other hand, we can write  
\begin{gather}\label{eta2}
\E{ \Phi_{\ell,k}^*(\bX_{\pm}) \big( (X_{\pm})^{(i_1)}_{j_1} \big)^2 
\big( (X_{\pm})^{(i_1)}_{j_2} \big)^2 }   =\frac{1}{4} \E{\Phi_{\ell,k}^*(\bX)(X^{(i_1)}_{j_1}+X^{(i_2)}_{j_1})^2(X^{(i_1)}_{j_2}+X^{(i_2)}_{j_2})^2}  \notag \\
=\frac{1}{4} \left( 2\E{\Phi_{\ell,k}^*(\bX) (X^{(i_1)}_{j_1})^2 (X^{(i_1)}_{j_2})^2 }
+2 \E{\Phi_{\ell,k}^*(\bX) (X^{(i_1)}_{j_1})^2 (X^{(i_2)}_{j_2})^2 } 
+ 4 E(I)\right) \ .\notag
\end{gather}
Notice that the terms of the form $\E{\Phi_{\ell,k}^*(\bX) (X^{(i_1)}_{j_1})^2 
X^{(i_1)}_{j_2} X^{(i_2)}_{j_2}} $ are zero, by (iii). Hence, combining \eqref{eta1} and \eqref{eta2} together with the results obtained in (iv), we obtain
\begin{eqnarray*}
E(I) &=& \alpha(\ell,k)\left(\frac{(k+1)(k+3)}{k(k+2)}
- \frac{1}{2}\frac{(k+1)(k+3)}{k(k+2)} 
- \frac{1}{2}(k+1)\frac{(k+1)(k+2)-(k+3)}{k(k-1)(k+2)}\right) \\
&=& -\alpha(\ell,k)\frac{k+1}{k(k-1)(k+2)} \ ,
\end{eqnarray*}
which proves the formula. 
\end{proof}
The following proposition follows immediately from Lemma \ref{LemMonomials} and extends the results derived in Lemma 3.3 \cite{DNPR16} (corresponding to $(\ell,k)=(2,2)$ in our notation) to arbitrary integers $1\leq \ell \leq k$.
\begin{Prop}\label{ChaosPhi4}
The following properties hold:
\begin{enumerate}[label=\rm{(\roman*)}]
\item for every $(i,j) \in [\ell]\times [k]$, we have 
\begin{equation*}
\E{\Phi_{\ell,k}^*(\bX) H_4(X^{(i)}_j)} = -\frac{3}{k(k+2)}\alpha(\ell,k) \ ;
\end{equation*}

\item for every $(i_1,j_1) \neq (i_2,j_2)  \in [\ell] \times [k]$, we have 
\begin{equation*}
\E{\Phi_{\ell,k}^*(\bX) H_3(X^{(i_1)}_{j_1}) H_1( X^{(i_2)}_{j_2}) } = 0 \ , 
\end{equation*}

\item for every $(i_1,j_1) \neq (i_2,j_2) \neq (i_3,j_3) \in [\ell] \times [k]$, we have 
\begin{equation*}
\E{\Phi_{\ell,k}^*(\bX) H_2(X^{(i_1)}_{j_1}) H_1(X^{(i_2)}_{j_2})H_1(X^{(i_3)}_{j_3})} = 0 \ , 
\end{equation*}

\item for every  $(i_1,j_1)\neq(i_2,j_2)\in[\ell]\times[k]$, we have
\begin{eqnarray*}
\E{\Phi_{\ell,k}^*(\bX) H_2(X^{(i_1)}_{j_1})H_2(X^{(i_2)}_{j_2})}  &=&
-\frac{1}{k(k+2)}\alpha(\ell,k)\ind{i_1=i_2} \\
&& -  \frac{1}{k(k+2)}\alpha(\ell,k)  \ind{i_1 \neq i_2} \ind{j_1 = j_2}\ind{\ell\geq2} \\
&& +   \frac{k+3}{k(k-1)(k+2)}\alpha(\ell,k) \ind{i_1\neq i_2}\ind{j_1\neq j_2}\ind{\ell\geq2} \ ;  
\end{eqnarray*}

\item  for every collection $I = \{(i_1,j_1)\neq(i_2,j_2)\neq(i_3,j_3)\neq(i_4,j_4) \in [\ell]\times [k]\} $, we have 
\begin{equation*}
\E{\Phi_{\ell,k}^*(\bX)\prod_{a=1}^4H_1(X^{(i_a)}_{j_a})} 
=   -\frac{k+1}{k(k-1)(k+2)}\alpha(\ell,k) \ind{I \in S} \ind{\ell\geq2}\ ,
\end{equation*}
where 
$S=\{ \{(i_1,j_1),(i_1,j_2),(i_2,j_1),(i_2,j_2)\}: i_1 \neq i_2, j_1 \neq j_2\}$.
\end{enumerate}
\end{Prop}
\begin{proof} 
These formulae follow when writing $H_4(x) =x^4-6x^2+3, H_3(x)=x^3-3x, H_2(x)=x^2-1$ and $H_1(x)=x$ and then combining the formulae for monomials proved in Lemma \ref{LemMonomials} with Lemma \ref{LemWC}. We include the proof of (i): 
\begin{gather*}
\E{\Phi_{\ell,k}^*(\bX) H_4(X^{(i)}_j)} = 
\E{\Phi_{\ell,k}^*(\bX) (X^{(i)}_j)^4} - 6\E{\Phi_{\ell,k}^*(\bX) (X^{(i)}_j)^2} + 3\E{\Phi_{\ell,k}^*(\bX) } \\
= 3\alpha(\ell,k) \frac{(k+1)(k+3)}{k(k+2)} - 6\left(\frac{1}{k}+1\right) \alpha(\ell,k) + 3 \alpha(\ell,k)  
= -\frac{3}{k(k+2)}\alpha(\ell,k) \ ,
\end{gather*}
where we used \eqref{mean}. The remaining formulae are proved in the same spirit.
\end{proof}

\section{On the two-point correlation function}\label{2pointSec} 

\subsection{Covariances}
Fix $\ell \in [3]$ and $i\in [\ell]$. The following lemma gives the joint distribution of the vector $(\nabla T_n^{(i)}(z),\nabla T_n^{(i)}(0)) \in \R^6$ conditioned on $\{T_n^{(i)}(z)=T_n^{(i)}(0)=u_i\}$ for $u_i \in \R$ and $0\neq z \in \T$.
\begin{Lem}\label{JointCovMatrix}
For every $z \in \T$ such that $r_n(z) \neq \pm1$, the distribution of the vector $(\nabla T_n^{(i)}(z),\nabla T_n^{(i)}(0)) \in \R^6$ conditioned on $\{T_n^{(i)}(z)=T_n^{(i)}(0)=u_i\}$ is $\mathcal{N}_6(\mu^{(i)}_n,\Omega_n)$, where 
\begin{equation}\label{mu}
\mu^{(i)}_n = \mu^{(i)}_n(z)= \frac{u_i}{1+r_n(z)}
\begin{pmatrix}
\nabla r_n(z)^T \\ -\nabla r_n(z)^T
\end{pmatrix}
\end{equation}
and 
\begin{equation}\label{Omega}
\Omega_n = \Omega_n(z) =
\begin{pmatrix}
\Omega_{1,n}(z) & \Omega_{2,n}(z) \\ 
\Omega_{2,n}(z)^T & \Omega_{1,n}(z)  
\end{pmatrix} \ ,
\end{equation}
where 
\begin{gather*}
\Omega_{1,n} = \Omega_{1,n}(z) = \frac{E_n}{3} \Id_3 - \frac{\nabla r_n(z) \nabla r_n(z)^T}{1-r_n(z)^2} \ ; \\
\Omega_{2,n} = \Omega_{2,n}(z) = -\mathrm{Hess }(r_n(z)) + \frac{r_n(z)}{1-r_n(z)^2}\nabla r_n(z) \nabla r_n(z)^T  \ ,
\end{gather*}
with $\mathrm{Hess }(r_n(z))$ denoting the Hessian matrix of $r_n(z)$. 
\end{Lem}
\begin{proof}
We write $\partial_{a}:=\partial/\partial_{z_a}$  and 
$\partial_{ab}:=\partial^2/\partial_{z_a}\partial_{z_b}$ for $a,b=0,1,2,3$ with the convention $\partial_0=\Id$. 
Computing the covariance $\E{\partial_a T_n^{(i)}(z)\cdot\partial_bT_n^{(i)}(0)}$ and relating it to the covariance function $r_n$ given in \eqref{covARW}, we obtain that 
the covariance matrix of the vector $(\nabla T_n^{(i)}(z), \nabla T_n^{(i)}(0), T_n^{(i)}(z), T_n^{(i)}(0) ) \in \R^8$ is given by
\begin{eqnarray*}
\begin{pmatrix} 
A_n & B_n \\ B_n^T & C_n 
\end{pmatrix} \ ,
\end{eqnarray*}
where 
\begin{gather*}
A_n=A_n(z) = \begin{pmatrix}
E_n/3 \Id_3 & -\mathrm{Hess}(r_n(z)) \\
-\mathrm{Hess}(r_n(z)) & E_n/3 \Id_3 
\end{pmatrix} \ , \quad
B_n=B_n(z) = \begin{pmatrix}
\mathbf{0}^T & \nabla r_n(z)^T  \\
-\nabla r_n(z)^T &\mathbf{0}^T  
\end{pmatrix} \ , \\
C_n=C_n(z) = \begin{pmatrix}
1 & r_n(z) \\
r_n(z) & 1
\end{pmatrix} \ , 
\end{gather*}
and $\mathbf{0}:=(0,0,0)$. Thus, the covariance matrix of $(\nabla T_n^{(i)}(z), \nabla T_n^{(i)}(0))$ conditioned on $\{T_n^{(i)}(z)=T_n^{(i)}(0)=u_i\}$ is given by $
\Omega_n = \Omega_n(z) = A_n- B_nC^{-1}_nB_n^T $,
which yields the matrix in \eqref{Omega} after a standard computation. Its mean is given by 
\begin{eqnarray*}
\mu^{(i)}_n = \mu^{(i)}_n(z) = B_nC_n^{-1} 
\begin{pmatrix} 
u_i \\  u_i
\end{pmatrix} = \frac{u_i}{1+r_n(z)}
\begin{pmatrix} 
\nabla r_n(z)^T \\  -\nabla r_n(z)^T 
\end{pmatrix} \ .
\end{eqnarray*}
\end{proof}
\subsection{Two-point correlation function}
For $\ell \in [3]$, we fix  $u^{(\ell)}:=(u_1,\ldots,u_{\ell})\in\R^{\ell}$. The two-point correlation function associated with the random field $\bT_n^{(\ell)}$ is given by
\begin{eqnarray}\label{K}
K^{(\ell)}(x,y;u^{(\ell)}) := 
\E{\Phi_{\ell,3}^*(\Jac_{\bT_n^{(\ell)}}(x)) 
\Phi_{\ell,3}^*(\Jac_{\bT_n^{(\ell)}}(y)) | \bT_n^{(\ell)}(x) = \bT_n^{(\ell)}(y) = u^{(\ell)}} \notag\\
\times p_{(\bT_n^{(\ell)}(x),\bT_n^{(\ell)}(y))}(u^{(\ell)},u^{(\ell)}) ,
\end{eqnarray}
where $p_{(\bT_n^{(\ell)}(x),\bT_n^{(\ell)}(y))}(\cdot, \cdot) $ denotes the density function of the vector $(\bT_n^{(\ell)}(x),\bT_n^{(\ell)}(y)) \in \R^{2\ell}$ and 
$\Phi_{\ell,3}^*(A)=\sqrt{\det(AA^T)}$ for $A\in \mathrm{Mat}_{\ell,3}(\R)$. The function $K^{(\ell)}$ is defined whenever the distribution of $(\bT_n^{(\ell)}(x),\bT_n^{(\ell)}(y))$ is non-degenerate, that is, whenever
$r_n(x-y)\neq\pm 1$.

\medskip
The following lemma gives an upper bound for $K^{(\ell)}(z,0;u^{(\ell)})$ for $ z \in \T$ in terms of the covariance function $r_n$ and the norm of its gradient.
\begin{Lem}
For every $z \in\T$ such that $r_n(z) \neq \pm 1$, we have 
\begin{eqnarray}\label{Kupperbound}
K^{(\ell)}(z,0;u^{(\ell)} )
&\leq&  \big(1-r_n(z)^2\big)^{-\ell/2} \cdot (3)_{\ell}\bigg(\frac{E_n}{3}\bigg)^{\ell-1}
\bigg( \frac{E_n}{3} 
-\frac{\ell}{3} \frac{\norm{\nabla r_n(z)}^2}{1-r_n(z)^2}+ \frac{ \norm{u^{(\ell)}}^2 }{3}\frac{\norm{\nabla r_n(z)}^2}{(1+r_n(z))^2}\bigg) \notag\\
&=:&q^{(\ell)}(z,0;\norm{u^{(\ell)}}) .
\end{eqnarray}
\end{Lem}
\begin{proof}
By independence, the density factorizes as follows
\begin{equation*}
p_{(\bT_n^{(\ell)}(z),\bT_n^{(\ell)}(0))}(u^{(\ell)},u^{(\ell)})
= \prod_{i=1}^{\ell} p_{(T_n^{(i)}(z),T_n^{(i)}(0))}(u_i,u_i)   \ ,
\end{equation*}
and moreover satisfies
\begin{equation}\label{density}
p_{(\bT_n^{(\ell)}(z),\bT_n^{(\ell)}(0))}(u^{(\ell)},u^{(\ell)})
\leq 
\prod_{i=1}^{\ell} p_{(T_n^{(i)}(z),T_n^{(i)}(0))}(0,0) 
\leq  \big(1-r_n(z)^2\big)^{-\ell/2} \ .
\end{equation}
We now deal with the conditional expectation in \eqref{K}. 
First, by the Cauchy-Schwarz inequality, we have
\begin{gather*}
\E{\Phi_{\ell,3}^*(\Jac_{\bT_n^{(\ell)}}(z)) \Phi_{\ell,3}^*(\Jac_{\bT_n^{(\ell)}}(0)) | \bT_n^{(\ell)}(z) = \bT_n^{(\ell)}(0) = u^{(\ell)}}  \\
\leq \E{\Phi_{\ell,3}^*(\Jac_{\bT_n^{(\ell)}}(z))^2| \bT_n^{(\ell)}(z) = \bT_n^{(\ell)}(0) = u^{(\ell)}}^{1/2} 
\cdot \E{\Phi_{\ell,3}^*(\Jac_{\bT_n^{(\ell)}}(0))^2 | \bT_n^{(\ell)}(z) = \bT_n^{(\ell)}(0) = u^{(\ell)}}^{1/2}. 
\end{gather*}
By symmetry, we conclude that the two expectations above coincide, yielding 
\begin{gather}\label{condexp}
\E{\Phi_{\ell,3}^*(\Jac_{\bT_n^{(\ell)}}(z)) \Phi_{\ell,3}^*(\Jac_{\bT_n^{(\ell)}}(0)) | \bT_n^{(\ell)}(z) = \bT_n^{(\ell)}(0) = u^{(\ell)}} \notag \\
\leq \E{\Phi_{\ell,3}^*(\Jac_{\bT_n^{(\ell)}}(z))^2| \bT_n^{(\ell)}(z) = \bT_n^{(\ell)}(0) = u^{(\ell)}} =: \E{\Phi_{\ell,3}^*(X(z,u^{(\ell)}))^2},
\end{gather}
where $X(z,u^{(\ell)})=\big\{X^{(i)}_j(z,u^{(\ell)}) :(i,j)\in [\ell]\times[3]\big\} \in \mathrm{Mat}_{\ell,3}(\R)$ is a random matrix having the same distribution as 
$\Jac_{\bT_n^{(\ell)}}(z)$ conditionally on $\{\bT_n^{(\ell)}(z) = \bT_n^{(\ell)}(0) = u^{(\ell)}\}$.
Now, the Cauchy Binet formula \eqref{CB} yields
\begin{equation*}
\Phi_{\ell,3}^*(X(z,u^{(\ell)}))^2    = 
\sum_{j_1<\ldots<j_{\ell}\in[3]}
\det\big(X(z,u^{(\ell)})_{j_1,\ldots,j_{\ell}} \big)^2  \ ,
\end{equation*}  
where, as previously, $X(z,u^{(\ell)})_{j_1,\ldots,j_{\ell}}$ is the matrix obtained from $X(z,u^{(\ell)})$ by only keeping the columns labeled $j_1,\ldots,j_{\ell}$.
By definition of the determinant, we have
\begin{equation*}
\det\big(X(z,u^{(\ell)})_{j_1,\ldots,j_{\ell}}\big) 
= \sum_{\sigma\in \mathfrak{S}_{\ell}}
\eps(\sigma) \prod_{i=1}^{\ell}
X^{(i)}_{j_{\sigma(i)}}(z,u^{(\ell)}) \ ,
\end{equation*}
where $\eps(\sigma)$ denotes the signature of the permutation $\sigma \in \mathfrak{S}_{\ell}$. 
Then, developing the square, taking expectations and using independence,
\begin{eqnarray}\label{Flk}
&&\E{\Phi_{\ell,3}^*(X(z,u^{(\ell)}))^2}
= \sum_{j_1<\ldots<j_{\ell}\in[3]} 
\E{\det\big(X(z,u^{(\ell)})_{j_1,\ldots,j_{\ell}} \big)^2}  \notag \\
&=& \sum_{j_1<\ldots<j_{\ell}\in[3]} 
\sum_{\sigma, \sigma' \in \mathfrak{S}_{\ell}}
\eps(\sigma)\eps(\sigma')
\E{ \prod_{i=1}^{\ell}
X^{(i)}_{j_{\sigma(i)}}(z,u^{(\ell)})
\cdot
\prod_{l=1}^{\ell}
X^{(l)}_{j_{\sigma'(l)}}(z,u^{(\ell)})
}\notag\\
&=& \sum_{j_1<\ldots<j_{\ell}\in[3]} 
\sum_{\sigma, \sigma' \in \mathfrak{S}_{\ell}}
\eps(\sigma)\eps(\sigma') \prod_{i=1}^{\ell}
\E{ X^{(i)}_{j_{\sigma(i)}}(z,u^{(\ell)})  \cdot
X^{(i)}_{j_{\sigma'(i)}}(z,u^{(\ell)})}  .
\end{eqnarray}
For notational ease, we write 
\begin{equation*}
\mathbf{E}^{(i)}_{\ell,ab}=\mathbf{E}^{(i)}_{\ell,ab}(z,u^{(\ell)}):= \E{X^{(i)}_{a}(z,u^{(\ell)})X^{(i)}_{b}(z,u^{(\ell)})} \ , \quad
i \in [\ell], a,b \in [3] \ .
\end{equation*}
Exploiting once more the independence of the fields $T_n^{(1)},\ldots,T_n^{(\ell)}$, we have that
\begin{gather*}
\mathbf{E}^{(i)}_{\ell,ab} = 
\E{\partial_aT_n^{(i)}(z) \partial_bT_n^{(i)}(z) | \bT_n^{(\ell)}(z)=\bT_n^{(\ell)}(0)=u^{(\ell)}}  
= \E{\partial_aT_n^{(i)}(z) \partial_bT_n^{(i)}(z) | T_n^{(i)}(z)=T_n^{(i)}(0)=u_i}  .
\end{gather*}
Writing formula \eqref{Flk} for $\ell=1,2,3$ gives the respective relations
\begin{equation}\label{F1}
\E{\Phi_{1,3}^*(X(z,u^{(1)}))^2}
= \sum_{a \in[3]} \mathbf{E}^{(1)}_{1,aa} \ , 
\end{equation}
\begin{equation}\label{F2}
\E{\Phi_{2,3}^*(X(z,u^{(2)}))^2}
= \sum_{a\neq b\in[3]} \left\{
\mathbf{E}^{(1)}_{2,aa} \mathbf{E}^{(2)}_{2,bb} - 
\mathbf{E}^{(1)}_{2,ab}\mathbf{E}^{(2)}_{2,ab} \right\} 
\end{equation}
and 
\begin{eqnarray}\label{F3}
&&\E{\Phi_{3,3}^*(X(z,u^{(3)}))^2} = \sum_{a\neq b\neq c\neq a\in[3]} \bigg\{
\mathbf{E}^{(1)}_{3,aa}\mathbf{E}^{(2)}_{3,bb}\mathbf{E}^{(3)}_{3,cc}  \\
&& \qquad -\left( \mathbf{E}^{(1)}_{3,cc}\mathbf{E}^{(2)}_{3,ab}\mathbf{E}^{(3)}_{3,ab} 
+ \mathbf{E}^{(1)}_{3,ab}\mathbf{E}^{(2)}_{3,cc}\mathbf{E}^{(3)}_{3,ab}
+ \mathbf{E}^{(1)}_{3,ab}\mathbf{E}^{(2)}_{3,ab}\mathbf{E}^{(3)}_{3,cc}\right)   
+ 2\mathbf{E}^{(1)}_{3,ab}\mathbf{E}^{(2)}_{3,bc}\mathbf{E}^{(3)}_{3,ac} \bigg\} \ . \notag
\end{eqnarray}
We will now provide an explicit expression for the formulae on the right hand side of \eqref{F1}, \eqref{F2} and \eqref{F3}. 
For $z=(z_1,z_2,z_3) \in \T$ and $(a,b) \in [3]\times [3]$, we use the shorthand notations 
\begin{equation*}
\partial_a r_{n}(z) := \frac{\partial}{\partial z_a} r_n(z) \ ; \quad \partial_{ab}r_n(z):=\frac{\partial^2}{\partial z_a \partial z_b}r_n(z) 
\end{equation*}
and
\begin{equation*}
\rho_{ab}=\rho_{n,ab}(z) := 
\frac{\partial_{a} r_n(z)\cdot\partial_{b} r_n(z)}{1-r_n(z)^2}   \ ; \quad
\mu_{ab}=\mu_{n,ab}(z) := \frac{\partial_{a} r_n(z)\cdot\partial_{b} r_n(z)}{(1+r_n(z))^2} \ .
\end{equation*} 
Note that
\begin{equation}\label{Rel2}
\rho_{ab}^2=\rho_{aa}\rho_{bb} \ ,  \quad 
\mu_{ab}^2=\mu_{aa}\mu_{bb} \ , \quad
\rho_{aa}\mu_{bb}=\rho_{ab}\mu_{ab} \ .
\end{equation}
From Lemma \ref{JointCovMatrix}, it follows that for every $i \in [\ell]$ and $(a,b) \in [3]\times[3]$, 
\begin{equation}\label{aa}
\mathbf{E}^{(i)}_{\ell,aa}= 
\V{X^{(i)}_a(z,u^{(\ell)})}+\E{X^{(i)}_a(z,u^{(\ell)})}^2 
= \frac{E_n}{3} - \rho_{aa}+u_i^2\mu_{aa}
\end{equation}
and for $a\neq b$,
\begin{eqnarray}\label{ab}
\mathbf{E}^{(i)}_{\ell,ab} = 
\Cov{X^{(i)}_a(z,u^{(\ell)})}{X^{(i)}_b(z,u^{(\ell)})} + 
\E{X^{(i)}_a(z,u^{(\ell)})}\E{X^{(i)}_b(z,u^{(\ell)})} =-\rho_{ab} + u_i^2 \mu_{ab} \ .
\end{eqnarray}
Then, it is immediate that 
\begin{eqnarray*}
\E{\Phi_{1,3}^*(X(z,u^{(1)}))^2}
= \sum_{a \in[3]} \mathbf{E}^{(1)}_{1,aa}
= \sum_{a \in[3]}\bigg\{\frac{E_n}{3} - \rho_{aa}+u_1^2\mu_{aa} \bigg\} \ . 
\end{eqnarray*}
Similarly, using \eqref{aa} and \eqref{ab} in \eqref{F2} and \eqref{F3} and exploiting the identities in \eqref{Rel2} yields after simplifications
\begin{eqnarray*} 
&& \E{\Phi_{2,3}^*(X(z,u^{(2)}))^2} \\
&=& \sum_{a\neq b\in[3]} \bigg\{
\bigg(\frac{E_n}{3} - \rho_{aa}+u_1^2\mu_{aa}\bigg)
\bigg(\frac{E_n}{3} - \rho_{bb}+u_2^2\mu_{bb}\bigg)
-\big(-\rho_{ab} + u_1^2 \mu_{ab}\big)\big(-\rho_{ab} + u_2^2 \mu_{ab}\big) \bigg\}   \\
&=&\sum_{a\neq b \in [3]} \bigg\{
\bigg(\frac{E_n}{3}\bigg)^2 
- \frac{E_n}{3}(\rho_{aa}+\rho_{bb})
+\frac{E_n}{3}u_1^2\mu_{aa} 
+ \frac{E_n}{3}u_2^2\mu_{bb}
\bigg\} \ .
\end{eqnarray*}
and 
\begin{eqnarray*}
&& \E{\Phi_{3,3}^*(X(z,u^{(3)}))^2}  \notag\\
&=&
\sum_{a\neq b \neq c \neq a\in [3]}\bigg\{
\bigg(\frac{E_n}{3}\bigg)^3
-\bigg(\frac{E_n}{3}\bigg)^2
(\rho_{aa}+\rho_{bb}+\rho_{cc}) +\bigg(\frac{E_n}{3}\bigg)^2u_1^2\mu_{aa}  
+\bigg(\frac{E_n}{3}\bigg)^2u_2^2\mu_{bb} 
+\bigg(\frac{E_n}{3}\bigg)^2u_3^2\mu_{cc} \bigg\}    
\end{eqnarray*}
respectively.
Then, we note that for every $\ell\in [3]$, writing 
$\Delta_{\ell}:= \{i^{(\ell)}=(i_1,\ldots,i_{\ell}) \in [3]^{\ell}: i_a\neq i_b, \forall a\neq b \in [\ell]\}$, the following identities hold
\begin{eqnarray*}
 \sum_{i^{(\ell)} \in \Delta_{\ell} }1&=&(3)_{\ell} \ ; \\
 \sum_{i^{(\ell)} \in \Delta_{\ell} } (\rho_{i_1i_1}+\ldots+\rho_{i_{\ell}i_{\ell}}) 
&=& \ell \sum_{i^{(\ell)} \in \Delta_{\ell} } \rho_{i_1i_1} 
= \ell \frac{(3)_{\ell}}{3} \frac{\norm{\nabla r_n(z)^2}}{1-r_n(z)^2} \ ; \\
 \sum_{i^{(\ell)} \in \Delta_{\ell} } 
(u_1^2\mu_{i_1i_1}+ 
\ldots+ u_{\ell}^2\mu_{i_{\ell}i_{\ell}}) 
&=& u_1^2\left[\sum_{i^{(\ell)} \in \Delta_{\ell}  }  \mu_{i_1i_1} \right]
+ \ldots  +
u_{\ell}^2 \left[\sum_{i^{(\ell)} \in \Delta_{\ell} } \mu_{i_{\ell}i_{\ell}}\right] \\
&=& (u_1^2+\ldots+u_{\ell}^2)   \sum_{i^{(\ell)} \in \Delta_{\ell} }
\mu_{i_1i_1} 
=\norm{u^{(\ell)}}^2 \frac{(3)_{\ell}}{3}\frac{\norm{\nabla r_n(z)}^2}{(1+r_n(z))^2} \ .
\end{eqnarray*}
Using these identities,  \eqref{F1}, \eqref{F2} and \eqref{F3} finally reduce to
\begin{eqnarray}\label{expell}
\E{\Phi_{\ell,3}(X(z,u^{(\ell)}))^2}    
&=& (3)_{\ell}\bigg(\frac{E_n}{3}\bigg)^{\ell}
-\bigg(\frac{E_n}{3}\bigg)^{\ell-1}\ell \frac{(3)_{\ell}}{3}\frac{\norm{\nabla r_n(z)}^2}{1-r_n(z)^2}
+\bigg(\frac{E_n}{3}\bigg)^{\ell-1}
\norm{u^{(\ell)}}^2   \frac{(3)_{\ell}}{3}\frac{\norm{\nabla r_n(z)}^2}{(1+r_n(z))^2} \notag\\
&=& (3)_{\ell}\bigg(\frac{E_n}{3}\bigg)^{\ell-1}
\bigg( \frac{E_n}{3} 
-\frac{\ell}{3} \frac{\norm{\nabla r_n(z)}^2}{1-r_n(z)^2}+ \frac{ \norm{u^{(\ell)}}^2 }{3}\frac{\norm{\nabla r_n(z)}^2}{(1+r_n(z))^2}\bigg)  \ .
\end{eqnarray}
Plugging the bounds obtained in \eqref{density} and \eqref{expell} into \eqref{K} yields the desired upper bound for the two-point correlation function in \eqref{Kupperbound}.
\end{proof}

\begin{Lem}\label{ContK}
For every fixed $(x,y) \in \T\times \T$ such that $r_n(x-y)\neq\pm1$, the function $u^{(\ell)}:=(u_1,\ldots,u_{\ell})\mapsto K^{(\ell)}(x,y;u^{(\ell)})$ is continuous.
\end{Lem}
\begin{proof}
Denoting by $\Sigma=\Sigma(x-y)$ the covariance matrix of the vector $(T_n^{(i)}(x),T_n^{(i)}(y))$ for $i \in [\ell]$, the Gaussian density is given by
\begin{eqnarray*}
p_{(\bT_n^{(\ell)}(x),\bT_n^{(\ell)}(y))}(u^{(\ell)},u^{(\ell)})  
&=& \bigg(\frac{1}{2\pi \sqrt{1-r_n(x-y)^2}}\bigg)^{\ell}\prod_{i=1}^{\ell}  
\exp\bigg\{-\frac{1}{2} (u_i,u_i)^T\Sigma^{-1}(u_i,u_i)\bigg\}  \\
&=& \bigg(\frac{1}{2\pi \sqrt{1-r_n(x-y)^2}}\bigg)^{\ell}\prod_{i=1}^{\ell}  
\exp\bigg\{ -\frac{u_i^2}{2(1+r_n(x-y))} \bigg\}
\ ,
\end{eqnarray*}
which is a continuous function of $u^{(\ell)}$. We will now argue that the conditional expectation appearing in \eqref{K} is a continuous function of $u^{(\ell)}$. It can be rewritten as
\begin{gather*}
\E{\Phi_{\ell,3}^*(\Jac_{\bT_n^{(\ell)}}(x)) \Phi_{\ell,3}^*(\Jac_{\bT_n^{(\ell)}}(y)) | \bT_n^{(\ell)}(x) = \bT_n^{(\ell)}(y) = u^{(\ell)}}  
=\E{\Phi_{\ell,3}^*(X(x,u^{(\ell)})) \Phi_{\ell,3}^*(X(y,u^{(\ell)}))} , 
\end{gather*}
where, for every $x \in \T$, the random $\ell\times 3$ matrix $X(x,u^{(\ell)})=\big\{X^{(i)}_j(x,u^{(\ell)}):(i,j)\in [\ell]\times [3]\big\}$ has the same distribution as 
$\Jac_{\bT_n^{(\ell)}}(x)$ conditionally on $\{\bT_n^{(\ell)}(x) = \bT_n^{(\ell)}(y) = u^{(\ell)}\}$. From Lemma \ref{JointCovMatrix}, it follows that the mean in \eqref{mu} depends linearly on $u^{(\ell)}$. In view of the definition of $\Phi_{\ell,3}^*$, and the structure of the covariance function in \eqref{Omega}, we conclude that the above expected value is also a continuous function of 
$u^{(\ell)}$, showing that $K^{(\ell)}(x,y;u^{(\ell)})$ is a continuous function with variable $u^{(\ell)}$.
\end{proof}

\subsection{Taylor expansions}\label{TaylorSec}
We compute an expansion of $q^{(\ell)}(z,0;\norm{u^{(\ell)}})$ in \eqref{Kupperbound} around $z=0$. 
In order to do so, we start by deriving the Taylor expansions of $r_n$ and its first-order partial derivatives near $z=0$. For $n \in S_3$, let 
\begin{equation}\label{psi}
\Psi_n:=\frac{1}{n^2\Nn}\sum_{\lambda \in \Lambda_n} \lambda_k^4 \ , \qquad k=1,2,3 \ .
\end{equation}
and set $e_n:=E_n/3$. Note that $\Psi_n \leq 1$ since $\lambda_k^4 \leq n^2$.

\begin{Lem}
For $z=(z_1,z_2,z_3)\in \T$ and every $k\in [3]$, the following Taylor expansions hold near $z=0$:
\begin{eqnarray}
r_n(z) &=& 1-\frac{E_n}{6}\norm{z}^2+\frac{E_n^2}{24} \Psi_n\sum_{j=1}^3 z_j^4
+ \frac{E_n^2}{4}\bigg(\frac{1}{6}-\frac{1}{2}\Psi_n\bigg)\sum_{i<j\in [3]} z_i^2z_j^2 + R^{(0)}_n \notag \\
&=:& 1-\frac{e_n}{2}\norm{z}^2 + t_n(z) + R^{(0)}_n 
\label{Taylorr} \\
\partial_kr_n(z) &=& -\frac{E_n}{3}z_k+\frac{E_n^2}{6}\Psi_n\sum_{j=1}^3z_j^3 
+ \frac{E_n^2}{2}\bigg(\frac{1}{6}-\frac{1}{2}\Psi_n\bigg)\sum_{i \neq j\in [3]} z_jz_i^2 
+ R^{(k)}_n \notag \\ 
&=:& -e_nz_k + u_{n,k}(z) + R^{(k)}_n \ , \label{Taylorrk}
\end{eqnarray}
where $R^{(0)}_n=E_n^3O(\norm{z}^6)$ and $R^{(k)}_n = E_n^3O(\norm{z}^5)$, and the constants involved in the big-O notation are independent of $n$.
\end{Lem}

\begin{proof}
These expansions follow from direct computations of partial derivatives. Note that all derivatives of odd (resp. even) order of $r_n$ (resp. $\partial_kr_n$) vanish in view of the fact that, by symmetry,
$\sum_{\lambda \in \Lambda_n} \lambda_j^{\alpha}$ is zero whenever $\alpha$ is odd. 
Also, we note that 
\begin{equation*}
\frac{1}{n^2\Nn}\sum_{\lambda \in \Lambda_n} \lambda_a^2 \lambda_b^2 = \frac{1}{6}-\frac{1}{2}\Psi_n 
\end{equation*}
for $a\neq b\in [3]$, where $\Psi_n$ is as in \eqref{psi}.
The remainders are of the form   
$R^{(0)}_n = O( \norm{\partial^6 r_n}_{\infty} \norm{z}^6)$ and $R^{(k)}_n = O(\norm{\partial^6 r_n}_{\infty} \norm{z}^5)$, where   
\begin{gather*}
\partial^6 r_n  := \sup_{i_1,\ldots,i_6 \in [3]} 
\partial_{i_1,\ldots,i_6}r_n  
\end{gather*}
and $\partial_{i_1,\ldots,i_6}r_n(z)$ denotes partial derivatives of $r_n$ of cumulative order equal to $6$. Observe that for every $z \in \T$,
\begin{equation*}
\left|\partial^6r_n(z)\right| \leq \frac{(2\pi)^6}{\Nn}
\sum_{\lambda\in \Lambda_n} \lambda_1^{\alpha} \lambda_2^{\beta} \lambda_3^{\gamma} \ , 
\end{equation*}
where $\alpha, \beta, \gamma$ are non-negative even integers such that $\alpha+\beta+\gamma=6$. Therefore, we can write $ \lambda_1^{\alpha} \lambda_2^{\beta} \lambda_3^{\gamma} = 
\lambda_a^2\lambda_b^2\lambda_c^2$ for $a,b,c \in \{1,2,3\}$ not necessarily distinct. Then it follows that $\lambda_a^2\lambda_b^2\lambda_c^2 \leq \lambda_a^6/3 + \lambda_b^6/3 + \lambda_c^6/3$, so that 
\begin{equation*}
\left|\partial^6 r_n(z)\right| \leq \frac{(2\pi)^6}{\Nn}\sum_{\lambda \in \Lambda_n}\lambda_1^6 
\leq 
\frac{(2\pi)^6}{\Nn}\sum_{\lambda \in \Lambda_n} (\lambda_1^2+\lambda_2^2+\lambda_3^2)^3 = (2\pi)^6n^3 \leq E_n^3 \ , 
\end{equation*}
which concludes the proof.
\end{proof} 
The following result contains the expansion around zero of $q^{(\ell)}(z,0;\norm{u^{(\ell)}})$. In particular, we remark a singularity in the coefficient of $\norm{z}^{-\ell}$ in the case $\ell=3$, which is consistent with the fact that the mapping $z \mapsto \norm{z}^{-3}$ is not integrable on $\T$.
\begin{Lem}\label{Taylorq}
For $\ell\in[3]$, as $\norm{z} \to 0$, we have
\begin{equation}\label{Q}
q^{(\ell)}(z,0;\norm{u^{(\ell)}})
= (3)_{\ell} \left(1-\frac{\ell}{3}\right)
e_n^{\ell/2}\norm{z}^{-\ell} +
(3)_{\ell} \left(1+\norm{u^{(\ell)}}^2\right)E_n^{\ell/2+1}O(\norm{z}^{2-\ell}) ,
\end{equation}
where the constants involved in the big-O notation are independent of $n$.
\end{Lem}
\begin{proof}
From the expansion in \eqref{Taylorr} we obtain that  
\begin{eqnarray}\label{Taylor1}
1-r_n(z)^2 &=&  (1-r_n(z))(1+r_n(z)) 
\notag \\
&=& \left( \frac{e_n}{2}\norm{z}^2 - t_n(z) + E_n^3O(\norm{z}^6) \right) 
\left( 2-\frac{e_n}{2}\norm{z}^2 + t_n(z) + E_n^3O(\norm{z}^6)  \right) \notag \\
&=& e_n\norm{z}^2 - 
\bigg[\big(\frac{e_n}{2}\big)^2\norm{z}^4 + 2t_n(z) \bigg] + E_n^3O(\norm{z}^6)  \notag \\
&=:& e_n\norm{z}^2 - f_n(z) + E_n^3O(\norm{z}^6) \ , 
\end{eqnarray}
and 
\begin{eqnarray}\label{Taylor2}
(1+r_n(z))^2 &=&  
\left( 2-\frac{e_n}{2}\norm{z}^2 + t_n(z) + E_n^3O(\norm{z}^6)  \right)^2 \notag\\
&=& 4-2e_n\norm{z}^2
+\bigg[ \big(\frac{e_n}{2}\big)^2\norm{z}^4 + 4t_n(z) \bigg]+E_n^3O(\norm{z}^6)\notag\\
&=:&4-2e_n\norm{z}^2
+h_n(z)+E_n^3O(\norm{z}^6)  \ ,
\end{eqnarray}
where $t_n(z)$ is as in \eqref{Taylorr}. Note that since $\Psi_n \leq 1$, we have $t_n(z)=E_n^2O(\norm{z}^4)$ where the constant in the big-O notation is independent of $n$. Therefore, we have 
$f_n(z) := (e_n/2)^2 \norm{z}^4+2t_n(z)=E_n^2O(\norm{z}^4)$ and 
$h_n(z):=(e_n/2)^2 \norm{z}^4+4t_n(z)=E_n^2O(\norm{z}^4)$.
From \eqref{Taylorrk}, we have
\begin{equation*}
\partial_kr_n(z)^2 = 
\left(-e_nz_k + u_{n,k}(z) + E_n^3O(\norm{z}^5)\right)^2
=e_n^2z_k^2 
- 2e_nz_ku_{n,k}(z) + E_n^4O(\norm{z}^6) \ ,
\end{equation*}
so that summing over $k=1,2,3$ leads to
\begin{eqnarray}\label{Taylor2}
\norm{\nabla r_n(z)}^2 
= e_n^2\norm{z}^2  
-2e_n\sum_{k=1}^3 z_ku_{n,k}(z) +E_n^4O(\norm{z}^6) 
=: e_n^2\norm{z}^2+g_n(z)+E_n^4O(\norm{z}^6) \ , 
\end{eqnarray}
where $g_n(z) = E_n^3O(\norm{z}^4)$ and again the constant in the big-O notation does not depend on $n$.
Hence we obtain the expansions of the quotients
\begin{eqnarray}\label{Q2}
&&\frac{\norm{\nabla r_n(z)}^2}{1-r_n(z)^2} = \frac{e_n^2\norm{z}^2+g_n(z)+E_n^4O(\norm{z}^6)}{e_n\norm{z}^2 - f_n(z) + E_n^3O(\norm{z}^6)} \notag \\
&=&e_n\frac{1+e_n^{-2}\norm{z}^{-2}g_n(z)+E_n^2O(\norm{z}^4)}{1-e_n^{-1}\norm{z}^{-2}f_n(z)+E_n^2O(\norm{z}^4)} \notag \\
&=&e_n\left(1+\frac{g_n(z)}{e_n^2\norm{z}^2}+E_n^2O(\norm{z}^4)\right)
\left(1+\frac{f_n(z)}{e_n\norm{z}^2}+E_n^2O(\norm{z}^4)\right)\notag\\
&=& e_n\left( 1 
+\frac{g_n(z)}{e_n^2\norm{z}^2}+\frac{f_n(z)}{e_n\norm{z}^2}+E_n^2O(\norm{z}^4)\right) \notag\\
&=&e_n+\frac{g_n(z)}{e_n\norm{z}^2}+\frac{f_n(z)}{\norm{z}^2}+E_n^3O(\norm{z}^4) 
= e_n +E_n^2O(\norm{z}^2),
\end{eqnarray} 
since $e_n^{-1}\norm{z}^{-2}g_n(z)+\norm{z}^{-2}f_n(z) = E_n^2O(\norm{z}^2)$
and 
\begin{eqnarray}\label{Q2}
&&\frac{\norm{\nabla r_n(z)}^2}{(1+r_n(z))^2} = 
\frac{e_n^2\norm{z}^2+g_n(z)+E_n^4O(\norm{z}^6)}
{4-2e_n\norm{z}^2
+h_n(z)+E_n^3O(\norm{z}^6)} \notag \\
&=&\left(\frac{e_n}{2}\right)^2\norm{z}^2\frac{1+e_n^{-2}\norm{z}^{-2}g_n(z)+E_n^2O(\norm{z}^4)}{1-e_n/2 \norm{z}^{2} + h_n(z)/4 + E_n^3O(\norm{z}^6)} \notag \\
&=&\left(\frac{e_n}{2}\right)^2\norm{z}^2
\left(1+\frac{g_n(z)}{e_n^2\norm{z}^2}+E_n^2O(\norm{z}^4)\right)
\left(1+\frac{e_n}{2} \norm{z}^{2} + \frac{h_n(z)}{4} + E_n^3O(\norm{z}^6)\right)\notag\\
&=&\left(\frac{e_n}{2}\right)^2\norm{z}^2
\left(1+\frac{e_n}{2}\norm{z}^2 +\frac{h_n(z)}{4}+\frac{g_n(z)}{e_n^2\norm{z}^2}+E_n^2O(\norm{z}^4)
\right)  \notag\\
&=& 
\left(\frac{e_n}{2}\right)^2 \norm{z}^2
+E_n^4O(\norm{z}^4) \ .
\end{eqnarray} 
Using \eqref{expell}, we obtain the expansion near $z=0$ of $\E{\Phi_{\ell,3}^*(X(z,u^{(\ell)}))^2}$:
\begin{eqnarray}\label{TaylorCondExp}
&&\E{\Phi_{\ell,3}(X(z,u^{(\ell)}))^2} = 
(3)_{\ell}e_n^{\ell-1}
\bigg( e_n 
-\frac{\ell}{3} \frac{\norm{\nabla r_n(z)}^2}{1-r_n(z)^2}+ \frac{\norm{u^{(\ell)}}^2}{3}\frac{\norm{\nabla r_n(z)}^2}{(1+r_n(z))^2}\bigg)\notag\\
&=& (3)_{\ell}e_n^{\ell-1}
\bigg(e_n - \frac{\ell}{3}\left(e_n+E_n^2O(\norm{z}^2)\right)  
+\frac{\norm{u^{(\ell)}}^2}{3}\bigg\{
\left(\frac{e_n}{2}\right)^2 \norm{z}^2
+E_n^4O(\norm{z}^4)\bigg\}
\bigg)\notag\\
&=&(3)_{\ell}e_n^{\ell-1}
\left( e_n\left(1-\frac{\ell}{3}\right) + \left(1+\norm{u^{(\ell)}}^2\right)E_n^2O(\norm{z}^2) \right) \notag\\
&=& (3)_{\ell} \left(1-\frac{\ell}{3}\right)
e_n^{\ell} +
(3)_{\ell} \left(1+\norm{u^{(\ell)}}^2\right)E_n^{\ell+1}O(\norm{z}^2)  .
\end{eqnarray}
Then, using $1-r_n(z)^2 = e_n\norm{z}^2(1+E_nO(\norm{z}^2))$,
\begin{eqnarray*} 
&&q^{(\ell)}(z,0;\norm{u^{(\ell)}}) =
\big(1-r_n(z)^2\big)^{-\ell/2}\cdot 
\E{\Phi_{\ell,3}(X(z,u^{(\ell)}))^2} \notag\\
&=& e_n^{-\ell/2}\norm{z}^{-\ell}\E{\Phi_{\ell,3}(X(z,u^{(\ell)}))^2}(1+E_nO(\norm{z}^2)) \notag\\
&=&(3)_{\ell} \left(1-\frac{\ell}{3}\right)
e_n^{\ell/2}\norm{z}^{-\ell} +
(3)_{\ell} \left(1+\norm{u^{(\ell)}}^2\right)E_n^{\ell/2+1}O(\norm{z}^{2-\ell})  ,
\end{eqnarray*}
which has the desired form.
\end{proof}

The following lemma justifies the use of Kac-Rice formulae in a sufficiently small cube around the origin, $Q_0$.
\begin{Lem}\label{LemKR}
For every $n\in S_3$, there exists a sufficiently small constant $c_0>0$ such that for every $(x,y) \in \T\times \T$ satisfying  $0<\norm{x-y}<c_0/\sqrt{E_n}$, we have $r_n(x-y) \neq \pm 1$. 
\end{Lem}
\begin{proof}
We set $z=x-y$ and perform a Taylor expansion of $1-r_n(z)^2$ around $z=0$. From  \eqref{Taylor1}, we have
\begin{equation*}
1-r_n(z)^2 = \frac{E_n}{3}\norm{z}^2 + E_n^2O(\norm{z}^4) = \frac{E_n}{3}\norm{z}^2\big(1+E_nO(\norm{z}^2)\big) \ .
\end{equation*}
Thus, for every  $0<\norm{z}\ll 1/\sqrt{E_n}$, we obtain
\begin{equation*}
1-r_n(z)^2 = \frac{E_n}{3}\frac{C^2}{E_n}(1+O(1)) = \frac{C^2}{3}(1+O(1)) ,
\end{equation*}
for some absolute constant $C>0$, so that there exists a sufficiently small constant $c_0>0$ such that $1-r_n(z)^2>0$ for every $0<\norm{z}<c_0/\sqrt{E_n}$.
\end{proof}

\section{Continuity of nodal volumes}\label{AppCont}
In this section, we prove a more general version of the continuity theorem proved in Theorem 3 of \cite{AVP18}. Our version applies to vector-valued functions on the torus. For completeness, we give the arguments for the $d$-dimensional torus $\Td, d\geq2$. Recall that $\Td = \R^d/\Z^d \simeq [0,1]^d/_{\sim}$, where $\sim$ denotes the equivalence relation given by $(x_1,\ldots,x_d) \sim (x'_1,\ldots,x'_d)$ if and only if $x_i-x_i' \in \Z$ for every $i=1,\ldots,d$. Let us introduce some notation.
\paragraph{Topology on $\Td$.} (see e.g. \cite{Sh16}) Denote by $\pi_d: [0,1]^d \to \Td$ the quotient map associated with $\sim$. We endow the torus with the quotient topology, that is, the open (closed) subsets of $\Td$ are precisely the subsets $U \subset \Td$ such that $\pi_d^{-1}(U) \subset [0,1]^d$ are open (closed) in $[0,1]^d$ for the Euclidean topology. 
Moreover, we equip the torus with the quotient metric given by 
\begin{equation*}
\mathrm{dist}_d(\pi_d(x),\pi_d(x')) = \inf_{a\in \Z^d}\norm{x-x'+a}_d \ , \quad x,x' \in [0,1]^d \ ,
\end{equation*}
where $\norm{\cdot}_d$ denotes the standard Euclidean norm in $\R^d$. 
From now on, we will write $x$ instead of $\pi_d(x)$ for a point on the torus. Since the equivalence relation $\sim$ is defined coordinate-wise, we will implicitly use the fact that the $\Td$ is a realisation of the cartesian product of $d$ copies of $\tor{1}$.

\paragraph{Banach space of continuous functions on $\Td$.} 
For integers $1 \leq k<d$, let $E=C^1(\Td,\R^k)$  be the set of $C^1$ real vector-valued functions on $\Td$. Then, for a compact space $K \subset \Td$ (note that a compact subset on the torus has the form $\pi_d(\tilde{K})$ for some compact $\tilde{K} \subset [0,1]^d$), and $F=(F^{(1)},\ldots,F^{(k)}) \in E$, we define the norm 
\begin{equation*}
\norm{F}_K := \max_{i=1,\ldots,k} \ \sup_{ x \in K } 
\bigg( |F^{(i)}(x)| + \sum_{j=1}^d |\partial_{j}F^{(i)}(x) | \bigg) \ .
\end{equation*}  
We will use the following version of the Implicit Function Theorem for Banach spaces  (see e.g. \cite{E12} p.417).
\begin{Lem}[Implicit Function Theorem for Banach spaces]\label{LemIFT}
Let $X,Y,Z$ be Banach spaces and $f: X \times Y \to Z$ be a function of class $C^1$. Let $(x_0,y_0) \in X \times Y$ such that 
$f(x_0,y_0) = 0$ and  $(d_yf)_{(x_0,y_0)}:Y \to Z$ is an isomorphism. 
Then there exist neighborhoods $U(x_0) \subset X$ of $x_0$ and $U(x_0,y_0) \subset X\times Y$ of $(x_0,y_0)$ and a function $g: U(x_0) \to Z$ of class $C^1$ such that 
\begin{gather*}
\big( (x,y) \in U(x_0,y_0) , x \in U(x_0) \big)
\Rightarrow
\big( f(x,y)=0 \iff y= g(x)\big).
\end{gather*} 
\end{Lem}
Here $(d_yf)_{(x_0,y_0)}$ denotes the partial differential of $f$ with respect to $y \in Y$ computed at $(x_0,y_0)$.

\paragraph{Some notation.}
For $F \in E$, let $Z_K(F)$ be the set of zeros of $F$ lying in the compact $K \subset \Td$, i.e. $Z_K(F) = \{ x \in K: F(x) = 0 \}$. We denote by 
$\textrm{vol}(Z_K(F)):=\mathcal{H}_{d-k}(Z_K(F))$ the $(d-k)$-dimensional Hausdorff measure of $Z_K(F)$. As usual, we write  $\Jac_F(x) \in \mathrm{Mat}_{k,d}(\R)$ to indicate the Jacobian matrix of $F$ computed at $x$. 
We introduce the set $D_k := \{ J \subset [d]: \mathrm{card}(J) = k\} $, that is, the set of all subsets of $[d]$ that have cardinality $k$.
For $J  \in D_k$ and $x  \in \tor{d}$, we denote $x_J:=(x_l: l \in J)$ and $p_J(x):=\hat{x}_J:=(x_l: l \notin J)$. For $x_J$ as just defined, we write 
$\Jac_{F,x_J}$ for the $k \times k$ Jacobian matrix obtained when differentiating with respect to the variable $x_J$.
We say that $F$ is non-degenerate on $K$ if $\Jac_F(x_0)$ has full rank $k$ whenever $x_0 \in Z_F(K)$, that is, whenever there exists $J=J(x_0) \in D_k$ such that $\Jac_{F,x_J}(x_0)$ is invertible.

\medskip
We first prove the following lemma, adapted from \cite{AVP18}.
\begin{Lem}\label{LemEps}
Let $(F_n)_{n \geq 1} \subset E$ and $F \in E$ be such that $F_n \to F$ in the $C^1$ topology on $K\subset \Td$ as $n \to \infty$. Then, for $n$ sufficiently large and for every $\eps>0$, we have that $Z_K(F_n) \subset Z_K^{+\eps}(F)$, where
\begin{equation*}
Z_K^{+\eps}(F) := \{ x \in K: \mathrm{dist}_d(x,Z_K(F)) \leq \eps \} \ .
\end{equation*}
\end{Lem}
\begin{proof}
We proceed by contradiction. Assume that there exists 
$\eps>0$ such that $Z_K(F_n)$ is not a subset of $Z_K^{+\eps}(F)$ for $n$ big enough, i.e. such that for every $N\geq 1$, there exists $n \geq N$ and $x_n \in Z_K(F_n)$ with $\mathrm{dist}_d(x_n, Z_K(F)) >\eps$. As 
$(x_n)_{n \geq N} \subset K$ and $K$ is compact, we can extract a converging subsequence $(x_{n_j})_{j \geq 1}$; denote $x_{\infty}:= \lim_{j}x_{n_j} \in K$ and note that $\mathrm{dist}_d(x_{\infty},Z_K(F)) > \eps$ by assumption. Then, using the triangular inequality, we have for every $j\geq1$,
\begin{eqnarray}\label{conv}
\norm{F(x_{\infty})}_k &=& \norm{F(x_{\infty})-F_{n_j}(x_{n_j})}_k \notag \\
&\leq& 
\sum_{i=1}^k |F^{(i)}(x_{\infty})-F_{n_j}^{(i)}(x_{n_j})| \notag \\ 
&\leq& 
\sum_{i=1}^k |F^{(i)}(x_{\infty})-F_{n_j}^{(i)}(x_{\infty})|+
\sum_{i=1}^k |F_{n_j}^{(i)}(x_{\infty})-F_{n_j}^{(i)}(x_{n_j})| \notag \\
&\leq &k \cdot \norm{F-F_{n_j}}_K + \lambda \cdot \mathrm{dist}_d(x_{n_j},x_{\infty}) \ ,
\end{eqnarray}
where 
\begin{equation*}
\lambda := \sum_{i=1}^k  \sum_{l=1}^d \sup_{x \in K} |\partial_{l} F_{n_j}^{(i)}(x)| \leq k\cdot \max_{i=1,\ldots,k} \sum_{l=1}^d \sup_{x \in K} |\partial_{l} F_{n_j}^{(i)}(x)| \leq k \cdot \norm{F_{n_j}}_K < \infty \ ,
\end{equation*}
because $(F_n)_{n\geq1} \subset E$. Letting $j\to \infty$ in \eqref{conv} leads to $F(x_{\infty})=0$, since  $F_{n_j} \to F$ in the $C^1$ topology on $K$ and $x_{n_j} \to x_{\infty}$. Hence $x_{\infty} \in Z_K(F)$, but this contradicts the fact that $\mathrm{dist}_d(x_{\infty},Z_K(F)) \geq \eps>0$. 
\end{proof}

We now prove the continuity result about nodal volumes. The strategy of our proof is inspired by the proof in \cite{AVP18}.
\begin{Thm}[Continuity of the nodal volume]\label{ThmCont}
Let $(F_n)_{n \geq 1} \subset E$ and $F \in E$ be such that $F$ is non-degenerate on a compact $K \subset \Td$ and $F_n \to F$ in the $C^1$ topology on $K$ as $n \to \infty$. Then, as $n \to \infty$, 
\begin{equation*}
\mathrm{vol}(Z_K(F_n)) \to \mathrm{vol}(Z_K(F)) \ .
\end{equation*}
\end{Thm}
\begin{proof}
Denote by $\phi: E \times \tor{d} \to \R^k$ the evaluation map 
$\phi(f,x) := f(x)$. Since $F$ is non-degenerate, for all $x_0 \in K$ such that $\phi(F,x_0)=0$, there exists $J_0=J_0(x_0) \in D_k$ such that 
$\Jac_{F,x_{J_0}}(x_0)$ is invertible, that is, the linear map $(d_{x_{J_0}}\phi)_{(F,x_0)}: \tor{k} \to \R^k$
is an isomorphism. Therefore, by the Implicit Function Theorem  stated in Lemma \ref{LemIFT}, there exist open neighborhoods $U(F) \subset E$ of $F$, $U((x_0)_{J_0}) \subset \tor{k}$ of $(x_0)_{J_0}$ and $U((\hat{x_0})_{J_0}) \subset \tor{d-k}$ of $(\hat{x_0})_{J_0}$ 
as well as a function $X_0: E \times \tor{d-k} \to \R^k$
of class $C^1$ such that 
\begin{gather}\label{IFT}
\big( f \in U(F), x_{J_0} \in U((x_0)_{J_0}) , 
\hat{x}_{J_0} \in U((\hat{x_0})_{J_0})   \big)
\Rightarrow
\big( \phi(f,x)=0 \iff  x_{J_0} = X_0(f,\hat{x}_{J_0})\big).
\end{gather} 
Now denote $W_{0}=W_{0}(J_0) \subset \Td$ the set of points of $x \in \Td$ such that 
$x_{J_0} \in U((\hat{x_0})_{J_0})$ and $\hat{x}_{J_0} \in U((\hat{x_0})_{J_0}) $.
Then, choosing $f=F$ in \eqref{IFT}, we obtain that $Z_K(F)$ restricted to $W_{0}$ is the $(d-k)$-dimensional submanifold of $\Td$
\begin{equation*}
Z_K(F)\cap W_{0} = \left\{ x \in W_{0}: x_{J_0} = X_0(F,\hat{x}_{J_0}) 
= (X_0^{(1)}(F,\hat{x}_{J_0}),\ldots, X_0^{(k)}(F,\hat{x}_{J_0})) \right\} 
\end{equation*}
parametrized by 
\begin{equation}\label{par}
g_{0}=g_0(J_0) : \tor{d-k} \to \tor{d-k} \times \R^{k} \ , \quad \hat{x}_{J_0} \mapsto (\hat{x}_{J_0},X_0(F,\hat{x}_{J_0})) \ .
\end{equation}
Exploiting the compactness of $Z_K(F)$ together with the Implicit Function Theorem, there is $m\geq 1$ such that for every $j\in [m]$, there are $x_j \in Z_K(F), J_j=J_j(x_j) \in D_k$ and $W_{j}=W_j(J_j) \subset \Td$, such that 
\begin{equation*}
Z_K(F) \subset \bigcup_{j=1}^m W_{j} ,
\end{equation*}
and moreover, for every $j \in [m]$, the Implicit Function Theorem ensures the existence of an implicit function $X_j$ of class $C^1$ that yields a local parametrization 
\begin{equation*}
g_{j}=g_{j}(J_j):\tor{d-k} \to \tor{d-k} \times \R^{k} \ , \quad \hat{x}_{J_j} \mapsto (\hat{x}_{J_j},X_j(F,\hat{x}_{J_j} ))  
\end{equation*}
of $Z_K(F) \cap W_{j}$. 
Hence, if $T=\{j_1,\ldots,j_r\} \subset [m]$ for $r\leq m$ and $\bigcap_{j\in T}W_{j} \neq \emptyset$, then 
\begin{equation}\label{Gamma}
\Gamma_T(F):=Z_K(F) \cap \bigg( \bigcap_{j\in T} \overline{W_{j}} \bigg)
\end{equation}
describes a $(d-k)$-dimensional surface whose volume is computed when integrating the corresponding volume element $y \mapsto \sqrt{\det(\Jac^T_{g_{j_1}}(y) \Jac_{g_{j_1}}(y))}$ (see e.g. \cite{HJE17} Section 10.4). An application of the chain rule gives
\begin{eqnarray*}
\textrm{vol}(\Gamma_T(F)) = \int_{ Y_T} 
\sqrt{\det(\Jac^T_{g_{j_1}}(y) \Jac_{g_{j_1}}(y))} \ dy  
= \int_{Y_T}\sqrt{1+ \sum_{i\in[k]} \norm{ \nabla X_{j_1}^{(i)}(F,y)}_k^2 } \ dy \ ,
\end{eqnarray*}
where the region of integration is $Y_T=p_{J_1} \big(\bigcap_{j \in T} \overline{W_{j}} \big)$.
The total volume of $Z_K(F)$ is then computed by
\begin{equation}\label{vol}
\textrm{vol}(Z_K(F)) = \sum_{\emptyset \neq T\subset [m]} (-1)^{\mathrm{card}(T)}\textrm{vol}(\Gamma_T(F)) \ .
\end{equation}
Now we can find $\eps >0$ small enough such that 
$Z^{+\eps}_K(F) \subset \bigcup_{j=1}^m W_{j}$ and in view of Lemma \ref{LemEps}, it follows that $Z_K(F_n) \subset \bigcup_{j=1}^m W_{j}$ for $n$ sufficiently large, so that  
\begin{equation*}
Z_K(F_n) = \bigcup_{j=1}^m \big(Z_K(F_n) \cap \overline{W_{j}}\big) \ .
\end{equation*}
Since for $T=\{j_1,\ldots,j_r\} \subset [m]$, $\Gamma_T(F_n)$ as defined in \eqref{Gamma} identifies with a $(d-k)$-dimensional surface of volume $\mathrm{vol}(\Gamma_T(F_n))$, the total nodal volume of $F_n$ in $K$ is given by
\begin{equation*}
\textrm{vol}(Z_K(F_n)) = \sum_{\emptyset \neq T\subset [m]} (-1)^{\mathrm{card}(T)}\textrm{vol}(\Gamma_T(F_n)) \ .
\end{equation*}
Using Lipschitz continuity of $x \mapsto \sqrt{1+x}$ for $x>0$, it follows that
\begin{eqnarray*}
&& \big|\textrm{vol}(Z_K(F_n))  - \textrm{vol}(Z_K(F)) \big| \\
&\leq& \sum_{\emptyset \neq T\subset [m]}
\int_{Y_T} \left| \sqrt{1+ \sum_{i\in[k]} \norm{ \nabla X_{j_1}^{(i)}(F_n,y)}_k^2 }
- \sqrt{1+ \sum_{i\in[k]} \norm{\nabla X_{j_1}^{(i)}(F,y)}_k^2 }\right| \ dy \\
&\leq& \sum_{\emptyset \neq T\subset [m]}
\int_{Y_T}  \sum_{i\in[k]} 
\bigg| \norm{ \nabla X_{j_1}^{(i)}(F_n,y)}_k^2 
- \norm{ \nabla X_{j_1}^{(i)}(F,y)}_k^2 \bigg| \ dy \ .
\end{eqnarray*}
Now, using the reversed triangular inequality $\big|\norm{u}-\norm{v}\big|\leq \norm{u-v}$ yields
\begin{eqnarray*}
&&\left| \norm{ \nabla X_{j_1}^{(i)}(F_n,y)}_k^2 
- \norm{ \nabla X_{j_1}^{(i)}(F,y)}_k^2 \right| \\
&=&
\left| \norm{ \nabla X_{j_1}^{(i)}(F_n,y)}_k 
- \norm{ \nabla X_{j_1}^{(i)}(F,y)}_k\right| \cdot 
\left( \norm{ \nabla X_{j_1}^{(i)}(F_n,y)}_k 
+ \norm{ \nabla X_{j_1}^{(i)}(F,y)}_k\right)\\
&\leq& \norm{ \nabla X_{j_1}^{(i)}(F_n,y) - \nabla X_{j_1}^{(i)}(F,y) }_k \cdot \left( \norm{ \nabla X_{j_1}^{(i)}(F_n,y)}_k 
+ \norm{ \nabla X_{j_1}^{(i)}(F,y)}_k\right) \ .
\end{eqnarray*}
In order to conclude, it suffices to show that the first factor converges to $0$ uniformly on $Y_T$ as $n\to \infty$.
Consider the equation 
\begin{gather}\label{EqF}
F(\hat{y}_{J_1}, y_{J_1}) = F( \hat{y}_{J_1},X_{j_1}(F,\hat{y}_{J_1}) )=0,
\end{gather}
where, for the vector $(\hat{y}_{J_1}, y_{J_1})$ it is implicitly understood that coordinates with indices in $J_1$ are located in the corresponding position. Differentiating \eqref{EqF}
with respect to the coordinates $\hat{y}_{J_1}$, we obtain 
\begin{gather*}
\Jac_{F,\hat{y}_{J_1}}(\hat{y}_{J_1}, y_{J_1}) \cdot \Id_{d-k} + 
\Jac_{F,j_{J_1}}(\hat{y}_{J_1}, y_{J_1} ) \cdot 
\Jac_{X_{j_1},\hat{y}_{J_1}}(F,\hat{y}_{J_1})=0 ,
\end{gather*}
where the zero in the right-hand side denotes the zero $k\times (d-k)$ matrix. Therefore, since $\Jac_{F,y_{J_1}}(\hat{y}_{J_1}, y_{J_1})$ is invertible,
\begin{equation}\label{JacX}
\Jac_{X_{j_1},\hat{y}_{J_1}}(F,\hat{y}_{J_1}) = - \big[\Jac_{F,y_{J_1}}(\hat{y}_{J_1}, y_{J_1})\big]^{-1} \cdot 
\Jac_{F,\hat{y}_{J_1}}(\hat{y}_{J_1}, y_{J_1})  \ .
\end{equation}
Since $F_n$ converges to $F$ in the $C^1$ topology, we have that, for $n$ sufficiently large,  \eqref{JacX} holds true for $F_n$.
Writing out the $i$-th row for $i \in [k]$ of this relation, 
and using the fact that all the partial derivatives of $F_n$ converge uniformly to the corresponding partial derivatives 
of $F$ (as $F_n \to F$), we conclude that $\norm{ \nabla X_{j_1}^{(i)}(F_n,\hat{y}_{J_1}) - \nabla X_{j_1}^{(i)}(F,\hat{y}_{J_1}) }_k $ converges to zero uniformly on $Y_T$ as $n \to \infty$, proving the statement. 
\end{proof}

\section{Singular and non-singular cubes}\label{AppSing}
\subsection{Definitions and ancillary results}
\subsubsection{Singular and non-singular pairs of points and cubes}\label{Partition}
For every $n \in S_3$, we partition the torus into a disjoint union of cubes of length $1/M$, where $M=M_n\geq1$ is an integer proportional to $\sqrt{E_n}$ as follows: Let $Q_0 = [0,1/M)^3$; then we consider the partition of $\T$ obtained by translating $Q_0$ in the directions $k/M,k \in \Z^3$. Denote by $\mathcal{P}(M)$ the partition of $\T$ that is obtained in this way. By construction, $\textrm{card}({\mathcal{P}(M)})=M^3$.
By linearity, we can decompose the random variable $L_n^{(\ell)}$ as
\begin{equation}\label{decomp}
L_n^{(\ell)} = \sum_{Q \in \mathcal{P}(M)} L_n^{(\ell)}(Q) \ , \quad \ell \in [3]
\end{equation}
where $L_n^{(\ell)}(Q)$ denotes the nodal volume restricted to $Q$. 
From now on, we fix a small number $0 < \eta < 10^{- 10}$. In the forthcoming definition, we  define singular pairs of points and cubes. \begin{Def}[Singular pairs of points and cubes]\label{DefSing}
A pair of points $(x, y) \in \T \times \T $ is called a \textit{singular pair of points} if one of the following inequalities is satisfied:
\begin{equation*}
|r_n(x - y)| > \eta \ ,  \quad
|\partial_i r_n(x - y)| > \eta \sqrt{E_n/3} \ ,  \quad
|\partial_{ij} r_n (x - y)| > \eta E_n/3 \end{equation*}
for $(i, j) \in [3]\times[3]$.
A pair of cubes $(Q, Q') \in \mathcal{P}(M)^2 $ is called a \textit{singular pair of cubes} if the product $Q\times Q'$ contains a singular pair of points. We denote by $\mathcal{S}=\mathcal{S}(M) \subset \mathcal{P}(M)^2$ the set of singular pairs of cubes. 
A pair of cubes $(Q,Q') \in \mathcal{S}^c$ is called \textit{non-singular}. By construction, $\mathcal{P}(M)^2 = \mathcal{S} \cup \mathcal{S}^c$. 
\end{Def}
For fixed $Q \in \mathcal{P}(M)$, let us furthermore denote by $\mathcal{B}_Q$ the union  over all cubes $Q' \in \mathcal{P}(M)$ such
that $(Q,Q') \in \mathcal{S}$. In particular, analogously as in Lemma 6.3 of \cite{DNPR16}, we have
\begin{equation}\label{LebBQ}
\Leb(\mathcal{B}_Q) = O(\mathcal{R}_n(6)) ,
\end{equation} 
where $\mathcal{R}_n(6)=\int_{\T} r_n(z)^6 dz$.
We write 
\begin{eqnarray*}
\tilde{r}_{a,b}(x-y):=\E{\tilde{\partial}_a T_n^{(i)}(x) \cdot \tilde{\partial}_b T_n^{(i)}(y)} \ , \quad a,b=0,1,2,3 \ , \quad i \in [\ell] \ ,
\end{eqnarray*}
where, we recall that $\tilde{\partial_{a}}=(E_n/3)^{-1/2}\partial_a$ 
with the convention $\tilde{\partial_0}:=\Id$. Note that $\tilde{r}_{0,0}=r_n$ and that we dropped the dependence on $n$ in order to simplify notations. We need the following lemma:
\begin{Lem}\label{rab}
For every $a,b \in \{0,1,2,3\}$ and every integer $m \geq 1$,   
\begin{eqnarray}
\int_{\T}  \tilde{r}_{a,b}(z)^{2m} \ dz 
\ll \int_{\T}  r_n(z)^{2m} \ dz =\mathcal{R}_n(2m) ,
\end{eqnarray}
where the constant involved in the '$\ll$' notation depends only on $m$.
\end{Lem}
\begin{proof} 
By definition, we have
\begin{eqnarray*}
\tilde{r}_{a,b}(z) = r_n(z) \ind{a=b=0} +  \partial_{a}\partial_{b} \frac{r_n(z)}{\sqrt{E_n/3}} 
\ind{a \neq 0, b=0  \vee a=0,b\neq0}  
+ \partial_{a}\partial_{b} \frac{r_n(z)}{ E_n/3} \ind{a \neq 0 \wedge b\neq0} .  
\end{eqnarray*}
If $a=b=0$, the statement is clearly true. If $a\neq 0$ and $ b\neq 0$, we have 
\begin{eqnarray*}
\tilde{r}_{a,b}(z)  =  \frac{4\pi^2}{\Nn} 
\sum_{\lambda \in \Lambda_n}
\frac{\lambda_a}{\sqrt{E_n/3}} \frac{\lambda_b}{\sqrt{E_n/3}} e_{\lambda}(z) \ ,  
\end{eqnarray*}
so that for $m\geq 1$,  
\begin{eqnarray*}
&& \int_{\T} \tilde{r}_{a,b}(z)^{2m}  dz  = \frac{(4\pi^2)^{2m}}{\Nn^{2m}}\sum_{\lambda^{(1)},\ldots,\lambda^{(2m)} \in \Lambda_n} 
\frac{\lambda^{(1)}_a}{\sqrt{E_n/3}} \frac{\lambda^{(1)}_b}{\sqrt{E_n/3}} \cdots 
\frac{\lambda^{(2m)}_a}{\sqrt{E_n/3}} \frac{\lambda^{(2m)}_b}{\sqrt{E_n/3}} 
\int_{\T} e_{\lambda^{(1)}+\ldots+\lambda^{(2m)}}(z) dz  \\
&=&  \frac{(4\pi^2)^{2m}}{\Nn^{2m}} \frac{3^{2m}}{E_n^{2m}} \sum_{\lambda^{(1)},\ldots,\lambda^{(2m)} \in \Lambda_n} 
\lambda^{(1)}_a  \lambda^{(1)}_b \cdots 
\lambda^{(2m)}_a  \lambda^{(2m)}_b \cdot   \ind{ \lambda^{(1)}+\ldots+\lambda^{(2m)} = 0} \\
&\leq& C_m \frac{\mathrm{card}( \mathcal{C}_n(2m))}{\Nn^{2m}}   = C_m \mathcal{R}_n(2m) \ , \quad  C_m=3^{2m}
\end{eqnarray*}
where the last bound follows since $ \lambda^{(1)}_a  \lambda^{(1)}_b \cdots 
\lambda^{(2m)}_a  \lambda^{(2m)}_b \leq \sqrt{n}^{2m} \sqrt{n}^{2m}=n^{2m}$. The remaining case is shown in a similar way.
\end{proof}

\subsubsection{A diagram formula}\label{Diagram}
The proofs to be presented in the forthcoming sections are based on the following diagram formula. Such a formula is  counterpart to Proposition 8.1 in \cite{DNPR16}, and is based on the Leonov-Shiryaev formulae (see e.g. Proposition 3.2.1 \cite{PT11}). We introduce some notation:
For $i\in [\ell]$, write 
\begin{equation}\label{notX}
\big(X^{(i)}_0(x),X^{(i)}_1(x),X^{(i)}_2(x),X^{(i)}_3(x)\big) 
:= \big(T_n^{(i)}(x), \tilde{\nabla} T_n^{(i)}(x)\big) \ , \quad x \in \T
\end{equation}
and consider families of non-negative integers  
\begin{eqnarray*}
p^{(i)} = \big\{ p^{(i)}_j :j=0,1,2,3\big\} \ , \quad q^{(i)} = \big\{ q^{(i)}_j :j=0,1,2,3\big\}
\end{eqnarray*}
for which we write
\begin{eqnarray}
S(p^{(i)}) := \sum_{j=0}^3 p^{(i)}_j \ , \quad
S(q^{(i)}) := \sum_{j=0}^3 q^{(i)}_j \ .
\end{eqnarray}
For $m \in\{p^{(i)},q^{(i)}\}$, we also define the vector of $\R^{m_0}\times \R^{m_1}\times \R^{m_2} \times \R^{m_3}$ given by
\begin{gather*}
X_{m}^{(i)}(x) := \left([X^{(i)}_0(x)]_{m_0},[X^{(i)}_1(x)]_{m_1},[X^{(i)}_2(x)]_{m_2},[X^{(i)}_3(x)]_{m_3}\right) \ , 
\end{gather*}
where for an integer $n\geq1$ and a real number $N$, we write $[N]_n:=(N,\ldots,N) \in \R^{n}$.
\begin{Prop}\label{Diagram}
For $i \in [\ell]$, consider families of non-negative integers  
$p^{(i)} = \{ p^{(i)}_j :j=0,1,2,3\}$ and $q^{(i)} = \{ q^{(i)}_j :j=0,1,2,3\}$ as above, as well as $x,y \in \T$.  
Then,
\begin{gather*}
\E{ \prod_{i=1}^{\ell}  \prod_{j=0}^3 
H_{p^{(i)}_j}\left(X_j^{(i)}(x)\right) \cdot H_{q^{(i)}_j}\left(X_j^{(i)}(y)\right)}   
=\prod_{i=1}^{\ell} \E{  \prod_{j=0}^3 
H_{p^{(i)}_j}\left(X_j^{(i)}(x)\right) \cdot H_{q^{(i)}_j}\left(X_j^{(i)}(y)\right)}\\
= \prod_{i=1}^{\ell} \ind{S(p^{(i)})=S(q^{(i)})} \sum_{\sigma_i}
\prod_{j=1}^{S(p^{(i)})} \E{\left(X^{(i)}_{p^{(i)}}\right)_j(x) \cdot \left(X^{(i)}_{q^{(i)}}\right)_{\sigma_i(j)}(y)} \ ,
\end{gather*}
where the sum runs over all permutations $\sigma_i$ of  $\{1,\ldots, S(p^{(i)})\}$. 
\end{Prop}

\subsection{Proof of Lemma \ref{LemASL2}}\label{SecProofASL2}
\noindent\underline{\textit{Proof of the almost sure convergence:}} 
In the case $\ell=3$, one can argue similarly as in the proof of Lemma 3.1 in \cite{DNPR16}.
We present the arguments for $\ell=2$. Since, $\bT_n^{(2)}$ is of class $C^{\infty}$, Sard's Theorem (see e.g. \cite{S42}) implies that its set of critical values has almost surely zero Lebesgue measure. Therefore, applying the Co-Area formula (Proposition 6.13, \cite{AW09}) to the functions $f= \bT_n^{(2)}:\T \to \R^{2}$ and $g:\R^2\to\R, g(x_1,x_2)=(2\eps)^{-2}\prod_{i=1}^{2}
\ind{[-\eps,\eps]}(x_i)$ yields
\begin{eqnarray}\label{con}
L_{n,\eps}^{(2)} 
= (2\eps)^{-2} \int_{[-\eps,\eps]^{2}} 
L_n^{(2)}(\T; (u_1,u_2)) \ du_1du_2 \ ,
\end{eqnarray}
where for $B \subset \T$, we set $L_n^{(2)}(B; (u_1,u_2)) = \mathcal{H}_{1}\{ (\bT_n^{(2)})^{-1}(\{(u_1,u_2)\}) \cap B\}$.
Now, as $(u_1,u_2) \to (0,0)$, the random field $\bT_n^{(2)}-(u_1,u_2)$ converges in the $C^1$ topology on $\T$ to the random field $\bT_n^{(2)}$, which is non-degenerate - as can be seen e.g. by checking the assumptions of Proposition 6.12 in \cite{AW09} - so that by the continuity of the nodal volume proved in Theorem \ref{ThmCont}, 
\begin{equation*}
\lim_{(u_1,u_2) \to (0,0)} 
\mathcal{H}_{1}\left\{ (\bT_n^{(2)}-(u_1,u_2))^{-1}(\{(0,0)\})\right\}
= \mathcal{H}_{1}\left\{ (\bT_n^{(2)})^{-1}(\{(0,0)\})\right\}
= L_n^{(2)}(\T;(u_1,u_2))\ .  
\end{equation*} 
This proves the continuity of $L_n^{(2)}(\T; (u_1,u_2))$ at $(u_1,u_2)=(0,0)$. The almost sure convergence then follows by letting $\eps \to 0$ in \eqref{con}. 

\medskip
\noindent\underline{\textit{Proof of the $L^2(\Prob)$-convergence:}} We now prove that the convergence also takes place in $L^2(\Prob)$. For completeness, we include the three cases corresponding to $\ell=1,2,3$ in our proof.
We start by proving an auxiliary result. Recall that $Q_0$ is the small cube around the origin of side length $1/M$.
\begin{Lem}\label{ContMap}
The map $(u_1,\ldots,u_{\ell}) \mapsto 
\E{L_n^{(\ell)}(Q_0; (u_1,\ldots,u_{\ell}))^2}$
is continuous at $(0, \ldots, 0)$.
\end{Lem}
\begin{proof}
Writing $u^{(\ell)}:=(u_1,\ldots,u_{\ell})$, we will prove that 
\begin{equation}\label{showcont}
\lim_{u^{(\ell)}\to(0,\ldots,0)}
\E{L_n^{(\ell)}(Q_0;u^{(\ell)})^2} = 
\E{L_n^{(\ell)}(Q_0;(0,\ldots,0))^2} \ .
\end{equation}
By virtue of Lemma \ref{LemKR} the random field $(\bT_n^{(\ell)}(x),\bT_n^{(\ell)}(y))$ is non-degenerate in $Q_0$ so that we may use Kac-Rice formulae in the cube $Q_0$. For $\ell=1,2$, by Theorem 6.9 \cite{AW09}, 
\begin{equation*}
\E{L_n^{(\ell)}(Q_0; u^{(\ell)})^2}
= \int_{Q_0 \times Q_0} K^{(\ell)}(x,y;u^{(\ell)}) \ dx dy \ ,
\end{equation*} 
where $K^{(\ell)}$ is as in \eqref{K}, whereas for $\ell=3$, we write 
\begin{equation*}
\E{L_n^{(3)}(Q_0; u^{(3)})^2}
=\E{L_n^{(3)}(Q_0; u^{(3)})\big(L_n^{(3)}(Q_0; u^{(3)})-1\big)}
+ \E{L_n^{(3)}(Q_0; u^{(3)})} \ , 
\end{equation*} 
and apply Theorem 6.2 resp. Theorem 6.3 \cite{AW09} to the respective summands, so that 
\begin{eqnarray*}
&&\E{L_n^{(3)}(Q_0; u^{(3)})^2}   \\
&=&\int_{Q_0\times Q_0} K^{(3)}(x,y;u^{(3)}) dx dy 
+ \int_{Q_0} \E{\Phi_{\ell,3}^*(\Jac_{\bT_n^{(3)}}(x))|\bT_n^{(3)}(x)=u^{(3)}} \cdot p_{\bT_n^{(3)}(x)}(u^{(3)})dx \\ 
&=& \int_{Q_0\times Q_0} K^{(3)}(x,y;u^{(3)}) dx dy   + \int_{Q_0} \E{\Phi_{\ell,3}^*(\Jac_{\bT_n^{(3)}}(x))} \cdot p_{\bT_n^{(3)}(x)}(u^{(3)})\ dx \ ,
\end{eqnarray*} 
where the last line follows from the independence of $\bT_n^{(3)}(x)$ and $\Jac_{\bT_n^{(3)}}(x)$. 
Thus, the LHS of \eqref{showcont} reduces to 
\begin{eqnarray}\label{limu}
\lim_{u^{(\ell)}\to(0,\ldots,0)} \E{L_n^{(\ell)}(Q_0; u^{(\ell)})^2} =\lim_{u^{(\ell)}\to(0,\ldots,0)} \biggl(
\int_{Q_0\times Q_0} K^{(\ell)}(x,y;u^{(\ell)}) dx dy \notag\\  
+  \ind{\ell=3}\times 
\int_{Q_0} \E{\Phi_{\ell,3}^*(\Jac_{\bT_n^{(3)}}(x))} \cdot p_{\bT_n^{(3)}(x)}(u^{(3)}) dx  \biggr)  .
\end{eqnarray}
Let us deal with the additional term appearing in the case $\ell=3$: The Hadamard inequality (see e.g. \cite{RWH17}) and independence yield
\begin{eqnarray*}
\E{\Phi_{\ell,3}^*(\Jac_{\bT_n^{(3)}}(x))}
\leq \prod_{i=1}^3 \E{\norm{\nabla T_n^{(i)}(x)}}  
\leq   \E{\norm{\nabla T_n^{(1)}(x)}^2}^{3/2} = E_n^{3/2} .
\end{eqnarray*}
Moreover, the Gaussian probability density 
$u^{(3)} \mapsto p_{\bT_n^{(3)}(x)}(u^{(3)})$ satisfies 
\begin{equation*}
p_{\bT_n^{(3)}(x)}(u^{(3)}) = \prod_{i=1}^3
p_{T_n^{(i)}(x)}(u_i) \leq \big(p_{T_n^{(1)}(x)}(0)\big)^3 = (2\pi)^{-3/2} \ .
\end{equation*}
Therefore, applying dominated convergence yields,
\begin{gather*}
\lim_{u^{(\ell)} \to (0,\ldots,0)} 
\int_{Q_0} \E{\Phi_{\ell,3}^*(\Jac_{\bT_n^{(3)}}(x))} \cdot p_{\bT_n^{(3)}(x)}(u^{(3)}) dx \\
= \int_{Q_0} \E{\Phi_{\ell,3}^*(\Jac_{\bT_n^{(3)}}(x))}p_{\bT_n^{(3)}(x)}(0,0,0) \ dx  
= \E{L_n^{(3)}(Q_0;(0,0,0))} \ .
\end{gather*}
We now deal with the first summand of the RHS of \eqref{limu}. By stationarity, 
\begin{gather*}
\int_{Q_0\times Q_0} K^{(\ell)}(x,y;u^{(\ell)}) dxdy = \int_{Q_0-Q_0} \mathrm{Leb}(Q_0 \cap Q_0-z)K^{(\ell)}(z,0;u^{(\ell)}) dz.
\end{gather*}
Now, for every $u^{(\ell)}$ in a neighbourhood of $(0,\ldots,0)$, say $\norm{u^{(\ell)}} < \delta$, for some $\delta>0$, in view of \eqref{Kupperbound}, we have 
$K^{(\ell)}(z,0;u^{(\ell)})\leq q^{(\ell)}(z,0;\norm{u^{(\ell)}}) < q^{(\ell)}(z,0;\delta)$ for every $z$. Therefore, again by dominated convergence, we infer
\begin{gather*}
\lim_{u^{(\ell)}\to(0,\ldots,0)} \int_{Q_0\times Q_0} K^{(\ell)}(x,y;u^{(\ell)}) dxdy 
= \int_{Q_0\times Q_0} \lim_{u^{(\ell)}\to(0,\ldots,0)} K^{(\ell)}(x,y;u^{(\ell)})  dxdy  \\
=\E{L_n^{(\ell)}(\T;(0,\ldots,0))^2} ,
\end{gather*}
where, in the last line we used the continuity result proved in Lemma \ref{ContK}.
\end{proof}
Now, for a domain $B \subset \T$, we 
set $L_n^{(\ell)}(B):=L_n^{(\ell)}(B;(0,\ldots,0))$ and for $\eps>0$, we write
$L_{n,\eps}^{(\ell)}(B):=L_{n,\eps}^{(\ell)}(B;(0,\ldots,0))$ for the $\eps$-approximation of $L_n^{(\ell)}(B)$ (recall definition \eqref{intnorm}). We define the random variable 
\begin{gather}\label{Aeps}
A_n^{(\ell)}(B;\eps,\eps'):=L_{n,\eps}^{(\ell)}(B)-L_{n,\eps'}^{(\ell)}(B) \ , \quad n \in S_3, \ \eps>0, \ \eps'>0.
\end{gather}
Proving that $L_{n,\eps}^{(\ell)}$ converges to $L_n^{(\ell)}$ in $L^2(\Prob)$ as $\eps\to 0$ is equivalent to showing that for every $n \in S_3$, the random variable $A_n^{(\ell)}(\T;\eps,\eps')$ converges to zero in $L^2(\Prob)$ as $\eps,\eps'\to0$. 
We first show that the latter convergence holds in the small cube $Q_0$ around the origin. 
\begin{Lem}\label{L2Q0}
For every $n \in S_3$, one has that $A_{n}^{(\ell)}(Q_0;\eps,\eps') \to 0$ in $L^2(\Prob)$ as $\eps,\eps'\to 0$. 
\end{Lem}
\begin{proof}
We will show that, for every $n \in S_3$, the sequence $\{L_{n,\eps}^{(\ell)}(Q_0):\eps > 0\}$ converges in $L^2(\Prob)$ to $L_n^{(\ell)}(Q_0)$ as $\eps \to 0$. This implies that $\{L_{n,\eps}^{(\ell)}(Q_0):\eps > 0\}$ is a Cauchy sequence in $L^2(\Prob)$, and therefore $A_{n}^{(\ell)}(Q_0;\eps,\eps') \to 0$ in $L^2(\Prob)$ as $\eps,\eps'\to0$. 
Since almost sure convergence together with convergence of norms implies convergence in $L^2(\Prob)$ (see e.g. \cite{R06} p.73), it suffices to show that
$\E{L_{n,\eps}^{(\ell)}(Q_0)^2} \to \E{L_n^{(\ell)}(Q_0)^2}$ as $\eps \to 0$. We start by proving that $L_{n,\eps}^{(\ell)}(Q_0) \in L^2(\Prob)$ for every $\eps>0$:
Using the definition of $L_{n,\eps}^{(\ell)}(Q_0)$ and the Hadamard inequality, we have 
\begin{equation*}
L_{n,\eps}^{(\ell)}(Q_0)
\leq  (2\eps)^{-\ell}\int_{Q_0} \Phi_{\ell,3}^*(\mathrm{Jac}_{\bT_n^{(2)}}(x)) \ dx 
\leq  (2\eps)^{-\ell}\int_{Q_0} 
\prod_{i=1}^{\ell}\norm{\nabla T_n^{(i)}(x)} dx  
\leq  (2\eps)^{-\ell}\int_{\T} 
\prod_{i=1}^{\ell}\norm{\nabla T_n^{(i)}(x)} dx\ ,
\end{equation*}
and hence, using Jensen's inequality, 
\begin{gather*}
\E{L_{n,\eps}^{(\ell)}(Q_0)^2} 
\leq (2\eps)^{-2\ell}\E{\bigg(\int_{\T} 
\prod_{i=1}^{\ell}\norm{\nabla T_n^{(i)}(x)} dx \bigg)^2}  \\
\leq (2\eps)^{-2\ell}\E{ \int_{\T} 
\prod_{i=1}^{\ell}\norm{\nabla T_n^{(i)}(x)}^2 \ dx }   
=(2\eps)^{-2\ell} \int_{\T} 
\E{\norm{\nabla T_n^{(1)}(x)}^2}^{\ell}    \ dx = (2\eps)^{-2\ell} E_n^{\ell}
< + \infty .
\end{gather*}
In order to prove that $L_n^{(\ell)}(Q_0)$ is in $L^2(\Prob)$, we use Kac-Rice formulae for second moments and proceed as in the proof of Lemma \ref{ContMap}:
For $\ell=3$, we write
\begin{eqnarray*}
\E{L_n^{(3)}(Q_0)^2} =
\E{L_n^{(3)}(Q_0)(L_n^{(3)}(Q_0)-1)}
+ \E{L_n^{(3)}(Q_0)} ,
\end{eqnarray*}
and apply Kac-Rice formula for moments and use stationarity, to write
\begin{eqnarray}\label{intQ0}
\E{L_n^{(\ell)}(Q_0)^2} = \int_{Q_0\times Q_0} 
K^{(\ell)}(x,y;(0,\ldots,0)) \ dxdy 
+ \E{L_n^{(3)}(Q_0)}\ind{\ell=3} \notag \\
\leq \Leb(Q_0) \int_{2Q_0} K^{(\ell)}(z,0;(0,\ldots,0)) \ dz 
+ \frac{E_n^{3/2}}{M^3}\ind{\ell=3}\ ,
\end{eqnarray}
where the last line follows from the fact that $\E{L_n^{(3)}(Q_0)}= \Leb(Q_0)\E{L_n^{(3)}}\ll M^{-3}E_n^{3/2}$. From \eqref{Kupperbound} and the Taylor expansion in Lemma \ref{Taylorq}, we can upper bound \eqref{intQ0} by
\begin{eqnarray}\label{Est2}
&\leq& \frac{1}{M^3} \int_{2Q_0} q^{(\ell)}(z,0;0) \ dz
+ \frac{E_n^{3/2}}{M^3}\ind{\ell=3} \notag \\
&\ll& \frac{1}{M^3} \int_{0}^{1/M} \left[
(3)_{\ell} \left(1-\frac{\ell}{3}\right)
e_n^{\ell/2}r^{2-\ell} +
(3)_{\ell} \left(1+\norm{u^{(\ell)}}^2\right)E_n^{\ell/2+1} r^{4-\ell}  \right]dr + \frac{E_n^{3/2}}{M^3}\ind{\ell=3} \notag\\
&\ll&  E_n^{-2}\ind{\ell=1}+ E_n^{-1} \ind{\ell=2} + \ind{\ell=3}  .
\end{eqnarray}
This proves that $L_n^{(\ell)}(Q_0)$ is an element of $L^2(\Prob)$.
In order to show that the convergence holds in $L^2(\Prob)$, we will prove the inequalities
\begin{gather*}
\E{L_n^{(\ell)}(Q_0)^2} \leq \lim_{\eps\to0} \E{L_{n,\eps}^{(\ell)}(Q_0)^2}\leq \E{L_n^{(\ell)}(Q_0)^2}.
\end{gather*}
For the first inequality, we use the almost sure convergence proved above and Fatou's Lemma to write
\begin{equation*}
\E{L_n^{(\ell)}(Q_0)^2} = \E{\liminf_{\eps\to 0}L_{n,\eps}^{(\ell)}(Q_0)^2} \leq \liminf_{\eps \to 0}\E{L_{n,\eps}^{(\ell)}(Q_0)^2} = \lim_{\eps \to 0}\E{L_{n,\eps}^{(\ell)}(Q_0)^2}  .
\end{equation*}
The second inequality is proved as follows:
Applying the Co-Area formula (Proposition 6.13, \cite{AW09}) and then the Cauchy-Schwarz inequality
\begin{gather*}
\E{L_{n,\eps}^{(\ell)}(Q_0)^2} =
(2\eps)^{-2\ell} \int_{[-\eps,\eps]^{\ell}\times [-\eps,\eps]^{\ell}}
\E{L_n^{(\ell)}(Q_0; u^{(\ell)})\cdot 
L_n^{(\ell)}(Q_0; v^{(\ell)})}  
du^{(\ell)} dv^{(\ell)} \\
\leq\bigg( (2\eps)^{-\ell} \int_{[-\eps,\eps]^{\ell} }
\E{L_n^{(\ell)}(Q_0; u^{(\ell)})^2}^{1/2} \ du^{(\ell)}\bigg)^2 ,
\end{gather*}
where $u^{(\ell)}=(u_1,\ldots,u_{\ell})$ and $v^{(\ell)}=(v_1,\ldots,v_{\ell})$.
By Lemma \ref{ContMap}, the map $u^{(\ell)} \mapsto 
\E{L_n^{(\ell)}(\T; u^{(\ell)})^2}$
is continuous at $(0,\ldots,0)$, so that letting $\eps \to 0$ yields the desired inequality.
\end{proof}
Taking advantage of the partition of the torus introduced in Section \ref{Partition}, we decompose 
\begin{gather*}\label{Part}
\E{A_n^{(\ell)}(\T;\eps,\eps')^2} = \sum_{(Q,Q') \in \mathcal{P}(M)^2} \E{A_n^{(\ell)}(Q;\eps,\eps') A_n^{(\ell)}(Q';\eps,\eps')} \notag\\ 
=\bigg\{\sum_{(Q,Q') \in \mathcal{S}}+\sum_{(Q,Q') \in \mathcal{S}^c}\bigg\}
\E{A_n^{(\ell)}(Q;\eps,\eps') A_n^{(\ell)}(Q';\eps,\eps')}  
=: S_{n,1}^{(\ell)}(\eps,\eps') + S_{n,2}^{(\ell)}(\eps,\eps')
\end{gather*}
and control each term separately. This is the content of the next two lemmas.
\begin{Lem} 
For every $n \in S_3$, one has that $|S_{n,1}^{(\ell)}(\eps,\eps')|\to 0 $ as 
$\eps,\eps'\to 0$.
\end{Lem}
\begin{proof}
Using the triangular inequality and then the Cauchy-Schwarz inequality, we can write 
\begin{gather}\label{Est1}
\left|S_{n,1}^{(\ell)}(\eps,\eps')\right|\leq \sum_{(Q,Q') \in \mathcal{S}} 
\sqrt{\E{A_n^{(\ell)}(Q;\eps,\eps')^2} \E{ A_n^{(\ell)}(Q';\eps,\eps')^2} }  
= \textrm{card}(\mathcal{S}) \cdot\E{ A_n^{(\ell)}(Q_0;\eps,\eps')^2}  \ ,
\end{gather}
where we used translation-invariance of $\bT_n^{(\ell)}$ in order to reduce the arguments over the cube $Q_0$. 
Now, thanks to \eqref{LebBQ} and the fact that we are summing over pairs of cubes yields $\textrm{card}(\mathcal{S}) = M^6 \cdot \Leb(\mathcal{B}_Q) = O(E_n^{3}\mathcal{R}_n(6))$. By Lemma \ref{L2Q0}, $\E{ A_n^{(\ell)}(Q_0;\eps,\eps')^2}$ converges to $0$ as $\eps,\eps'\to 0$, which yields the desired conclusion.
\end{proof}

\begin{Lem}\label{Part2}
For every $n \in S_3$, one has that $|S_{n,2}^{(\ell)}(\eps,\eps')|\to 0$ as $\eps,\eps'\to0$.
\end{Lem}
\begin{proof}
Adopting the same notation as in Section \ref{Diagram},
we write $\mathbf{p}$ for multi-indices of the form $\{p_j^{(i)}: (i,j)\in [\ell]\times \{0,1,2,3\}\}$ and set $S(\mathbf{p}) = \sum_{i=1}^{\ell}\sum_{j=0}^3 p^{(i)}_j$. The Wiener-chaos decomposition of 
$A_n^{(\ell)}(Q; \eps,\eps')$ in \eqref{Aeps} is obtained from that of $L_n^{(\ell)}$ in \eqref{Peven} by replacing $\T$ with $Q$ and the coefficients $\beta_{p_0^{(1)}}\cdots\beta_{p_0^{(\ell)}}$ with 
\begin{gather*}
\delta_{p_0^{(1)},\ldots,p_0^{(\ell)}}(\eps,\eps'):=
\prod_{i=1}^{\ell}
\beta_{p_0^{(i)}}(\eps)-
\prod_{i=1}^{\ell}
\beta_{p_0^{(i)}}(\eps'),\end{gather*}
where the coefficients $\beta_j(\eps)$ are as in \eqref{betaepsfor}.
Moreover, using the notation in \eqref{notX} and writing 
\begin{equation}\label{gamma}
\gamma^{(\ell)}_3\left\{p^{(i)}_j\right\}=\gamma^{(\ell)}_3\left\{p^{(i)}_j: (i,j) \in [\ell]\times[3] \right\}
:= \alpha^{(\ell)}_3\left\{p^{(i)}_j: (i,j) \in [\ell]\times[3] \right\} \cdot \prod_{i=1}^{\ell}\prod_{j=1}^3 p_j^{(i)}! 
\end{equation}
for the Fourier-Hermite coefficients of the function $\Phi_{\ell,3}^*$, we infer that 
\begin{eqnarray*}
\left|S_{n,2}^{(\ell)}(\eps,\eps')\right|  
&\leq&
\bigg(\frac{E_n}{3}\bigg)^{\ell} 
\sum_{q\geq 0} \sum_{\mathbf{p},\mathbf{q}}
\left|\frac{\delta_{p_0^{(1)},\ldots,p_0^{(\ell)}}(\eps,\eps')}
{p_0^{(1)}! \ldots p_0^{(\ell)}!}
\frac{\gamma^{(\ell)}_3\left\{p^{(i)}_j\right\}}
{\prod_{i=1}^{\ell}\prod_{j=1}^3 p_j^{(i)}!}
\frac{\delta_{q_0^{(1)},\ldots,q_0^{(\ell)}}(\eps,\eps')}
{q_0^{(1)}! \ldots q_0^{(\ell)}!}
\frac{\gamma^{(\ell)}_3\left\{q^{(i)}_j\right\}}{\prod_{i=1}^{\ell}\prod_{j=1}^3 q_j^{(i)}!}\right| \notag\\
&&  \qquad \times  \ \ind{S(\mathbf{p}) = 2q} 
\ind{S(\mathbf{q}) = 2q}  
|W(\mathbf{p},\mathbf{q})| \\
&=:& \left(\frac{E_n}{3}\right)^{\ell} \times B_n^{(\ell)}(\eps,\eps'),
\end{eqnarray*}
where 
\begin{eqnarray}\label{Wpq}
W(\mathbf{p},\mathbf{q})=
\sum_{(Q,Q')\in \mathcal{S}^c}
\int_{Q}\int_{Q'} 
\E{ \prod_{i=1}^{\ell}  \prod_{j=0}^3 
H_{p^{(i)}_j}\left(X_j^{(i)}(x)\right)  
H_{q^{(i)}_j}\left(X_j^{(i)}(y)\right)  } \ dxdy . 
\end{eqnarray}
Applying Proposition \ref{Diagram}, using that
$\ind{\cdot}\leq1$ and the fact that $S(p^{(1)})!\cdots S(p^{(\ell)})! \leq 
\big(S(p^{(1)})+\ldots+S(p^{(\ell)})\big)! = S(\mathbf{p})! = (2q)!$, we see that $W(\mathbf{p},\mathbf{q})$ is a sum of at most $(2q)!$ terms of the type 
\begin{equation} \label{wterms}
w = \sum_{(Q,Q') \in \mathcal{S}^c} \int_Q \int_{Q'} 
\prod_{j=1}^{2q} \tilde{r}_{a_j,b_j}(x-y) \ dxdy
\end{equation}
for some $a_1,b_1, \ldots, a_{2q},b_{2q} \in \{0,1,2,3\}$.
Now, using the fact that for every $(x,y) \in Q\times Q' \subset \mathcal{S}^c$ and every $a,b \in \{0,1,2,3\}$, we have $|\tilde{r}_{a,b}(x,y)| \leq \eta$, we infer 
that $|W(\mathbf{p},\mathbf{q})|\leq (2q)!\times \eta^{2q}$.
Using  
\begin{eqnarray*}
\sum_{q\geq0} \sum_{\mathbf{p},\mathbf{q}} (2q)! \cdot
\ind{S(\mathbf{p})=2q}\ind{S(\mathbf{q})=2q} 
\leq \sum_{\mathbf{p},\mathbf{q}} \sqrt{S(\mathbf{p})!} \sqrt{S(\mathbf{q})!}\ ,
\end{eqnarray*} 
we obtain  
\begin{eqnarray}\label{final}
\left|B_{n}^{(\ell)}(\eps,\eps')\right| &\leq&
\sum_{\mathbf{p},\mathbf{q}}
\left|\frac{ \delta_{p_0^{(1)},\ldots,p_0^{(\ell)}}(\eps,\eps') }
{p_0^{(1)}! \ldots p_0^{(\ell)}!}
\frac{\gamma^{(\ell)}_3\left\{p^{(i)}_j\right\}}{\prod_{i=1}^{\ell}\prod_{j=1}^3 p_j^{(i)}!}\right| \cdot \sqrt{S(\mathbf{p})!} \sqrt{\eta}^{\frac{S(\mathbf{p})+S(\mathbf{q})}{2}}   \notag \\
&&  \times \left|\frac{\delta_{q_0^{(1)},\ldots,q_0^{(\ell)}}(\eps,\eps')}
{q_0^{(1)}! \ldots q_0^{(\ell)}!}
\frac{\gamma^{(\ell)}_3\left\{q^{(i)}_j\right\}}{\prod_{i=1}^{\ell}\prod_{j=1}^3 q_j^{(i)}!}  \right| 
\cdot \sqrt{S(\mathbf{q})!} \sqrt{\eta}^{\frac{S(\mathbf{p})+S(\mathbf{q})}{2}}  \notag \\
&\leq& 
\sum_{\mathbf{p},\mathbf{q}}
\frac{\left[\delta_{p_0^{(1)},\ldots,p_0^{(\ell)}}(\eps,\eps')\right]^2}
{p_0^{(1)}! \ldots p_0^{(\ell)}!}
\frac{\gamma^{(\ell)}_3\left\{p^{(i)}_j\right\}^2}{\prod_{i=1}^{\ell}\prod_{j=1}^3 p_j^{(i)}!} \cdot \frac{ S(\mathbf{p})! }{\prod_{i=1}^{\ell}\prod_{j=0}^3 p_j^{(i)}!} \sqrt{\eta}^{ S(\mathbf{p})+S(\mathbf{q})}\ , \notag\\
&&
\end{eqnarray}
where the last inequality follows from an application of the Cauchy-Schwarz inequality to the symmetric measure 
$(\mathbf{p},\mathbf{q}) \mapsto \sqrt{\eta}^{S(\mathbf{p})+S(\mathbf{q})}$.
We now argue that $|B_n^{(\ell)}(\eps,\eps')|\to 0$ as $\eps,\eps'\to0$. First, the estimate (see e.g. \cite{AS72}, formula 22.14.16), 
\begin{gather*}
|\beta_{j}(\eps)|\leq \gamma\left(\frac{\eps}{\sqrt{2}}\right)\frac{j!}{2^{j/2}(j/2)!}
< |\beta_j|, \quad \eps>0, \ j\geq 1, 
\end{gather*} 
implies that 
\begin{gather*}\label{dom}
\left|\delta_{p_0^{(1)},\ldots,p_0^{(\ell)}}(\eps,\eps')\right| < 2\times \left|\beta_{p_0^{(1)}}\ldots\beta_{p_0^{(\ell)}}\right|
\end{gather*}
so that we can apply dominated convergence and use the fact that $\delta_{p_0^{(1)},\ldots,p_0^{(\ell)}}(\eps,\eps') \to 0$ as $\eps,\eps' \to 0$ in view of \eqref{limbetaeps}.
We will now prove that the remaining series over $\mathbf{p},\mathbf{q}$ is finite. 
We note that (i) for every $\mathbf{p}$, the quantity 
\begin{eqnarray*}
\frac{\beta_{p_0^{(1)}}^2\cdots \beta_{p_0^{(\ell)}}^2}
{p_0^{(1)}! \ldots p_0^{(\ell)}!}
\frac{\gamma^{(\ell)}_3\left\{p^{(i)}_j\right\}^2}{\prod_{i=1}^{\ell}\prod_{j=1}^3 p_j^{(i)}!} 
\end{eqnarray*}
is bounded, and (ii) using the multinomial theorem 
\begin{eqnarray*}
\frac{ S(\mathbf{p})! }{\prod_{i=1}^{\ell}\prod_{j=0}^3 p_j^{(i)}!} 
\leq \sum_{\substack{\mathbf{m}=(m_j^{(i)}):  \\ 
S(\mathbf{m})=S(\mathbf{p}) }}
\frac{S(\mathbf{p})!}{\prod_{i=1}^{\ell}\prod_{j=0}^3 m_j^{(i)}!}  \cdot \prod_{i=1}^{\ell}\prod_{j=0}^3 1^{m^{(i)}_j}
= (4\ell)^{S(\mathbf{p})}
\ . 
\end{eqnarray*}
Plugging (i) and (ii) into \eqref{final} and using the fact that $4\ell\sqrt{\eta}<1$, gives 
\begin{gather*}
\left|B_{n}^{(\ell)}(\eps,\eps')\right| \ll\sum_{\mathbf{p,q}}(4\ell)^{S(\mathbf{p})} \sqrt{\eta}^{S(\mathbf{p})+S(\mathbf{q})} < +\infty.
\end{gather*}
This finishes the proof.
\end{proof}

\subsection{Proofs of Lemma \ref{Singl} and Lemma \ref{NonSingl}}\label{ProofsLemS}
\begin{proof}[Proof of Lemma \ref{Singl}.]
Arguing as in \eqref{Est1}, we have 
\begin{gather}\label{Est3}
\left|S_{n,1}^{(\ell)}\right|
\leq  \textrm{card}(\mathcal{S}) \cdot\V{\mathrm{proj}_{6+}(L_n^{(\ell)}(Q_0))}  
\ll E_n^3 \mathcal{R}_n(6)
\cdot\V{\mathrm{proj}_{6+}(L_n^{(\ell)}(Q_0))}  , 
\end{gather}
where $Q_0$ is the cube around the origin. Now we notice that  
\begin{equation*}
\V{ \mathrm{proj}_{6+}(L_n^{(\ell)}(Q_0))} \leq \V{L_n^{(\ell)}(Q_0)} 
\leq \E{L_n^{(\ell)}(Q_0)^2} \ .
\end{equation*}
Using Kac-Rice formulae and reasoning as in \eqref{intQ0} and \eqref{Est2}, we obtain that 
\begin{gather*}
\E{L_n^{(\ell)}(Q_0)^2} 
\leq E_n^{-2}\ind{\ell=1}+E_n^{-1} \ind{\ell=2}+\ind{\ell=3}.
\end{gather*}
Combining this with the estimate in \eqref{Est3}, yields the desired conclusion.
\end{proof}
\begin{proof}[Proof of Lemma \ref{NonSingl}]
Using the fact that $\mathrm{proj}_{6+}(L_n^{(\ell)}(Q))$ is centred and  the triangular inequality, we first write  
\begin{eqnarray*}
\left|S_{n,2}^{(\ell)}\right|
\leq 
\sum_{(Q,Q') \in \mathcal{S}^c} 
\E{ \mathrm{proj}_{6+}(L_n^{(\ell)}(Q))\cdot \mathrm{proj}_{6+}(L_n^{(\ell)}(Q')) }   .
\end{eqnarray*}
For a family of non-negative integers 
$\mathbf{p}:=\big\{p^{(i)}_j: (i,j) \in [\ell] \times \{0,1,2,3\}\big\}$, we  write $S(\mathbf{p}):=\sum_{i=1}^{\ell}\sum_{j=0}^3 p^{(i)}_j$. 
Adopting the notation introduced in \eqref{notX},
it follows from the chaotic expansion in Proposition \ref{WC23} that, 
\begin{eqnarray}\label{CovEst}
&&\left|S_{n,2}^{(\ell)}\right|  \leq \bigg(
\frac{E_n}{3}\bigg)^{\ell} 
\sum_{q\geq 3}\sum_{\mathbf{p},\mathbf{q}}
\left|\frac{\beta_{p_0^{(1)}}\ldots \beta_{p_0^{(\ell)}}}
{p_0^{(1)}! \ldots p_0^{(\ell)}!}
\frac{\gamma^{(\ell)}_3\left\{p^{(i)}_j\right\}}{\prod_{i=1}^{\ell}\prod_{j=1}^3 p_j^{(i)}!}
\frac{\beta_{q_0^{(1)}}\ldots \beta_{q_0^{(\ell)}}}
{q_0^{(1)}! \ldots q_0^{(\ell)}!}
\frac{\gamma^{(\ell)}_3\left\{q^{(i)}_j\right\}}{\prod_{i=1}^{\ell}\prod_{j=1}^3 q_j^{(i)}!}\right| \notag\\
&& \hspace{2cm} \times  \ \ind{S(\mathbf{p}) = 2q}
\ind{S(\mathbf{q}) = 2q}  
|W(\mathbf{p},\mathbf{q})| \ ,  
\end{eqnarray}
where $\gamma^{(\ell)}_3\{\cdot\}$ is as in \eqref{gamma} and
$W(\mathbf{p},\mathbf{q})$ as in \eqref{Wpq}.
Arguing as in the proof of Lemma 
\ref{Part2}, we see that $W(\mathbf{p},\mathbf{q})$ is a sum of at most $(2q)!$ terms of the type 
\begin{equation*} 
w = \sum_{(Q,Q') \in \mathcal{S}^c} \int_Q \int_{Q'} 
\prod_{j=1}^{2q} \tilde{r}_{a_j,b_j}(x-y) \ dxdy
\end{equation*}
for some $a_1,b_1, \ldots, a_{2q},b_{2q} \in \{0,1,2,3\}$.
Now, using the fact that for every $(x,y) \in Q\times Q' \subset \mathcal{S}^c$ and every $a,b \in \{0,1,2,3\}$, we have $|\tilde{r}_{a,b}(x,y)| \leq \eta$, we deduce that
\begin{eqnarray*}
|w| \leq  \eta^{2q-6}\sum_{(Q,Q')\in \mathcal{S}^c}
\int_{Q}\int_{Q'} 
\prod_{j=1}^{6} \tilde{r}_{a_j,b_j}(x-y)   dxdy  \leq 
\eta^{2q-6} 
\int_{\T}  \prod_{j=1}^{6} \tilde{r}_{a_j,b_j}(z)  dz \ .
\end{eqnarray*}
Then, by Cauchy-Schwarz inequality 
$|\tilde{r}_{a_j,b_j}(z)|\leq1$ for every $z\in \T$, we have $\tilde{r}_{a_j,b_j} \in L^6(dz)$ for every $j \in [6]$, so that applying the generalised H\"older inequality yields
\begin{eqnarray}\label{boundW}
|w| \leq  \eta^{2q-6} \prod_{j=1}^6 \bigg(
\int_{\T}  \tilde{r}_{a_j,b_j}(z)^6  dz \bigg)^{1/6}
\ll  \eta^{2q-6} \cdot  \mathcal{R}_n(6) 
= \frac{\mathcal{R}_n(6)}{\eta^6} \cdot 
\sqrt{\eta}^{\frac{S(\mathbf{p})+S(\mathbf{q})}{2}}  \sqrt{\eta}^{\frac{S(\mathbf{p})+S(\mathbf{q})}{2}}  \ ,
\end{eqnarray}
where we used Lemma \ref{rab} and the fact that $S(\mathbf{p})=S(\mathbf{q})=2q$. Then, arguing exactly as in \eqref{final}, we write 
\begin{eqnarray*}
|W(\mathbf{p},\mathbf{q})| 
\leq (2q)! \cdot  \frac{\mathcal{R}_n(6)}{\eta^6} \cdot 
\sqrt{\eta}^{\frac{S(\mathbf{p})+S(\mathbf{q})}{2}}  \sqrt{\eta}^{\frac{S(\mathbf{p})+S(\mathbf{q})}{2}} 
= \frac{\mathcal{R}_n(6)}{\eta^6} \cdot 
\sqrt{S(\mathbf{p})!}\sqrt{S(\mathbf{q})!}
\sqrt{\eta}^{\frac{S(\mathbf{p})+S(\mathbf{q})}{2}}  \sqrt{\eta}^{\frac{S(\mathbf{p})+S(\mathbf{q})}{2}},
\end{eqnarray*}
and obtain that  
\begin{gather*}
\left|S_{n,2}^{(\ell)}\right| 
\ll \left(\frac{E_n}{3}\right)^{\ell}  \frac{\mathcal{R}_n(6)}{\eta^6}
\sum_{\mathbf{p},\mathbf{q}}
\frac{\beta^2_{p_0^{(1)}}\ldots \beta^2_{p_0^{(\ell)}}}
{p_0^{(1)}! \ldots p_0^{(\ell)}!}
\frac{\gamma^{(\ell)}_3\left\{p^{(i)}_j\right\}^2}{\prod_{i=1}^{\ell}\prod_{j=1}^3 p_j^{(i)}!} \cdot \frac{ S(\mathbf{p})! }{\prod_{i=1}^{\ell}\prod_{j=0}^3 p_j^{(i)}!} \sqrt{\eta}^{ S(\mathbf{p})+S(\mathbf{q})}. 
\end{gather*}
Proceeding exactly as in the end of the proof of Lemma \ref{Part2}, shows that the series over $\mathbf{p},\mathbf{q}$ converges, which finishes the proof.
\end{proof}

\bibliographystyle{alpha}
\bibliography{Fluctuations_of_nodal_sets}

\end{document}